\documentclass[12pt, a4paper]{article}

\usepackage[a4paper, hmargin=3cm, vmargin=3cm]{geometry}
\usepackage{amsmath, amssymb, amsthm, graphicx}

\allowdisplaybreaks[4]

\newtheorem{lem}{Lemma}[section]
\newtheorem{thm}[lem]{Theorem}

\numberwithin{equation}{section}
\numberwithin{figure}{section}

\renewcommand{\leq}{\leqslant}
\renewcommand{\geq}{\geqslant}

\newcommand{\bA}{\mathbf{A}}
\newcommand{\bB}{\mathbf{B}}
\newcommand{\bC}{\mathbf{C}}
\newcommand{\bD}{\mathbf{D}}
\newcommand{\bb}{\mathbf{b}}
\newcommand{\bc}{\mathbf{c}}
\newcommand{\be}{\mathbf{e}}
\newcommand{\Bf}{\mathbf{f}}
\newcommand{\bg}{\mathbf{g}}
\newcommand{\bl}{\mathbf{l}}
\newcommand{\bM}{\mathbf{M}}
\newcommand{\bs}{\mathbf{s}} 
\newcommand{\bt}{\mathbf{t}}
\newcommand{\bv}{\mathbf{v}}
\newcommand{\bw}{\mathbf{w}}
\newcommand{\bx}{\mathbf{x}}
\newcommand{\G}{\mathcal{G}}
\newcommand{\transpose}{\textup{T}}

\def\dist{{\rm dist}}
\def\supp{{\rm supp}}

\newcommand{\dimB}{\dim_{\rm B}}
\newcommand{\dimH}{\dim_{\rm H}}

\title{The $q$th packing moments and the packing $L^q$-spectra of directed graph self-similar measures}     
\author{Graeme Boore }

\begin{document}

\maketitle

\begin{abstract}
Any self-similar directed graph iterated function system with probabilities, defined on $\mathbb{R}^m$, determines a unique list of self-similar Borel probability measures whose supports are the components of the attractor. Using an application of the Renewal Theorem we obtain an explicit calculable value for the power law behaviour of the $q$th packing moments  of the self-similar measures at scale $r$ as $r\to 0^+$ in the non-lattice case, with a corresponding limit for the lattice case. We do this  
\begin{description}
\item  (i) for any $q\in \mathbb{R}$ if the strong separation condition (SSC) holds, 
\item  (ii) for $q\geq 0$ if the weaker open set condition (OSC) holds, where we also assume that a non-negative  matrix associated with the system is irreducible.  
\end{description}
In the non-lattice case this enables the packing $L^q$-spectra and their exact rate of convergence to be determined. 
\end{abstract}

\section{Introduction}
The well-behaved properties of self-similar sets under the OSC are used repeatedly throughout this paper. We use \emph{standard IFS} to mean a self-similar $1$-vertex directed graph iterated function system and \emph{$n$-vertex IFS} as a shortening of self-similar $n$-vertex directed graph IFS. Under the terms of the Contraction Mapping Theorem any $n$-vertex IFS determines a unique $n$-component attractor. If in addition probability weights are assigned to the edges of the directed graph, as described in Subsection \ref{nvertex}, then we call the system an \emph{$n$-vertex IFS with probabilities}. A second application of the Contraction Mapping Theorem then ensures the existence of a unique list of $n$ self-similar Borel probability measures, one at each vertex, whose supports are the components of the attractor. Despite the presence of probabilities these are still deterministic IFSs as changing the probabilities doesn't change the attractor and so doesn't change the supports of the measures. Changing the probabilities only changes the values that the measures take. 

In \cite{Paper_Lalley} Lalley  obtained results for the power law behaviour of the packing and covering functions of a standard (OSC) IFS attractor. More general results were obtained by Olsen in \cite{Paper_Olsen_1} for the power law behaviour of the $q$th packing and covering moments of the self-similar measure of a standard (OSC) IFS with probabilities, along with the rate of convergence of the corresponding $L^q$-spectra. 

In fact the strong open set condition (SOSC) is assumed in \cite{Paper_Lalley} as it was only later that Schief  \cite{Paper_Schief} proved that the OSC is equivalent to the SOSC for standard IFSs. This equivalence was extended to $n$-vertex IFSs by Wang \cite{Paper_Wang}.

In Theorem \ref{theorem 1} we extend the packing part of \cite[Theorem 1.1]{Paper_Olsen_1} for standard IFSs with probabilities so that it applies to $n$-vertex IFSs with probabilities. This is useful work because it is reasonable to expect most $n$-vertex $(n \geq 2)$ IFSs defined on $\mathbb{R}^m$ to have attractors with components that are not the attractors of any standard IFS defined on $\mathbb{R}^m$. In fact this has been shown to be the case (for $n$-vertex IFSs satisfying the convex strong separation condition) for $m=1$ and $n=2$ by Boore and Falconer in \cite{Paper_GCB_KJF} and extended further to $m=1$ and $n \geq 2$ by Boore in \cite{Paper_GCB}. 

As was the case in \cite{Paper_Lalley,Paper_Olsen_1}, we prove Theorem \ref{theorem 1} by applying the Renewal Theorem. We use the Renewal Theorem for a system of Renewal Equations as stated by Crump in \cite[Theorem 3.1(ii)]{Paper_Crump}. This has the advantage that (in theory) the limits obtained for the power law behaviour of the $q$th packing moments
\begin{equation*}
\lim_{n \to  +\infty} \frac { M_u^q\left(F_u, e^{-(t + n\lambda)} \right) }  { e^{-\left(t + n\lambda\right)\left(-\beta\left(q\right)\right)} } , \quad \lim_{r \to 0^+} \frac { M_u^q\left(F_u,r\right) }  { r^{-\beta\left(q\right)} },
\end{equation*}
can be calculated, see the text preceding  \cite[Theorem 3.1(ii)]{Paper_Crump} for the details. In the non-lattice case the packing $L^q$-spectra and their rate of convergence follow directly from these limits as we show in Subsection \ref{A system of renewal equations}. Other results in the literature concerning $L^q$-spectra can be found in  \cite{Paper_Olsen_1} and the references there. 

We prove Theorems \ref{theorem 1} and \ref{theorem 2} in Sections \ref{three}, \ref{four} and \ref{five}, they were first proved in \cite[Theorems 4.3.1 and 4.7.19]{phdthesis_Boore}. Theorem \ref{theorem 2} has a key role to play in the proof of Theorem \ref{theorem 1}. Definitions of the notation and terminology used in their statements are given in Section \ref{two}. (To keep the notation a little bit simpler we use $\mathbf{P}$, $f_u\left(t\right)$, $C_u$ in Theorem \ref{theorem 1} and $C_{e,f}$ in Theorem \ref{theorem 2}, whereas strictly speaking we should use $\mathbf{P}^q$, $f_u^q\left(t\right)$, $C_u\left(q\right)$ and $C_{e,f}\left(q\right)$ since all these objects do actually depend on the particular value of $q$ chosen).  

\begin{thm}  
\label{theorem 1}
Let $\bigl(V,E^*,i,t,r,p,((\mathbb{R}^m,\left| \ \  \right|))_{u \in V},(S_e)_{e \in E^1}\bigr)$ be an $n$-vertex IFS with probabilities with attractor $\left(F_u\right)_{u \in V}$ where $\left(S_e\right)_{e \in E^1}$ are contracting similarities. Suppose that either
\begin{description}
\item  \quad  \ \textup{(i)} $q\in \mathbb{R}$ and the SSC holds, 
\item  or  \textup{(ii)} $q\geq 0$ the OSC holds and the non-negative $n \times n$ matrix $\bB\left(q, \gamma, l\right)$, as defined in Subsection \ref{matrix B}, is irreducible with $\rho\left(\bB\left(q, \gamma, l\right)\right) = 1$. 
\end{description}

Let $\mathbf{P}$ be the $n \times n$ matrix of measures as defined in Subsection \ref{lattice matrix} with $\mathbf{A}\left(q, \beta\right)$ the corresponding non-negative matrix as defined in Subsection \ref{matrix A}, where $\beta=\beta\left(q\right)\in \mathbb{R}$ is the unique number such that $\rho\left(\mathbf{A}\left(q, \beta \right)\right) = 1$ and let $u \in V$. 

\begin{description}
\item  \textup{(a)}    If  $\mathbf{P}$ is a lattice matrix with span $\lambda>0$ there exists a positive periodic function $f_u: \mathbb{R} \to \mathbb{R}^+$, with period $\lambda$, such that
\begin{equation*}
\lim_{n \to  +\infty} \frac { M_u^q\left( F_u, e^{-(t + n\lambda)} \right) }  { e^{-\left(t + n\lambda\right)\left(-\beta\left(q\right)\right)} }   =   f_u\left(t\right)
\end{equation*}
for each $t \geq \max \left\{  \ln\left(r_e^{-1}\right)  :  e \in E^1 \right\}$.

It follows that
\begin{equation*}
\lim_{n \to  +\infty} \frac { \ln \left(M_u^q\left( F_u, e^{-(t + n\lambda)} \right) \right)}  { \left(t + n\lambda\right) }  =   \beta\left(q\right)
\end{equation*}
for each $t \geq \max \left\{  \ln\left(r_e^{-1}\right) : e \in E^1 \right\}$ and the rate of convergence as $n \to +\infty$ is 
\begin{equation*} 
\frac { \ln\left( M_u^q\left( F_u, e^{-\left(t + n\lambda\right)} \right) \right) }  { \left(t + n\lambda\right) }  =   \beta\left(q\right)  +    O \left(\frac{1}{n}\right). 
\end{equation*}	
	
\item  \textup{(b)}  If  $\mathbf{P}$ is not a lattice matrix there exists a constant $C_u > 0$ such that
\begin{equation*} 
\lim_{r \to 0^+} \frac { M_u^q\left(F_u,r\right) }  { r^{-\beta\left(q\right)} }   =   C_u. 
\end{equation*}
		 
It follows that the packing $L^q$-spectrum on $F_u$ of $\mu_u$ is given by 
\begin{equation*}
\lim_{r\to 0^+}\frac{\ln \left(M_u^q\left(F_u,r\right)\right)} {-\ln r}  =   \beta\left(q\right)  
\end{equation*}
and the rate of convergence as $r \to 0^+$ is 
\begin{equation*}
\frac{\ln \left(M_u^q\left(F_u,r\right)\right)} {-\ln r}  =  \beta\left(q\right) +   O\left(\frac{1}{-\ln r}\right).  
\end{equation*}
\end{description}
\end{thm}   

Theorem \ref{theorem 2} is also of interest in its own right. For example it can be used to provide information relating the upper (multifractal) $q$ box-dimension and the (multifractal) $q$ Hausdorff dimension, as shown in \cite[Theorem 4.3.2]{phdthesis_Boore}, which extends some of the work in \cite{Paper_Olsen_2} from standard IFSs with probabilities to more general $n$-vertex IFSs with probabilities. (The (multifractal) $q$ box-dimension is another name for the packing $L^q$-spectrum used to emphasise the fact that the packing $L^q$-spectrum is a measure-theoretic generalisation of the box-counting dimension).

\begin{thm} 
\label{theorem 2}
Let $\bigl(V,E^*,i,t,r,p,((\mathbb{R}^m,\left| \ \  \right|))_{u \in V},(S_e)_{e \in E^1}\bigr)$ be an $n$-vertex IFS with probabilities with attractor $\left(F_u\right)_{u \in V}$ where $\left(S_e\right)_{e \in E^1}$ are contracting similarities. Suppose that the OSC holds and the non-negative $n \times n$ matrix $\bB\left(q, \gamma, l\right)$, as defined in Subsection \ref{matrix B}, is irreducible. Let $u \in V$, $q\in \mathbb{R}$ and let $\gamma=\gamma\left(q\right)\in \mathbb{R}$ be the unique number such that $\rho\left(\bB\left(q,\gamma,l\right)\right)=1$. Let $r\in \left(0,\delta\right)$ where $\delta$ is as defined in Equation (\ref{delta}) and let $e,f \in E_u^1,\, e \neq f$.

Then
\begin{equation*}
Q_{e,f}^q\left(r\right)=M_u^q\left(S_e(F_{t(e)})\cap S_f(F_{t(f)})(r),r\right)\leq C_{e,f}\ r^{-\gamma\left(q\right)} 
\end{equation*}
for some constant $C_{e,f} > 0$.
\end{thm}

\section{Notation and background theory}\label{two}
We often use a notation of the form $\left(A_c\right)_{c\in B}$ and $\left(A\right)_{c\in B}$ when $B$ is a finite set of $n$ elements as this is just a convenient way of writing down ordered $n$-tuples. That is, if $B$ is ordered as $B=\left(b_1,b_2,\ldots,b_n\right)$, then $\left(A_c\right)_{c\in B}=\left(A_{b_1},A_{b_2},\ldots,A_{b_n}\right)$ and $\left(A\right)_{c\in B}=\left(A,A,\ldots,A\right)$.

We use  $\left| \ \ \right|$ to indicate the length of a sequence, the Euclidean metric and the diameter of a set in $\mathbb{R}^m$, where the intended meaning should be clear from the context. 

Let $x \in \mathbb{R}^m$ then the \emph{closed ball of centre $x$ and radius $r>0$} is defined as
\begin{equation*}
B\left(x,r\right)=\left\{y : y \in \mathbb{R}^m, \left|x - y\right| \leq r\right\}
\end{equation*}
with $S(x,r)$ the corresponding open ball.

The \emph{diameter of a non-empty set} $A\subset \mathbb{R}^m$ is defined as
\begin{equation*}
\left|A\right|=\sup\left\{\left|x - y\right|: x,y\in A\right\}
\end{equation*}
and we put $\left|\emptyset\right|=0$.

The \emph{distance between a point $x \in \mathbb{R}^m$ and a non-empty set} $A\subset \mathbb{R}^m$ is defined as
\begin {equation*}
\dist\left(x,A\right)=\inf\left\{\left|x-a\right| : a \in A\right\}.
\end{equation*}

The \emph{distance between non-empty sets} $A,B\subset \mathbb{R}^m$ is defined as
\begin {equation*}
\dist\left(A,B\right)=\inf\left\{\left|a - b\right|: a \in A, \, b \in B\right\}.
\end{equation*}

For $r > 0$ the \emph{closed $r$-neighbourhood of a non-empty set $A \subset \mathbb{R}^m$} is defined as
\begin {equation*}
A\left(r\right)=\left\{ x : x \in \mathbb{R}^m, \, \dist(x,A)\leq r \right\}.
\end {equation*}

The rest of this section provides the notation and definitions for the terminology used in the statements of Theorems \ref{theorem 1} and \ref{theorem 2}. This is the machinery required in order to apply the Renewal Theorem for a system of renewal equations as stated in \cite[Theorem 3.1(ii)]{Paper_Crump} and Subsection \ref{2renewal_theorem}, which is used in the proof of Theorem \ref{theorem 1} in Section \ref{three}.

\subsection{$n$-vertex IFSs with probabilities}\label{nvertex} 

We use $\bigl(V,E^*,i,t,r,p,((X_{v},d_{v}))_{v \in V},(S_e)_{e \in E^1}\bigr)$ to indicate an \emph{$n$-vertex IFS with probabilities} where $\bigl(V,E^*,i,t\bigr)$ is the associated \emph{directed graph}, $V$ is the set of all vertices, $E^*$ is the set of all finite (directed) paths, $i:E^* \to V$ and $t:E^* \to V$ are the initial and terminal vertex functions. The set of all (directed) edges in the graph, that is the set of paths of length $1$, is written as $E^1$ with $E^1\subset E^*$. $V$ and $E^1$ are always assumed to be finite sets. We use $E^1_u$ to indicate the set of all edges with initial vertex $u$, $E^k_u$ for the set of all paths of length $k$ with initial vertex $u$, $E^k_{uv}$ for the set of all paths of length $k$ starting at the vertex $u$ and finishing at $v$, $E_u^{\mathbb{N}}$ for the set of all infinite paths starting at the vertex $u$ and so on. 

A finite \emph{(directed) path} $\be \in E^*$ is a finite string of consecutive edges so a path of length $k$ can be written as $\be=e_1 \cdots e_k$ for some edges $e_i \in E^1$ with $t(e_i)=i(e_{i+1})$ for $1\leq i < k$. The initial vertex of a path is the initial vertex of its first edge so $i(\be)=i(e_1)$ and similarly $t(\be)=t(e_k)$. The \emph{vertex list} of a path $\be=e_1\cdots e_k \in E^*$ is  $v_1v_2v_3\cdots v_{k+1} = i(e_1) t(e_1) t(e_2) \cdots $ $t(e_k)$ and shows the order in which a path visits its vertices. 

For $\Bf, \bg \in E^*$,  $\Bf$ is a \emph{subpath} of $\bg$ if and only if $\bg = \bs \Bf \bt$ for some $\bs, \bt \in E^*$, where we assume the empty path is an element of $E^*$. We use the notation $\Bf \subset \bg$ to indicate that $\Bf$ is a subpath of $\bg$.
   
For $\Bf, \bg \in E^*$,  $\Bf$ is  \emph{not a subpath} of $\bg$ if and only if $\bg \neq \bs \Bf \bt$ for all $\bs, \bt \in E^*$ and we use the notation $\Bf \not\subset \bg$ to indicate that $\Bf$ is not a subpath of $\bg$.

We assume the directed graph is strongly connected and that each vertex in the directed graph has at least two edges leaving it, this is to avoid components of the attractor (defined below) that consist of single point sets or are just scalar copies of those at other vertices (see \cite{Paper_Edgar_Mauldin}). 

The \emph{contraction ratio function} $r:E^*\to (0,1)$ assigns contraction ratios to the finite paths in the graph. To each vertex $v \in V$ is associated the non-empty complete metric space $(X_{v},d_{v})$ and to each directed edge $e\in E^1$ is assigned a contraction $S_{e}:X_{t(e)} \to X_{i(e)}$ which has the contraction ratio given by the function $r(e) = r_{e}$. We follow the convention already established in the literature, see \cite{Book_Edgar2, Paper_Edgar_Mauldin}, that $S_e$ maps in the opposite direction to the direction of the edge $e$ that it is associated with in the graph.  The contraction ratio along a path $\be=e_1e_2\cdots e_k \in E^*$ is defined as $r(\be)=r_{\be}=r_{e_1}r_{e_2}\cdots r_{e_k}$. The ratio $r_{\be}$ is the ratio for the contraction $S_{\be}:X_{t(\be)} \to X_{i(\be)}$ along the path $\be$ where $S_{\be}=S_{e_1}\circ S_{e_2}\circ \cdots \circ S_{e_k}$.

The \emph{probability function} $p:E^*\to (0,1)$, where for an edge $e \in E^1$ we write $p(e)=p_{e}$, is such that 
\begin{equation}
\label{probability function}
\sum_{ e\in E_{u}^1 }p_{e} = \sum_{v\in V} \biggl(  \  \sum_{e\in E_{uv}^1}p_{e} \  \biggr) = 1
\end{equation}
for any vertex $u \in V$. That is the probability weights across all the edges leaving a vertex always sum to one. For a path  $\be=e_1e_2\cdots e_k \in E^*$ we define $p(\be)=p_{\be}=p_{e_1}p_{e_2}\cdots p_{e_k}$. 

In this paper we are only going to be concerned with $n$-vertex IFSs defined on $m$-dimensional Euclidean space where $((X_{v},d_{v}))_{u \in V} = ((\mathbb{R}^m,\left| \ \  \right|))_{u \in V}$ and $(S_e)_{e \in E^1}$ are contracting similarities and not just contractions. However it's worth pointing out that the Invariance Equations \eqref{Invariance 1} and \eqref{Invariance 2} are the result of applying the Contraction Mapping Theorem and so they also hold when the similarities $(S_e)_{e \in E^1}$ are more general contractions. We use $K\left(\mathbb{R}^m\right)$ to denote the set of all non-empty compact subsets of $\mathbb{R}^m$. Using the Contraction Mapping Theorem it can be shown that an $n$-vertex IFS $\bigl(V, E^*, i, t, r, p, ((\mathbb{R}^m,\left| \ \  \right|))_{u \in V}, (S_e)_{e \in E^1} \bigr)$ with probabilities determines a unique list of non-empty compact sets $\left(F_u\right)_{u \in V} \in \left(K\left(\mathbb{R}^m\right)\right)^n$ which satisfies the invariance equation 
\begin{equation}
\label{Invariance 1}
 \left(F_u\right)_{u \in V} = \biggl( \ \bigcup_{ \substack{e\in E_u^1} }S_e\left(F_{t(e)}\right) \ \biggr)_{u \in V}
\end{equation} 
see  \cite[Theorem 1.3.4]{phdthesis_Boore}, \cite[Theorem 4.3.5]{Book_Edgar2} or \cite[Theorem 1]{Paper_Mauldin_Williams}. Under the terms of the Contraction Mapping Theorem $(F_u)_{u \in V}$ is known as the \emph{attractor} of the system and  we call the $n$ non-empty compact sets $F_u$,  $u \in V$, the \emph{components of the attractor}. 

Another application of the Contraction Mapping Theorem determines a unique list of Borel probability measures, $(\mu_u)_{u \in V}$, such that  
\begin{equation}
\label{Invariance 2}
\left(\mu_u\left(A_u\right)\right)_{u \in V}=\biggl( \ \sum_{ \substack{ e\in E_u^1}  }p_{e} \mu_{t(e)}\left(S_{e}^{-1}(A_u)\right)  \ \biggr)_{u \in V}
\end{equation}
for all Borel sets $\left(A_u\right)_{u \in V}\subset \left(\mathbb{R}^m\right)^{n}$, with $\left(\supp \mu_u\right)_{u \in V} = \left(F_u\right)_{u \in V}$, see \cite[Proposition 3]{Paper_Wang}. For the standard $1$-vertex case see \cite[Theorem 2.8]{Book_KJF1}. When $(S_e)_{e \in E^1}$ are contracting similarities $(\mu_u)_{u \in V}$ are known as \emph{directed graph self-similar measures}. 

For a set $A\subset \mathbb{R}^m$, we use the usual notation $\mathcal{H}^s(A)$ for the $s$-dimensional Hausdorff measure, $\dimH A$  and $\dimB A$ for the Hausdorff and box-counting dimension. The next theorem is the main dimension result for $n$-vertex IFSs defined on $\mathbb{R}^m$, see \cite[Theorem 3]{Paper_Mauldin_Williams} and also \cite[Theorem 6.9.6]{Book_Edgar2}. For standard IFSs with $n=1$ this is the same as \cite[Theorem 9.3]{Book_KJF2}.

\begin{thm} 
\label{dimension}
Let $\bigl(V,E^*,i,t,r,((\mathbb{R}^m,\left| \ \  \right|))_{u \in V},(S_e)_{e \in E^1}\bigr)$ be an $n$-vertex IFS with attractor $(F_u)_{u \in V}$ where the mappings $(S_e)_{e \in E^1}$ are contracting similarities. Let $\bA(t)$ denote the $n\times n$ matrix 
whose $uv$th entry is
\begin{equation*}
A_{uv}(t) = \sum_{e\in E_{uv}^1} r_e^t,
\end{equation*}
let $\rho\left(\bA(t)\right)$ be the spectral radius of $\bA(t)$, and let $s$ be the unique non-negative real number that is the solution of $\rho\left(\bA(t)\right)=1$.  

If the OSC is satisfied then, for each $u \in V$, $s = \dimH F_u = \dimB F_u$ and $0 < \mathcal{H}^s \left(F_u\right) < +\infty$.
\end{thm} 

\subsection{Separation conditions} \label{separation conditions}

The \emph{open set condition} (OSC) is satisfied if and only if there exist non-empty bounded open sets $\left(U_u\right)_{u\in V}\subset \left(\mathbb{R}^m\right)^n$ such that for each $u\in V$
\begin{gather*}
S_e\left(U_{t(e)}\right) \subset U_u   \textrm{ for all } e \in E_u^1, \\
\textrm{ and }S_e\left(U_{t(e)}\right)\cap S_f\left(U_{t(f)}\right) = \emptyset \textrm{ for all } e,f\in E_u^1,\textrm{ with } e \neq f.
\end{gather*}
See \cite{Book_Edgar2,Book_KJF2,Paper_Hutchinson}.

The \emph{strong open set condition} (SOSC) is satisfied if and only if the OSC is satisfied for non-empty bounded open sets $(U_u)_{u\in V}\subset \left(\mathbb{R}^{m}\right)^{n}$, where for each $u \in V$,
\begin{equation*}
F_u\cap U_u \neq \emptyset.
\end{equation*} 
When $(S_e)_{e \in E^1}$ are contracting similarities the SOSC is equivalent to the OSC, see \cite{Paper_Wang} for a proof for $n$-vertex IFSs and \cite{Paper_Schief} for standard IFSs. 

The \emph{strong separation condition} (SSC) is satisfied if and only if for each $u\in V$,
\begin{equation*}
S_e\left(F_{t(e)}\right)\cap S_f\left(F_{t(f)}\right) = \emptyset \textrm{ for all } e,f\in E_u^1,\textrm{ with } e \neq f.
\end{equation*}

We write $C(F_u)$ for the convex hull of $F_u$. 

The \emph{convex strong separation condition} (CSSC) is satisfied if and only if for each $u\in V$,
\begin{equation*}
S_e(C(F_{t(e)}))\cap S_f(C(F_{t(f)})) = \emptyset \textrm{ for all } e,f\in E_u^1,\textrm{ with } e \neq f.
\end{equation*}

\subsection{The $q$th packing moment, $Q_{e,f}^q\left(r\right)$, and the packing $L^q$-spectrum} \label{The $q$th packing moment}

Let $A\subset \mathbb{R}^m$ be a bounded subset of $\mathbb{R}^m$. For $r>0$, a subset $D\subset A$ is \emph{an $r$-separated subset of $A$} if $\left|x-y\right| > 2r$ for all $x,y \in D$ with $x \neq y$. This means that
\begin{equation*}
B\left(x,r\right)\cap B\left(y,r\right) = \emptyset
\end{equation*}
for all $x,y \in D$ with $x \neq y$. 
 
For $u \in V,\, q \in \mathbb{R},\, r > 0$, and $A \subset F_u$, we define the \emph{$q$th packing moment of $\mu_u$ on $A$ at scale $r$} as 
\begin{equation}
\label{$q$th packing moment}
M_u^q\left(A,r\right) = \sup \left\{ \sum_{x \in D} \mu_u\left(B(x,r)\right)^q :  D\ \textrm{is an}\ r\textrm{-separated subset of}\ A \right\}.
\end{equation}

Theorem \ref{theorem 2} gives information about a special type of $q$th packing moment that turns out to be important enough to give it its own notation $Q_{e,f}^q\left(r\right)$. (Strictly speaking we should write $Q_{u,e,f}^q\left(r\right)$ and $C_{u,e,f}$ in Theorem \ref{theorem 2} but we only use $Q_{e,f}^q\left(r\right)$ when the vertex $u$ is fixed and this should always be clear from the context). For edges $e,f \in E_u^1$, $f \neq e$, putting $A = S_e\left(F_{t(e)}\right)\cap S_f\left(F_{t(f)}\right)(r)$ in  Equation \eqref{$q$th packing moment} we obtain
\begin{equation}
\label{Q_{e,f}}
Q_{e,f}^q\left(r\right) = M_u^q\left(S_e(F_{t(e)})\cap S_f(F_{t(f)})(r), r\right)
\end{equation}
where $S_f\left(F_{t(f)}\right)(r)$ is the closed $r$-neighbourhood of  $S_f\left(F_{t(f)}\right)$.

The \emph{packing $L^q$-spectrum on $A\subset F_u$ of $\mu_u$} is defined by the following limit, if it exists,
\begin{equation}
\label{$L^q$-spectrum}
\lim_{\ r\to 0^+}\frac{\ln (M_u^q(A,r))} {-\ln r}.
\end{equation}
The packing $L^q$-spectrum is also called the \emph{multifractal $q$ box-dimension} and it can be regarded as a measure-theoretic generalisation of the box-counting dimension. In fact for the specific case with $q=0$ in Equation \eqref{$L^q$-spectrum}, that is what it is.

\subsection{$M_u^q\left(F_u,r\right)$ is Riemann integrable} \label{Riemann integrable}

The next lemma is used in the proof that $M_u^q\left(F_u,r\right)$ is Riemann integrable in Lemma \ref{MFvrcty}.

\begin{lem}
\label{mcty}
Let $\bigl(V,E^*,i,t,r,p,((\mathbb{R}^m,\left| \ \  \right|))_{u \in V},(S_e)_{e \in E^1}\bigr)$ be an $n$-vertex (OSC) IFS with probabilities
where $\left(F_u\right)_{u \in V}$ and $\left(\mu_u\right)_{u \in V}$ are as given in Equations \eqref{Invariance 1} and \eqref{Invariance 2}. Let $x \in \mathbb{R}^m$ and $r > 0$ with $C(x,r) = B(x,r)\setminus S(x,r)$ where $B(x,r)$ and $S(x,r)$ are the closed and open ball in $\mathbb{R}^m$. Then for each $u \in V$

\textup{(a)} $F_u$ is uncountable with no isolated points,

\textup{(b)} $\mu_u\left(C(x,r)\right) = 0$,

\textup{(c)} $\mu_u\left(B(x,r)\right)$ is a continuous function of $x$ and $r$. 

\end{lem}

\begin{proof}(a) If $F_u$ is countable then $s= \dimH F_u = 0$ and as $0 < \mathcal{H}^0 \left(F_u\right) < +\infty$ by Theorem \ref{dimension}, $F_u$ must be finite, which it isn't as can be seen by repeated iteration of Equation \eqref{Invariance 1}. That every point of $F_u$ is a limit point of $F_u$ follows by Equation \eqref{phi_u}.

(b) It is convenient to use the following notation where $\be \in E_u^*$ and $\left(U_u\right)_{u\in V}$ are the open sets of the OSC as defined in Subsection \ref{separation conditions}. Let 
\begin{gather*}
\mathcal{U}_\be = S_{\be}(F_{t(\be)}) \cap S_{\be}(U_{t(\be)}) \textrm{ and } \mathcal{U}_\be^\prime = S_{\be}(F_{t(\be)}) \setminus S_{\be}(U_{t(\be)})
\end{gather*} 
then iterating Equation \eqref{Invariance 1} $k$ times we obtain
\begin{equation}
\label{Ue}
F_u =  \bigcup_{ \substack{\be\in E_u^k} }S_{\be}(F_{t(\be)}) = \bigcup_{ \substack{\be\in E_u^k} }\left(\mathcal{U}_\be \cup \mathcal{U}_\be^\prime \right).
\end{equation}

In Statements (i)-(iv), $\be \in E_u^*$ is any finite path. 

(i) \emph{$\mu_u\left(F_u \setminus U_u\right)=0$ and $\mu_u\left(\mathcal{U}_\be^\prime \right)=0$.} 

We show in the proof of Lemma \ref{OSCe} that $\mu_u\left(F_u \setminus U_u\right)=0$ for any $u \in V$ and this implies $\mu_u\left(\mathcal{U}_\be^\prime \right)=0$.

(ii) \emph{$\mu_u\left(F_u \cap U_u\right)=1$ and $\mu_u\left(\mathcal{U}_\be\right)=p_\be$.} 

This is implied by (i) and Lemma \ref{OSCe}.

(iii) \emph{$\left(F_u \cap U_u\right) \not\subset \left(C(x,r) \cap F_u\right)$ and $\mathcal{U}_\be \not\subset \left(C(x,r) \cap F_u\right)$.} 

This holds because  $\left(F_u \cap U_u\right)$ and $\mathcal{U}_\be$ are open in $F_u$ but $C(x,r) \cap F_u$ is closed with empty interior in $F_u$. Here we assume that $m$ is a minimum for the parent space $\mathbb{R}^m$, so for example we assume the Cantor set is constructed in $\mathbb{R}$ and not $\mathbb{R}^2$. (For $m\geq2$ this follows geometrically because contracting similarities increase curvature but the curvature is constant on $C(x,r)$).

(iv) \emph{We may choose a fixed $j$ so that for any $C(x,r)$ there is always at least one path $\Bf \in E_u^j$ such that $C(x,r) \cap \mathcal{U}_\Bf = \emptyset$. In general there is a fixed $j$ so that for any $C(x,r)$ and any $\mathcal{U}_\be$ there is always at least one path $\Bf \in E_{t(\be)}^j$ such that $C(x,r) \cap \mathcal{U}_{\be\Bf} = \emptyset$.} 

The first sentence follows by (iii) and and the fact that the OSC ensures that the union $\bigcup_{ \substack{\be \in E_u^k} }\mathcal{U}_{\be}$ is disjoint  and decreases component-wise as $k$ increases. Basically $j$ can be chosen large enough to ensure that no $C(x,r)$ can intersect all of the sets $\left\{\mathcal{U}_\be : \be \in E_u^j\right\}$. The general statement follows from this using the self-similarity of the attractor. 

Statements (i) and (ii) mean we can regard the measure $\mu_u$ as a mass distribution as described in \cite[Section 1.3] {Book_KJF2} with the mass $\mu_u\left(F_u \cap U_u\right)=1$ on $F_u \cap U_u$ repeatedly subdivided component-wise across the sets $\left\{\mathcal{U}_\be : \be \in E_u^k\right\}$ as $k$ increases. Remembering $\supp \mu_u = F_u$ and using Equation \eqref{Ue} along with (i)-(iv), with $j$ as chosen in (iv), we obtain
\begin{gather*}
\mu_u\left(C(x,r)\right) \leq \sum_{ \substack{\be\in E_u^j \\ C(x,r) \cap \mathcal{U}_\be \neq \emptyset} }\mu_u\left(\mathcal{U}_\be \right) = \sum_{ \substack{\be\in E_u^j \\ C(x,r) \cap \mathcal{U}_\be \neq \emptyset} }p_\be \\
\leq \biggl( \  \sum_{ \substack{\be\in E_u^j } }p_\be \ \biggr) - p_\Bf \leq 1 - p_{\min}^j
\end{gather*}
where $p_{\min}= \min\left\{p_e  :  e \in E^1 \right\} $ and the path $\Bf \in E_u^j$ is as given in (iv). Repeating this argument
\begin{gather*}
\mu_u\left(C(x,r)\right) \leq \sum_{ \substack{\be\in E_u^{2j} \\ C(x,r) \cap  \mathcal{U}_\be \neq \emptyset} }\mu_u\left(\mathcal{U}_\be \right) = \sum_{ \substack{\be\in E_u^{2j} \\ C(x,r) \cap \mathcal{U}_\be \neq \emptyset} }p_\be \\
= \sum_{ \substack{\bs\in E_u^{j} \\ C(x,r) \cap \mathcal{U}_\bs \neq \emptyset} }p_\bs \Biggl( \ \sum_{ \substack{\bt\in E_{t(\bs)}^{j} \\ C(x,r) \cap \mathcal{U}_{\bs \bt} \neq \emptyset} }p_\bt \ \Biggr) \leq \sum_{ \substack{\bs\in E_u^{j} \\ C(x,r) \cap \mathcal{U}_\bs \neq \emptyset} }p_\bs \Biggl( \ \biggl( \ \sum_{ \substack{\bt\in E_{t(\bs)}^{j} } }p_\bt \ \biggr) - p_{\min}^j \ \Biggr) \\ 
= \sum_{ \substack{\bs\in E_u^{j} \\ C(x,r) \cap \mathcal{U}_\bs \neq \emptyset} }p_\bs \left( 1 - p_{\min}^j \right) \leq \left( 1 - p_{\min}^j \right)^2
\end{gather*}
It follows by induction that $\mu_u\left(C(x,r)\right) \leq \left( 1 - p_{\min}^j \right)^k$ for all $k$ which implies $\mu_u\left(C(x,r)\right) =0$.

\begin{figure}[!htb]
\begin{center}
\includegraphics[trim = 28mm 185mm 28mm 20mm, clip, scale=0.7]{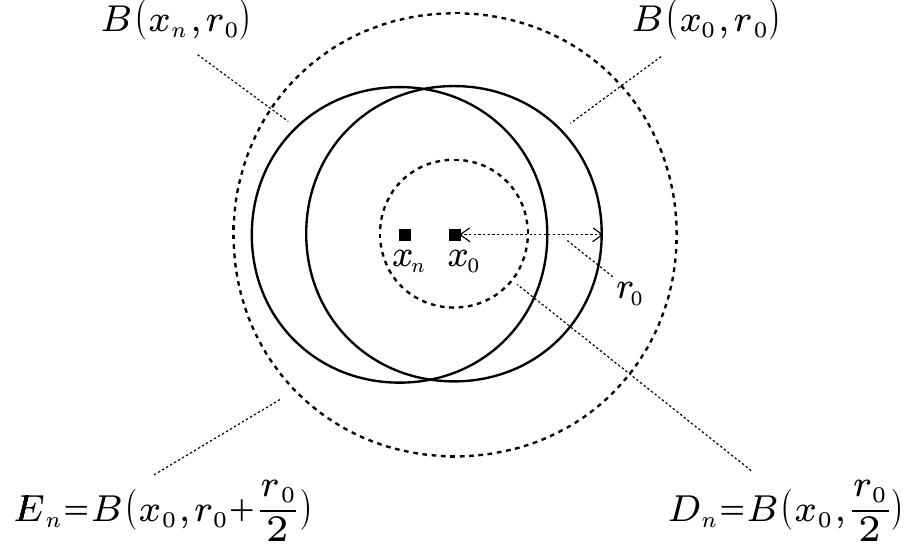}
\end{center}
\caption{A schematic representation in $\mathbb{R}^2$ of the closed balls $D_n$ and $E_n$ for $N_1\leq n < N_2$.}
\label{mcontinuity}
\end{figure}

(c)  We prove this for $x \in \mathbb{R}^m$ keeping $r= r_0$ fixed, the proof in the general situation can be constructed in a similar way. We aim to show $\mu_u\left(B(x,r_0)\right)$ is a continuous function of $x$. Let $x_0$ be chosen then by part (b) 
 \begin{equation}
\label{mSB}
\mu_u\left(S(x_0,r_0)\right) = \mu_u\left(B(x_0,r_0)\right).
\end{equation} 
 Let $\left(x_n\right)$ be a sequence that converges to $x_0$ and let $\left(N_k\right)$ be a strictly increasing sequence such that $\left|x_n - x_0\right| < r_0/2^k$ for all $n \geq N_k$. For $N_k \leq n < N_{k+1}$ let 
\begin{gather*}
D_n = B\biggl(x_0, \sum_{i=1}^k\frac{r_0}{2^i}\biggr) \textrm{ and } E_n = B\left(x_0,r_0 + \frac{r_0}{2^k}\right)
\end{gather*} 
so that $\left(D_n\right)$ and $\left(E_n\right)$ are respectively increasing and decreasing sequences of closed balls, as illustrated in Figure \ref{mcontinuity} for $\mathbb{R}^2$,  such that  
\begin{equation}
\label{DnBEn}
D_n \subset B\left(x_n,r_0\right)  \subset E_n.
\end{equation}
Applying the continuity of measures, see for example  \cite[Proposition 1.6]{Book_KJF1}, using Equation \eqref{mSB} and the fact that $\mu_u\left(E_{N_1}\right) < +\infty$, we obtain
\begin{gather*}
\lim_{n \rightarrow \infty}\mu_u\left(D_n\right) = \mu_u\left(\bigcup_{n=N_1}^\infty D_n\right) = \mu_u\left(S\left(x_0,r_0\right)\right)   \\ = \mu_u\left(B\left(x_0,r_0\right)\right) = \mu_u\left(\bigcap_{n=N_1}^\infty E_n\right) = \lim_{n \rightarrow \infty}\mu_u\left(E_n\right)
\end{gather*}
which, taken together with \eqref{DnBEn}, implies
\begin{gather*}
\lim_{n \rightarrow \infty}\mu_u\left(B(x_n,r_0)\right) = \mu_u\left(B(x_0,r_0)\right).
\end{gather*}
\end{proof}
\begin{lem}
\label{ctyae}
Let $f:[\alpha, \beta)\rightarrow \mathbb{R}$ be continuous from the right and let $A\subset [\alpha, \beta)$ be the set of points at which f is discontinuous. Then $A$ is countable.
\end{lem}

\begin{proof}
Let $A(1/2^k)\subset A$ be the set of points where the discontinuity is at least $1/2^k$, so for $a \in A(1/2^k)$ there exists a sequence $(x_n)$ with $x_n \rightarrow a^-$ such that 
\begin{equation}
\label{rtcts1}
\left|f(x_n) - f(a)\right| \geq 1/2^k \textrm{ for all } n.
\end{equation} 
As $f$ is continuous from the right at $a$ there exists $\delta(a,1/2^{k+1}) > 0$ such that $\left|f(x) - f(a)\right| < 1/2^{k+1}$ for all $x \in (a,a +  \delta(a,1/2^{k+1}))$. It follows that  
\begin{equation}
\label{rtcts2}
\left|f(x) - f(y)\right| < 1/2^k \textrm{ for all } x,y \in (a,a +  \delta(a,1/2^{k+1})).
\end{equation} 
For $a \in A(1/2^k)$ let $g(a)$ be a chosen rational number in the interval $(a,a +  \delta(a,1/2^{k+1}))$, then this defines a map $g: A(1/2^k)\rightarrow \mathbb{Q}$. To show $g$ is an injection we assume for a contradiction that there exist $a,b \in A(1/2^k)$ with $a \neq b$ and $g(a)=g(b)$. This means $(a,a +  \delta(a,1/2^{k+1})) \cap (b,b + \delta(b,1/2^{k+1})) \neq \emptyset$ and without loss of generality we may assume $b < a$. Further we may choose $n$ large enough so that $b < x_n < a < b + \delta(b,1/2^{k+1})$. It follows by \eqref{rtcts1} that $\left|f(x_n) - f(a)\right| \geq 1/2^k$ and by \eqref{rtcts2} that $\left|f(x_n) - f(a)\right| < 1/2^k$.
This contradiction implies $g$ is an injection and $A(1/2^k)$ is at most countable. As $A=\bigcup_{k=1}^\infty A(1/2^k)$, $A$ is countable. 
\end{proof}

For $q=0$, $M_u^0\left(F_u, r\right)$ is Lalley's packing function as described in \cite{Paper_Lalley} and is the maximum cardinality of any $r$-separated subset of $F_u$.  $M_u^0\left(F_u, r\right)$ is a decreasing integer-valued function which takes the value $1$ for all sufficiently large $r$ and is piecewise continuous with only finitely many discontinuities on any compact interval $[c,d]\subset (0,+\infty)$.

\begin{lem}
\label{MFvrcty}
Let $[\alpha, \beta) \subset (0,+\infty)$ be an interval where $M_u^0\left(F_u, r\right)$ is constant with $M_u^0\left(F_u, r\right) = N$ for all $r \in [\alpha, \beta)$ where $\alpha$ and $\beta$ are adjacent points of discontinuity of $M_u^0\left(F_u, r\right)$.

Then $M_u^q\left(F_u, r\right)$ is continuous from the right  and so Riemann integrable on $[\alpha, \beta)$.
\end{lem}

\begin{proof} For convenience let $M(r) = M_u^q\left(F_u, r\right)$. To prove this we need some new notation, let
\begin{equation*}
D_{[\alpha, \beta)} =  \left\{  (\bx,r) : r \in [\alpha, \beta), \ \bx = (x_1,\ldots , x_N) \in F_u^N \textrm{ and } \left|x_i - x_j\right| > 2r \textrm{ for } i \neq j\right\}
\end{equation*}
so that $(\bx, \widetilde{r}) \in D_{[\alpha, \beta)}$ means $\bx = (x_1,\ldots , x_N)$ is an $ \widetilde{r}$-separated subset of $F_u$. 

Let $\overline{D}_{[\alpha, \beta]}$ be the closure of $D_{[\alpha, \beta)}$ which is given by
\begin{equation*}
\overline{D}_{[\alpha, \beta]} =  \left\{  (\bx, r) : r \in [\alpha, \beta], \ \bx = (x_1,\ldots , x_N) \in F_u^N \textrm{ and } \left|x_i - x_j\right| \geq 2r \textrm{ for } i \neq j\right\}
\end{equation*}
where $\overline{D}_{[\alpha, \beta]}$ is a compact subset of $(\mathbb{R}^m)^N \times \mathbb{R} = \mathbb{R}^{mN+1}$ and so is a complete metric space with respect to the Euclidean metric on $\mathbb{R}^{mN+1}$.

Let $f: \overline{D}_{[\alpha, \beta]}\rightarrow \mathbb{R}^+$ be defined  by
\begin{equation*}
f(\bx, r) = \sum_{i=1}^N \mu_u\left(B(x_i, r)\right)^q 
\end{equation*}
then $f$ is continuous on $\overline{D}_{[\alpha, \beta]}$ by Lemma \ref{mcty} (c). As $\supp\mu_u = F_u$ for $x_i \in F_u$ $\mu_u\left(B(x_i,r)\right) > 0$ from the definition of the support of a measure. 

From these definitions it follows that 
\begin{equation}
\label{Mr}
M(r) = \sup_{(\bx, r) \in D_{[\alpha, \beta)} } f(\bx, r)
\end{equation}
where the supremum is over $D_{[\alpha, \beta)}$ not $\overline{D}_{[\alpha, \beta]}$ (see Equation \eqref{$q$th packing moment}).

\begin{figure}[!htb]
\begin{center}
\includegraphics[trim = 22mm 180mm 32mm 10mm, clip, scale=0.7]{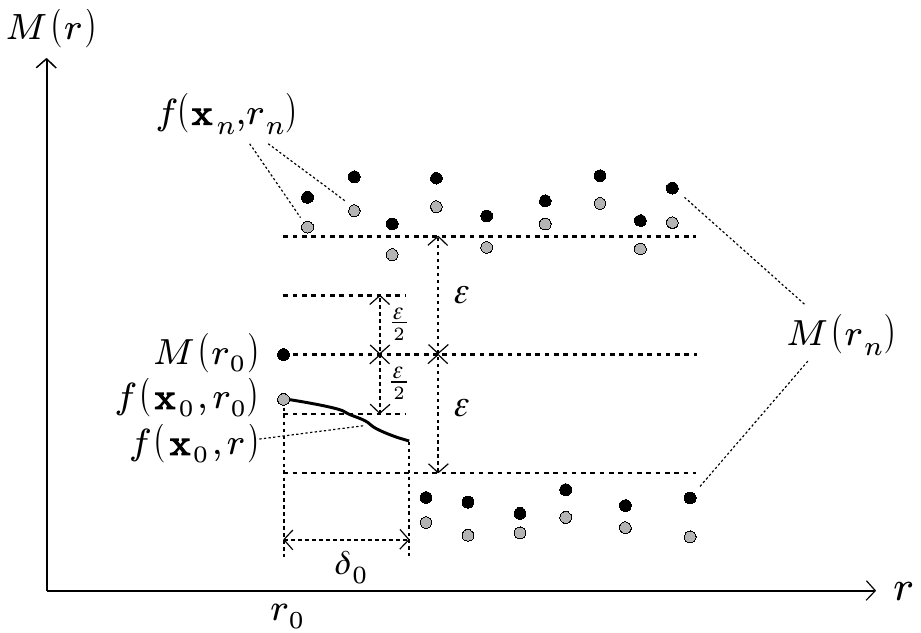}
\end{center}
\caption{A sequence $(r_n)$ with $r_n \rightarrow r_0^+$ used to show $M(r) = M_u^q\left(F_u, r\right)$ is continuous from the right.}
\label{Mrcts}
\end{figure}

For a contradiction we assume $M(r)$ is not continuous from the right. This means there exists $\varepsilon > 0$ and a sequence $(r_n)$ with $r_n\rightarrow r_0^+$ for some $r_0 \in [\alpha, \beta)$ such that $\left|M(r_n) - M(r_0)\right| \geq \varepsilon$ for all $n$. As illustrated in Figure \ref{Mrcts}, using Equation \eqref{Mr}, we can choose $(\bx_n, r_n) \in D_{[\alpha, \beta)}$ such that $M(r_n) - f(\bx_n, r_n) < \varepsilon/2$ and in particular at $M(r_0)$ we can choose $(\bx_0, r_0) \in D_{[\alpha, \beta)}$ so that $M(r_0) - f(\bx_0, r_0) < \varepsilon/2$.

Using the continuity of $f$ at $(\bx_0, r_0)$ and the fact that the closed balls in an $r$-separated packing always have a positive minimum distance between them, we can increase the radii of the balls centred at $\bx_0$ by some $\delta_0 > 0$ whilst maintaining $\left|M(r_0) - f(\bx_0, r)\right| < \varepsilon$. This is drawn for $q<0$ in Figure \ref{Mrcts} with $f(\bx_0, r)$ decreasing as a function of $r$, $f(\bx_0, r)$ increases for $q>0$. By Equation \eqref{Mr} this implies $M(r_n) - M(r_0) \geq \varepsilon$ for all $r_n$ with $r_n - r_0 \leq \delta_0$. Excluding those $r_n$ for which $r_n - r_0 > \delta_0$ and relabelling from now on we may assume 
\begin{equation}
\label{fMr0}
f(\bx_n, r_n) - M(r_0) \geq \frac{\varepsilon}{2}
\end{equation}
for all $n$, as shown in Figure \ref{Mrcts}.

As $\left((\bx_n, r_n)\right)$ is a sequence in $D_{[\alpha, \beta)}\subset \overline{D}_{[\alpha, \beta]}$ and $\overline{D}_{[\alpha, \beta]}$ is a compact metric space it has a convergent subsequence, see \cite[Chapter 2, Theorem 21]{Book_Maddox}. Relabelling again, we may now take $\left((\bx_n, r_n)\right)$ to be a sequence in $D_{[\alpha, \beta)}$ which converges to $(\bx, r_0) \in \overline{D}_{[\alpha, \beta]}$ where $\bx$ may or may not be an $r_0$-separated subset. Consider the sequence $\left((\bx_n, r_0)\right)$ which also converges to $(\bx, r_0)$ where each $\bx_n$ is an $r_0$-separated subset since $r_0 < r_n$. Because $f(\bx_n, r_0) \leq M(r_0)$ for each $n$, it follows that $f(\bx, r_0) \leq M(r_0)$. Using Inequality \eqref{fMr0} we obtain
\begin{equation*}
f(\bx_n, r_n) - f(\bx, r_0) \geq \frac{\varepsilon}{2}
\end{equation*}
for all $n$, which contradicts the continuity of $f$ on $\overline{D}_{[\alpha, \beta]}$. This proves $M(r)$ is continuous from the right.

By Lemma \ref{ctyae}, $M(r)$ is continuous on $[\alpha, \beta)$ except for at most countably many points and this is enough to ensure it is Riemann integrable on $[\alpha, \beta)$, see \cite[Theorem 5.9]{Book_Taylor}.
\end{proof}

In fact the proof of Lemma \ref{MFvrcty} can be used to ``almost'' prove the following:

\medskip

\emph{$M_u^q\left(F_u,r\right)$ is piecewise continuous with only finitely many discontinuities on any compact interval $[c,d]\subset (0,+\infty)$ and the points of discontinuity are those of the packing function $M_u^0\left(F_u,r\right) $.}

\medskip

To prove this it is enough to show $M(r)$ is continuous from the left on $(\alpha, \beta)$. A proof by contradiction goes through in the same way as the proof of  Lemma \ref{MFvrcty} except for the very last step, so we can construct a sequence $\left((\bx_n, r_n)\right)$  in $D_{(\alpha, \beta)} \subset D_{[\alpha, \beta)}$ which converges to $(\bx, r_0) \in \overline{D}_{[\alpha, \beta]}$ with $r_n\rightarrow r_0^{-}$ and for which 
\begin{equation*}
f(\bx_n, r_n) - M(r_0) \geq \frac{\varepsilon}{2}
\end{equation*}
for all $n$.
To complete the proof we need to show $f(\bx, r_0) \leq M(r_0)$ where $\bx$ may not be $r_0$-separated with $\left|x_i - x_j\right| = 2r_0$ for some $i \neq j$. We can't use the same method for this as the one used in the proof of Lemma \ref{MFvrcty} because now $r_n < r_0$, however as $F_u$ is uncountable with no isolated points we should be able to move the centres of the balls in $\bx$ by as small a distance as we like to create an $r_0$-separated packing $(\widetilde{\bx},r_0) \in D_{(\alpha, \beta)}$ with $\left|f(\bx, r_0)-f(\widetilde{\bx},r_0)\right| < \varepsilon^\prime$ for any $\varepsilon^\prime > 0$ by the continuity of $f$, so that $f(\bx, r_0)< f(\widetilde{\bx},r_0) + \varepsilon^\prime \leq M(r_0) + \varepsilon^\prime$. This would imply $f(\bx, r_0) \leq M(r_0)$ as required. All that is needed then is a formal proof that if we can't change $(\bx, r_0)$ into $(\widetilde{\bx},r_0)$ in this continuous way then $r_0$ must be a point of discontinuity of the packing function $M_u^0\left(F_u,r\right) $, where there is a cardinality drop in the maximum number of balls in any $r$-separated packing.

\subsection{Directly Riemann integrable functions} \label{Directly Riemann integrable functions}
 
Let $z\left(t\right)$ be a function, $z : [0, \infty) \to \mathbb{R}$. For a fixed $h > 0$ let $m_n$ and $M_n$ denote the infimum and supremum respectively of $z\left(t\right)$ in the interval $[(n-1)h, nh]$. The function $z\left(t\right)$ is \emph{directly Riemann integrable} whenever the sums $h\sum_{n=1}^{\infty}m_n$ and $h\sum_{n=1}^{\infty}M_n$ converge absolutely to the same limit $I$ as $h \to 0^+$, in which case $I = \int z\left(t\right)dt$, see \cite{Book_Feller}.

The following sufficient condition for direct Riemann integrability is used in \cite{Paper_Lalley, Paper_Olsen_1}. 

If $z\left(t\right)$ is Riemann integrable on all compact subintervals of $[0, \infty)$ and there exist $c_1, c_2 > 0$ such that
\begin{equation}
\label{directly_Riemann_integrable}
\left|z\left( t \right)\right|  \leq  c_1e^{-c_2t}
\end{equation}
for all $t \in [0, \infty)$, then $z\left(t\right)$ is directly Riemann integrable.

\subsection{The spectral radius of a non-negative matrix} \label{spectral radius} 

Let $\bv=\left(v_1,  v_2,  \ldots ,  v_n\right)^\transpose$ and $\bw=\left(w_1,  w_2,  \ldots ,  w_n\right)^\transpose$ be two real $n$-dimensional (column) vectors. We define $\bv \leq \bw$ as follows 
\begin{equation*}
\bv \leq \bw \textup{ if and only if } v_i \leq w_i \textup{ for all } 1 \leq i \leq n,
\end{equation*}
and similarly 
\begin{equation*}
\bv < \bw \textup{ if and only if } v_i < w_i \textup{ for all } 1 \leq i \leq n,
\end{equation*}
We denote the zero vector as $\mathbf{0}=(0,  0,  \ldots ,  0)^\transpose$. A vector $\bv$ is a \emph{positive vector} if $\mathbf{0}<\bv$ and is a \emph{non-negative vector} if $\mathbf{0}\leq \bv$. Using $M_{ij}$ for the $ij$th entry of a matrix $\bM$, these definitions  can be extended to the set of real $n \times n$ matrices in the obvious way. A \emph{non-negative matrix} $\bM$, has $0 \leq M_{ij}$ for all $1 \leq i,j \leq n$. For two non-negative matrices $\bC$, $\bD$, $\bC\leq \bD$ means $C_{ij} \leq D_{ij}$ for all $1 \leq i,j \leq n$ and $\bC= \bD$ means $C_{ij} = D_{ij}$ for all $1 \leq i,j \leq n$. 
 
A non-negative $n \times n$ matrix $\bM$ is \emph{irreducible} if, for each $1 \leq i,j \leq n$, there exists $k=k(i,j)\in \mathbb{N}$, which may depend on $i$ and $j$, such that the $ij$th entry of $\bM^k$ is positive, that is  
\begin{equation*}
M_{ij}^k>0.
\end{equation*}
  
The Perron-Frobenius Theorem \cite[Theorem 1.1]{Book_Seneta}, ensures that if $\bM$ is a non-negative irreducible matrix then 
\begin{itemize}
\item  $\bM$ has a real eigenvalue $r>0$, such that $r\geq \left|\lambda\right|$ for any eigenvalue $\lambda \neq r$. 	  
\item  $r$ has strictly positive left and right eigenvectors, which are unique up to a scaling factor.
\end{itemize}
Let $\rho\left(\bM\right)=r$, then $\rho\left(\bM\right)$ is the \emph{spectral radius of the matrix} $\bM$, with

\begin{equation}
\label{spectral radius eqn}
\rho\left(\bM \right) = \max \left\{ \left|\lambda \right|:  \lambda \textrm{ is an eigenvalue of } \bM\right\}.
\end{equation} 

\subsection{The matrix $\bA\left(q, \beta\right)$} \label{matrix A}

The $n \times n$ non-negative matrix, $\bA\left(q, \beta\right) = (A_{uv})$, where $n=\#V$ is the number of vertices in the directed graph, has entries defined by 
\begin{equation*}
A_{uv}\left(q,\beta\right) = 
\begin{cases}  \sum_{e\in E_{uv}^1} p_e^qr_e^{\beta\left(q\right)} & \  \text{if $E_{uv}^1 \neq \emptyset$,}
\\
0 &  \  \text{if $E_{uv}^1 = \emptyset$.}
\end{cases}
\end{equation*}
Because the directed graph is strongly-connected $\mathbf{A}\left(q, \beta\right)$ will be irreducible for $q, \beta\left(q\right) \in \mathbb{R}$. 
It can be shown, using the Perron-Frobenius Theorem, see \cite{Paper_Edgar_Mauldin}, that for a given $q \in \mathbb{R}$ there exists a unique value of $\beta\left(q\right)$ such that $\rho\left(\mathbf{A}\left(q, \beta\right)\right) = 1$, and this defines the function, $\beta : \mathbb{R} \to \mathbb{R}$, implicitly as a function of $q$. We require $\rho\left(\mathbf{A}\left(q, \beta\right)\right) = 1$ in the application of the Renewal Theorem in Subsection \ref{A system of renewal equations}.

We will also use the notation $\bA\left(q, \beta, l\right)$ for the $l$th power of the matrix $\bA\left(q, \beta\right)$, so
\begin{equation*}
\bA\left(q, \beta, l)\right)= \left(\bA\left(q, \beta\right)\right)^l, 
\end{equation*}
where the $uv$th entry is
\begin{equation}
\label{Auv(q,beta,l)}
A_{uv}\left(q, \beta, l\right)=\sum_{\substack{\be \in E_{uv}^l }}p_{\be}^q r_\be^{\beta\left(q\right)}, 
\end{equation}
if $E_{uv}^l \neq \emptyset$ and is zero otherwise.

\subsection{The matrix $\bB\left(q, \gamma, l\right)$} \label{matrix B}
 
The $n \times n$ non-negative matrix $\bB\left(q, \gamma, l\right)$ is very closely related to the matrix $\bA\left(q, \beta, l\right)$. The $n$-vertex IFSs of Theorems  \ref{theorem 1} and  \ref{theorem 2}, are such that the OSC holds and the maps $(S_e)_{e \in E^1}$ are contracting similarities, so the OSC is equivalent to the SOSC (see Subsection \ref{separation conditions}). This means we may take $(U_u)_{u \in V}\subset \left(\mathbb{R}^m\right)^{n}$ to be a list of non-empty bounded open sets where for each vertex $u \in V$, we may choose a point $x_u \in F_u\cap U_u$ and a radius $\varrho_u>0$ such that $B(x_u, \varrho_u)\subset U_u$. Let $\varrho_{\textup{min}}=\min\left\{\varrho_u: u \in V\right\}$, then $B(x_u, \varrho_\textup{min})\subset U_u$ for all $u\in V$. 

\begin{lem}
\label{surjection phi_u}  
Let $(F_u)_{u \in V} \in \left(K\left(\mathbb{R}^{m}\right)\right)^{n}$ be the attractor of an $n$-vertex IFS as given in Equation \eqref{Invariance 1}. For each $u \in V$, let the mapping, $\phi_u : E_u^\mathbb{N} \to F_u$, be defined for each infinite path $\be \in E_u^\mathbb{N}$ by
\begin{equation}
\label{phi_u}
\phi_u(\be) = x, \textrm{ where } \left\{x\right\}=\bigcap_{k=1}^{\infty} S_{\be \vert_{k}}(F_{t(\be \vert_{ k})}).
\end{equation}
Then $\phi_u$ is surjective. If the SSC is satisfied then $\phi_u$ is bijective.   
\end{lem} 
\begin{proof}See \cite[Lemma 1.3.5]{phdthesis_Boore}.
\end{proof}
For each $u \in V$, the mapping $\phi_u$ given in Equation \eqref{phi_u} is surjective so there exists an infinite path $\be_u \in E_u^\mathbb{N}$ with
\begin{equation*}
\left\{x_u\right\} = \bigcap_{k=1}^{\infty} S_{\be_u \vert_{ k}}(F_{t(\be_u \vert_{ k})})
\end{equation*}
Now $\left(S_{\be_u \vert_{ k}}(F_{t(\be_u \vert_{ k})})\right)$ is a decreasing sequence of non-empty compact sets whose diameters tend to zero as $k$ tends to infinity and so there exists $N(u)\in \mathbb{N}$ such that 
\begin{equation*}
\left|S_{\be_u \vert_j}(F_{t(\be_u \vert_j)})\right|< \varrho_\textup{min}
\end{equation*}
for all $j > N(u)$. Let $l>\max\left\{N(u) : u \in V\right\}$ be chosen and for each $u \in V$ put $\bl_u = \be_u \vert_l$. It follows that $S_{\bl_u}(F_{t(\bl_u)}) \subset B(x_u, \varrho_\textup{min})\subset U_u$, for each $u\in V$. This means that we can always create a family of paths $(\bl_u)_{u \in V}$, all of the same length $l=\left|\bl_u\right|$, where $l\in \mathbb{N}$ may be chosen as large as we like, such that
\begin{equation}
\label{l_u}
S_{\bl_u}(F_{t(\bl_u)}) \subset U_u,
\end{equation}  
for each $u \in V$. The paths $(\bl_u)_{u \in V}$ of Equation \eqref{l_u} are fundamental in the proofs of Theorems \ref{theorem 1} and \ref{theorem 2} and we use them now to define the matrix $\bB(q, \gamma, l)$.

The $n \times n$ non-negative matrix $\bB(q, \gamma, l)$ has its $uv$th entry given by
\begin{equation}
\label{Buv(q,gamma,l)}
B_{uv}\left(q,\gamma,l\right) =   \sum_{\substack{\be \in E_{uv}^l \\ \be \neq \bl_u}}p_{\be}^q r_{\be}^{\gamma\left(q\right)},
\end{equation}
if $\left\{\be: \be \in E_{uv}^l, \be \neq \bl_u\right\} \neq \emptyset$ and is zero otherwise. Equation \eqref{Buv(q,gamma,l)} should be compared with the $uv$th entry of  $\bA\left(q, \beta, l\right)$ in \eqref{Auv(q,beta,l)}. We always assume that $\bB\left(q, \gamma, l\right)$ is irreducible. It can be shown that for a given $q \in \mathbb{R}$ there exists a unique value of $\gamma\left(q\right)$ such that the spectral radius $\rho\left(\bB\left(q, \gamma, l\right)\right)=1$, and this defines the function, $\gamma : \mathbb{R} \to \mathbb{R}$, implicitly as a function of $q$. A proof can be constructed along the lines of that given for $\beta\left(q\right)$ and $\mathbf{A}\left(q,\beta\right)$ in \cite{Paper_Edgar_Mauldin}, using the Perron-Frobenius Theorem.  

Also the Perron-Frobenius Theorem, see Subsection \ref{spectral radius}, ensures the existence of a strictly positive right (column) eigenvector $\bb$ for the matrix  $\mathbf{B}\left(q,\gamma,l\right)$ with eigenvalue 1, so that
\begin{equation}
\label{b}
\mathbf{B}\left(q,\gamma,l\right)\bb  =  \bb.
\end{equation}

\subsubsection {The irreducibility of the matrix $\mathbf{B}\left(q,\gamma,l\right)$}
In the statements of Theorem \ref{theorem 1} and Theorem \ref{theorem 2} it is assumed that the matrix $\mathbf{B}\left(q,\gamma,l\right)$ is irreducible. As this is a genuine assumption it would need to be verified on a case by case basis. This is an area which it would be interesting to investigate further but for now we make do with a few facts about matrices that may be helpful in any specific cases.
 
From \cite{Book_Minc}, if $\mathbf{C}$ is a non-negative, irreducible matrix, with index of primitivity $h$, where 
\begin{equation*}
h = \#  \left\{\lambda: \lambda \textrm{ is an eigenvalue of } \mathbf{C} \textrm{ with } \left|\lambda\right|=\rho(\mathbf{C} )\right\},
\end{equation*}
then
\begin{center}
$\mathbf{C}^k$ is irreducible if and only if $(k, h)=1$.
\end{center}

A matrix is primitive if $h=1$ and so if $\mathbf{A}\left(q,\beta,1\right)$ is primitive then $\mathbf{A}\left(q,\beta,k\right)$ is irreducible for all $k \in \mathbb{N}$. 

We also note that if $\mathbf{A}\left(q,\beta,1\right)$ has at least one positive diagonal element then it is primitive. 

As noted in Subsection \ref{matrix B} the length $l$ of the paths $(\bl_v)_{v \in V}$, in Equation (\ref{l_u}) can be chosen to be as large as we like, which means we can always ensure that $(l,h)=1$ so that $\mathbf{A}\left(q,\beta,l\right)$ is irreducible. The matrix 
$\mathbf{B}\left(q,\gamma,l\right)$ is very closely related to $\mathbf{A}\left(q,\beta,l\right)$ so this may be of help in determining the irreducibility of $\mathbf{B}\left(q,\gamma,l\right)$.

Finally we note that our definition of an $n$-vertex IFS given in Subsection \ref{nvertex} assumes that each vertex in the directed graph has at least two edges leaving it. 

\subsection {The value of $\delta$} \label{deltadef}
In the statement of Theorem \ref{theorem 2} a small constant $\delta$ is used, here we define it using the paths $(\bl_v)_{v \in V}$ of Subsection \ref{matrix B} and the open sets $(U_v)_{v \in V}$ of the OSC/SOSC. The compactness of $S_{\bl_v}(F_{t(\bl_v)})$ means that the distance from $S_{\bl_v}(F_{t(\bl_v)})$ to the closed set $\mathbb{R}^m \setminus U_v$ is positive, and this leads to an associated list of  positive constants $(c_v)_{v \in V}$ where 
\begin{equation}
\label{l_2}
c_v = \dist\left(S_{\bl_v}(F_{t(\bl_v)}), \, \mathbb{R}^m \setminus U_v\right) > 0.
\end{equation}

We define $r_{\min}= \min\left\{r_e  :  e \in E^1 \right\}$ and $d_{\max} = \max \left\{\left|F_v\right|  :  v \in V \right\}$, with $r_{\max}$ and $d_{\min}$ defined similarly and we also put $c_{\min}=\min \{c_v  :  v \in V \}$.

We may always choose $N \in \mathbb{N}$ large enough so that
\begin{equation}
\label{N}
\frac{2d_{\max}}{c_{\min}} \leq \frac{1}{r_{\max}^{N-1}},
\end{equation}
and for such $N$ we now define $\delta$ as
\begin{equation}
\label{delta}
\delta=r_{\min}^{N+l+1}d_{\min}.
\end{equation} 
Because we are not restricted in our choice of the length $l$ of the paths $(\bl_v)_{v \in V}$ it is clear that from now on we may assume $0<\delta<1$ for any given system. We use Inequality \eqref{N} and Equation \eqref{delta} repeatedly in Section \ref{four}, for just one example see Lemma \ref{OSCh}.

\subsection{A lattice matrix of measures} \label{lattice matrix}
Let $\mathcal{M}\left(\mathbb{R}\right)$ denote the space of all Borel measures defined on $\mathbb{R}$ and let 
\begin{displaymath}
\pmb{\mu} = 
\left(\begin{array}{ccc} 
\mu_{11} & \cdots & \mu_{1n} \\
\vdots & \ & \vdots \\
\mu_{n1} & \cdots & \mu_{nn} \\
\end{array}\right)
\end{displaymath}
be an $n \times n$ matrix of measures with $\mu_{ij} \in \mathcal{M}\left(\mathbb{R}\right)$ for all $i,j$. 

In this subsection we use the notation $\left\langle A\right\rangle_{\textrm{group,+}}$ for the additive commutative group generated by the elements of a set $A\subset \mathbb{R}$.

A measure $\nu$ is \emph{arithmetic with span $\kappa > 0$} if and only if $\supp\nu \subset \kappa  \mathbb{Z} $, where $\kappa$ is the largest such number. This means that $\left\langle \supp \nu \right\rangle_{\textrm{group,+}} = \kappa \mathbb{Z}$. 

$\nu$ is a \emph{lattice measure with span $\kappa > 0$} if and only if there exist real numbers $c$, and $\kappa > 0$, such that $\supp\nu \subset   c  +  \kappa \mathbb{Z}$, where $\kappa$ is taken to be the largest such number. 

Following Definition 3.1 in Crump's paper, \cite{Paper_Crump}, we say that the matrix $\pmb{\mu}$ is a \emph{lattice matrix} if the following conditions are met: 
\begin{description}
\item (a)	  Each $\mu_{ii}$ is arithmetic with span $\lambda_{ii}>0$. That is $\left\langle \supp \mu_{ii}\right\rangle_{\textrm{group,+}} = \lambda_{ii} \mathbb{Z}$.
\item (b)		Each $\mu_{ij},\, i\neq j$, is a lattice measure with span $\lambda_{ij}>0$. That is there exist real numbers $b_{ij}$ and $\lambda_{ij} > 0$ such that $\supp \mu_{ij} \subset  b_{ij} + \lambda_{ij} \mathbb{Z} $ where $\lambda_{ij}$ is taken to be the largest such number.  
\item (c)		Each $\lambda_{ij}$ is an integer multiple of some number, the largest such number we shall call $\lambda$. That is $\left\langle \left\{ \lambda_{ij} : 1\leq i,j \leq n \right\} \right\rangle_{\textrm{group,+}} = \lambda \mathbb{Z}$.
\item (d)		If $a_{ij} \in \supp\mu_{ij}$, $a_{jk} \in \supp\mu_{jk}$ and $a_{ik} \in \supp\mu_{ik}$ then $a_{ij} + a_{jk} = a_{ik} + n\lambda$, for some $n\in \mathbb{Z}$ (where n may depend on $i,j$ and $k$). That is for all $i,j,k$ we have $\supp \mu_{ij} + \supp \mu_{jk} - \supp \mu_{ik} \subset \lambda \mathbb{Z}$.
\end{description}

The unique number $\lambda$ is called the \emph{span} of $\pmb{\mu}$.

\medskip

Associated with an $n$-vertex IFS with probabilities is a square $n\times n$ matrix of finite measures $\mathbf{P} = \left(P_{uv}\right)$, where $n=\#V$ is the number of vertices in the graph, and the $uv$th
entry $P_{uv}\in \mathcal{M}\left(\mathbb{R}\right)$ is defined as 
\begin{equation}
\label{P_uv}
P_{uv} = 
\begin{cases} \sum_{e\in E_{uv}^1} p_e^qr_e^{\beta\left(q\right)}\delta_{\ln (1/r_e)} & \  \text{if $E_{uv}^1 \neq \emptyset$,}
\\
0 &  \  \text{if $E_{uv}^1 = \emptyset$.}
\end{cases}
\end{equation}
Here $\delta_{\ln (1/r_e)}$ is the Dirac measure defined, for $x \in \mathbb{R}$, as 
\begin{equation*}
\delta_x\left(B\right)  =   
\begin{cases} 1 & \  \text{if $x \in B$,}
\\
0 &  \  \text{if $x\notin B$,}
\end{cases}
\end{equation*}
for all Borel sets $B \subset \mathbb{R}$. 

The function $\beta : \mathbb{R} \to \mathbb{R}$ in \eqref{P_uv} was defined implicitly in Subsection \ref{matrix A} using the matrix $\bA\left(q, \beta\right)$. Evaluating each measure $P_{uv}$ over $\mathbb{R}$ we obtain
\begin{equation*}
A_{uv}\left(q,\beta\right) = P_{uv}\left(\mathbb{R}\right) =
\begin{cases} \sum_{e\in E_{uv}^1} p_e^qr_e^{\beta\left(q\right)} & \  \text{if $E_{uv}^1 \neq \emptyset$,}
\\
0 &  \  \text{if $E_{uv}^1 = \emptyset$.}
\end{cases}
\end{equation*}
We also use the notation $\mathbf{P}\left(\mathbb{R}\right)=\mathbf{A}\left(q, \beta\right)$.

\section{ Proof of Theorem \ref{theorem 1}}\label{three}

For the proof of Theorem \ref{theorem 1} we need to apply the Renewal Theorem for a system of renewal equations as stated in \cite[Theorem 3.1(ii)]{Paper_Crump} and restated here as Theorem \ref{the_renewal_theorem}. This extends the standard version of the Renewal Theorem given in \cite[Chapter11, Theorem 2]{Book_Feller} which was used in \cite{Paper_Lalley,Paper_Olsen_1}.

\subsection{The Renewal Theorem for a system of renewal equations} \label{2renewal_theorem}

A system of renewal equations is of the form 
\begin{equation}
M_i\left(t\right) = \int_0^t{\sum_{j = 1}^m M_j\left(t-u\right)dF_{ij}\left(u\right)}\, + \, z_i\left(t\right),\quad \quad i=1,2 \cdots,m 
\label{system_of_renewal_equations}
\end{equation}
where each $z_i\left(t\right)$ is bounded on every finite interval and vanishes for $t < 0$ and each $F_{ij}$ is a finite Borel measure with $F_{ij}((-\infty,0))=0$. If we let $\mathbf{F} = (F_{ij})$, $\mathbf{Z}\left(t\right) = (z_1\left(t\right), \cdots ,z_m\left(t\right))^\transpose$ and $\mathbf{M}\left(t\right) = (M_1\left(t\right), \cdots ,M_m\left(t\right))^\transpose$ we can put Equation (\ref{system_of_renewal_equations}) in the compact form
\begin{equation}
\mathbf{M}\left(t\right) = \mathbf{F}\, * \,\mathbf{M}\left(t\right)\,+\, \mathbf{Z}\left(t\right)
\label{compact_system_of_renewal_equations}
\end{equation}
where $*$ behaves in exactly the same way as matrix multiplication except we convolve elements instead of multiplying them so that 
\begin{equation*}
\left(F_{ij}*M_j\right)\left(t\right) = \int_0^t{ M_j\left(t-u\right)dF_{ij}\left(u\right)}.
\end{equation*}
We use the notation $\mathbf{F}\left(B\right)$ to mean the matrix of real numbers $(F_{ij}\left(B\right))$ where $B\subset \mathbb{R}$ is a Borel set. There are three conditions given by Crump in \cite{Paper_Crump} that need to be satisfied: 
\begin{description}
\item (i) The largest eigenvalue of the matrix $\mathbf{F}\left((-\infty,0]\right)$ is less than 1. 
\item (ii)	The matrix $\mathbf{F}\left(\mathbb{R}\right)$ has all non-negative entries. 
\item (iii) For at least one pair i, j, the finite measure $F_{ij}$ is not concentrated at the origin.
\end{description}

In theory the values of $\alpha_{ij}$ in the statement of Theorem \ref{the_renewal_theorem} can be calculated explicitly for a particular system, see the text preceding \cite[Theorem 3.1(ii)]{Paper_Crump} for details, so the limits given can be determined precisely. The statement of the Renewal Theorem that follows is \cite[Theorem 3.1 (ii)]{Paper_Crump}.

\begin{thm}[The Renewal Theorem]
\label{the_renewal_theorem}
Suppose the spectral radius $\rho (\mathbf{F}(\mathbb{R})) = 1$. Let the vector $\mathbf{M}\left(t\right)$ be as in \textup{(\ref{compact_system_of_renewal_equations})} and suppose each $z_i\left(t\right)$ in $\mathbf{Z}\left(t\right)$ is directly Riemann integrable.
\begin{description}
\item \textup{(a)}  If  $\mathbf{F}$ is a lattice matrix  with span $\lambda > 0$ then for each $i$ 
\begin{equation*}
M_i\left(t + n\lambda \right) \to  \sum_{j=1}^m  \lambda  \alpha_{ij} \sum_{l = -\infty}^{\infty}z_j\left(t + \lambda l\right)
\end{equation*}
as $n \to \infty$.
\item \textup{(b)} If  $\mathbf{F}$ is not a lattice matrix then for each $i$ 
\begin{equation*}
M_i\left(t\right)  \to  \sum_{j=1}^m \alpha_{ij} \int_0^{\infty}z_j\left(u\right)du
\end{equation*}
as $t \to \infty$. 
\end{description}

\end{thm}

\subsection{A system of renewal equations} \label{A system of renewal equations}

In this subsection we derive a system of renewal equations of the form given in (\ref{compact_system_of_renewal_equations}). Consider any $q\in \mathbb{R}$ to be fixed. For each vertex $u \in V$, let $L_u:\left(0,\infty\right) \to \mathbb {R}$ be the error function defined, for $r>0$, by 
\begin{equation}
\label{calc_1}
L_u\left(r\right)= M_u^q\left(F_u,r\right)  - \sum_{e\in E_u^1} p_e^q  M_{t(e)}^q\left(F_{t(e)}, r_e^{-1}r\right).
\end{equation} 
Let $\mathbf{P} = \left(P_{uv}\right)$ be the $n\times n$ matrix of finite measures as defined in Subsection \ref{lattice matrix} and let $\mathbf{A}\left(q, \beta\right)=\left(A_{uv}\right)$ be the $n\times n$ matrix of real numbers as defined in Section \ref{matrix A}, where $n=\#V$, is the number of vertices in the directed graph.

For each $u\in V$, let $H_u : [0, \infty) \to \mathbb {R}^+$, be the function defined by 
\begin{equation*}
H_u\left(t\right) = e^{-t\beta\left(q\right)}M_u^q\left(F_u, e^{-t}\right),
\end{equation*}
and let $h_u : [0, \infty) \to \mathbb {R}$, be the function defined by
\begin{equation}
\label{h_u function}
h_u\left(t\right) = e^{-t\beta\left(q\right)}L_u\left(e^{-t}\right),
\end{equation}
which means, by Equation (\ref{calc_1}), that for $t \geq 0$,
\begin{equation}
\label{h_u}
h_u\left(t\right)= e^{-t\beta\left(q\right)}\biggl( M_u^q\left(F_u,e^{-t}\right) - \sum_{e\in E_u^1} p_e^q M_{t(e)}^q\left(F_{t(e)}, r_e^{-1}e^{-t}\right) \biggr).
\end{equation} 
We note that 
\begin{equation}
\label{calc_2}
H_u\left(t - \ln{(r_e^{-1})}\right) \ = \ (r_e^{-1}e^{-t})^{\beta\left(q\right)}M_u^q\left(F_u, r_e^{-1}e^{-t}\right).
\end{equation}
For each $u\in V$,
\begin{align*}
H_u\left(t\right)&= e^{-t\beta\left(q\right)}M_u^q\left(F_u, e^{-t}\right)   \\
&=  e^{-t\beta\left(q\right)}\biggl( \ \sum_{ e\in E_u^1 } p_e^q  M_{t(e)}^q\left(F_{t(e)}, r_e^{-1}e^{-t}\right)+ L_u\left(e^{-t}\right) \ \biggr)   && (\textrm{by (\ref{calc_1})}) \\
&= \sum_{ e\in E_u^1 }p_e^qr_e^{\beta\left(q\right)}(r_e^{-1}e^{-t})^{\beta\left(q\right)} M_{t(e)}^q\left(F_{t(e)}, e^{-t}r_e^{-1}\right) + h_u\left(t\right) && (\textrm{by (\ref{h_u function})})\\
&= \sum_{ e\in E_u^1 }p_e^qr_e^{\beta\left(q\right)}H_{t(e)}\left(t - \ln{(r_e^{-1})}\right) + h_u\left(t\right) && (\textrm{by (\ref{calc_2})})\\
&= \sum_{ \substack{v\in V} } \biggl( \ \sum_{ \substack{e\in E_{uv}^1} } p_e^q r_e^{\beta\left(q\right)}H_v\left(t - \ln{(r_e^{-1})}\right) \ \biggr) + h_u\left(t\right)  && \\
&=  \sum_{ v\in V } \biggl( \ \int_0^t H_v\left(t - x)\right)dP_{uv}(x) \ \biggr) + h_u\left(t\right) && \\
&=  \int_0^t \sum_{ v\in V }H_v\left(t - x\right)dP_{uv}(x) + h_u\left(t\right), 
\end{align*}
for large enough $t \geq \max \left\{  \ln(r_e^{-1}) \, : \, e \in E^1 \right\}$. The penultimate equality follows from the definition of the finite measure $P_{uv}$ in Subsection \ref{lattice matrix}. 

We now have a system of renewal equations (see Equation (\ref{system_of_renewal_equations}) in Subsection \ref{2renewal_theorem}), where for each $u \in V$ 
\begin{equation*}
H_u\left(t\right) = \int_0^t \sum_{ \substack{v\in V} } H_v\left(t - x\right)dP_{uv}(x) + h_u\left(t\right).
\end{equation*}
In compact form, (see Equation (\ref{compact_system_of_renewal_equations}) in Section \ref{2renewal_theorem}),
\begin{equation*}
\mathbf{H}\left(t\right) = \mathbf{P}\, * \,\mathbf{H}\left(t\right)\,+\, \mathbf{h}\left(t\right).
\end{equation*}

It is clear that Crump's conditions (i) and (ii) of Section \ref{2renewal_theorem} are satisfied by the matrices $\mathbf{P}(-\infty,0])$ and $\mathbf{P}\left(\mathbb{R}\right)$ respectively and that (iii) also holds. Also if $\mathbf{A}\left(q, \beta\right)$ is the matrix of Subsection \ref{matrix A}, then $\rho\left(\mathbf{P}\left(\mathbb{R}\right)\right)=\rho\left(\mathbf{A}\left(q, \beta\right)\right) =1$. 

If we assume for the moment that the functions $(h_u)_{u \in V}$ are directly Riemann integrable then we may apply the Renewal Theorem, Theorem \ref{the_renewal_theorem}, to obtain the following. 

By Theorem \ref{the_renewal_theorem}(a), if  $\mathbf{P}$ is a lattice matrix with span $\lambda>0$, then
\begin{equation*}
\lim_{n \to  +\infty}H_u(t+n\lambda) = \lim_{n \to  +\infty} \frac { M_u^q( F_u, e^{-(t + n\lambda)} ) }  { e^{-(t + n\lambda)(-\beta\left(q\right))} } = \sum_{v \in V}  \lambda  \alpha_{uv} \sum_{l = -\infty}^{\infty}h_v(t + \lambda l)= f_u\left(t\right),
\end{equation*}
for each $t \geq \max \left\{  \ln(r_e^{-1})  :  e \in E^1 \right\}$. Because the coefficients $\alpha_{uv}$ and the functions $h_v\left(t\right)$ depend on $q$ so does $f_u\left(t\right)$. Here $h_v\left(t\right)=0$ for $t<0$, for each $v \in V$, and $f_u:\mathbb{R} \to \mathbb{R}^+$ is a positive periodic function with $f_u\left(t\right)=f_u(t+n\lambda)$ for any $n \in \mathbb{Z}$. Taking logarithms of both sides of this equation gives
\begin{equation*}
\lim_{n \to  +\infty} \frac { \ln( M_u^q( F_u, e^{-(t + n\lambda)} ) ) }  { t + n\lambda }  =   \beta\left(q\right).  
\end{equation*}
For the rate of convergence of the last limit, it is convenient to define a function $g:\mathbb{N} \to \mathbb{R}^+$, by

\begin{equation*}
g\left(n\right)= \frac { M_u^q( F_u, e^{-(t + n\lambda)} ) }  { f_u\left(t\right)e^{-(t + n\lambda)(-\beta\left(q\right))} }, 
\end{equation*}
with $\lim_{n \to +\infty}f_u\left(t\right)g\left(n\right) = f_u\left(t\right)>0$.  It follows from Subsection \ref{thm1.1_OSC}, Statement (c$^\prime$) that $f_u\left(t\right)$ is bounded for all $t \geq \max \left\{  \ln(r_e^{-1}) : e \in E^1 \right\}$. This means that for any $t \geq \max \left\{  \ln(r_e^{-1}) : e \in E^1 \right\}$

\begin{align*}
\left|\frac { \ln\left( M_u^q\left( F_u, e^{-(t + n\lambda)} \right) \right) }  { (t + n\lambda) } - \beta\left(q\right)\right| &= \left|\frac { \ln\left( f_u\left(t\right)e^{-(t + n\lambda)(-\beta\left(q\right))}g\left(n\right) \right) }  { (t + n\lambda) } - \beta\left(q\right)\right| \\
&= \left|\frac { \ln\left( f_u\left(t\right)g\left(n\right) \right) }  { (t + n\lambda) } + \beta\left(q\right) - \beta\left(q\right)\right| \\
&= \left|\frac { \ln\left( f_u\left(t\right)g\left(n\right) \right) }  { (t + n\lambda) }\right| \\
&\leq \frac{K}{n},
\end{align*}
for large enough $n$, and some constant $K>0$. 

Therefore for $t \geq \max \left\{  \ln(r_e^{-1}) : e \in E^1 \right\}$

\begin{equation*} 
\frac { \ln\left( M_u^q\left( F_u, e^{-(t + n\lambda)} \right) \right) }  { (t + n\lambda) }  =   \beta\left(q\right)  +    O \left(\frac{1}{n}\right),
\end{equation*}		
as $n \to +\infty$. 

By Theorem \ref{the_renewal_theorem}(b), if  $\mathbf{P}$ is not a lattice matrix then

\begin{equation*}
\lim_{t \to +\infty}H_u\left(t\right) =  \lim_{t \to +\infty} \frac { M_u^q\left(F_u,e^{-t}\right) }  { e^{-t(-\beta\left(q\right))} } = \lim_{r \to 0^+} \frac { M_u^q\left(F_u,r\right) }  { r^{-\beta\left(q\right)} } = \sum_{v \in V} \alpha_{uv} \int_0^{\infty}h_v(x)dx = C_u.  
\end{equation*}
The constant $C_u$ is positive. Because the coefficients $\alpha_{uv}$ and the functions $h_v\left(t\right)$ depend on $q$ so does $C_u$. Taking logarithms of both sides of this equation gives
\begin{equation*}
\lim_{\ r\to 0^+}\frac{\ln \left(M_u^q\left(F_u,r\right)\right)} {-\ln r} = \beta\left(q\right).   
\end{equation*}
For the rate of convergence of this last limit, let $g:(0,+\infty) \to \mathbb{R}^+$, be defined as
\begin{equation*}
g(r)= \frac { M_u^q\left(F_u,r\right) }  { C_ur^{-\beta\left(q\right)} }, 
\end{equation*}
where $\lim_{r \to 0^+}g(r)=1$. We now obtain
\begin{align*}
\left|\frac{\ln \left(M_u^q\left(F_u,r\right)\right)} {-\ln r} - \beta\left(q\right)\right| &= \left|\frac{\ln \left(C_ur^{-\beta\left(q\right)}g(r)\right)} {-\ln r} - \beta\left(q\right)\right| \\
&= \left|\frac { \ln\left( C_ug(r) \right) }  {-\ln r} + \beta\left(q\right) - \beta\left(q\right)\right| \\
&= \left|\frac { \ln\left( C_ug(r) \right) }  {-\ln r}\right| \\
&\leq \frac{K}{-\ln r},
\end{align*}
for small enough $r$, and some constant $K>0$. Therefore, as $r \to 0^+$,
\begin{equation*} 
\frac{\ln \left(M_u^q\left(F_u,r\right)\right)} {-\ln r}  =   \beta\left(q\right)  +    O \left(\frac{1}{-\ln r}\right).
\end{equation*}		

This means that the proof of Theorem \ref{theorem 1} will be complete once we have shown that the functions $(h_u)_{u \in V}$, given by Equation (\ref{h_u}), are directly Riemann integrable in both of the following cases,

\begin{description}
\item (i) \emph{$q\in \mathbb{R}$ and the SSC holds, }
\item (ii) \emph{$q\geq 0$ the OSC holds and the non-negative matrix $\bB\left(q, \gamma, l\right)$, as defined in Subsection \ref{matrix B}, is irreducible with $\rho\left(\bB\left(q, \gamma, l\right)\right) = 1$.}
\end{description}

We do this for (i) in Subsection \ref{thm1.1_SSC} and for (ii) in Subsection \ref{thm1.1_OSC}. 

\subsection{ Proof of Theorem \ref{theorem 1} (i) - SSC} \label{thm1.1_SSC}

In this section we prove that the functions $(h_u)_{u \in V}$, as given in Equation (\ref{h_u}), are directly Riemann integrable for 

\begin{description}
\item (i) \emph{$q\in \mathbb{R}$ and the SSC holds. }
\end{description}

Consider $u \in V$ as fixed and let
\begin{equation}
\label{distance between components}
\varepsilon = \frac{1}{2}\min\left\{\dist \left(S_e(F_{t(e)}),S_f(F_{t(f)})\right):  e,f\in E_u^1, \, e \neq f  \right\}
\end{equation}
By the SSC, for $e,f\in E_u^1$, $e \neq f$, $S_e(F_{t(e)})\cap S_f(F_{t(f)})=\emptyset$, and as $S_e(F_{t(e)})$, $S_f(F_{t(f)})$ are non-empty compact subsets of $F_u$, this implies $\varepsilon>0$.
 
\begin{lem}
\label{SSCa}
Let $q \in \mathbb{R}$, let $u\in V$ be fixed and let $r \in (0,\varepsilon)$. Then 
\begin{align*}
&\textup{(a) } M_u^q\left(F_u,r\right)  =  \sum_{e\in E_u^1 } M_u^q\left(S_e(F_{t(e)}),r\right). \\
&\textup{(b) } For \ each \ e \in E_u^1, \ M_u^q\left(S_e(F_{t(e)}),r\right) =  p_e^q M_{t(e)}^q\left(F_{t(e)}, r_e^{-1}r\right). \\
&\textup{(c) } M_u^q\left(F_u,r\right) =  \sum_{e\in E_u^1 } p_e^q M_{t(e)}^q\left(F_{t(e)}, r_e^{-1}r\right).
\end{align*}					

\end{lem}

\begin{proof} Part (c) follows immediately from parts (a) and (b).

(a) Let $D$ be an $r$-separated subset of $F_u$, then by Equation (\ref{Invariance 1}), 
\begin{equation*}
D = \bigcup_{ e\in E_u^1 } \left(D \cap S_e(F_{t(e)})\right)= \bigcup_{ e\in E_u^1 }D_e,
\end{equation*}
where the union is disjoint by the SSC and each $D_e= D \cap S_e(F_{t(e)})$ is an $r$-separated subset of $S_e(F_{t(e)})$. 

For each $e \in E_u^1$ let $D_e^{\prime}$ be an $r$-separated subset of $S_e(F_{t(e)})$. Then for $e,f \in E_u^1$, $e \neq f$, it follows by the SSC and the definition of $\varepsilon$, that for $0<r<\varepsilon$, 
\begin{equation*}
\bigcup_{ x\in D_e^{\prime}} B(x,r) \  \cap  \ \bigcup_{ y\in D_f^{\prime}} B(y,r) = \emptyset. 
\end{equation*}
This means that
\begin{equation*}
D^{\prime} = \bigcup_{e\in E_u^1} D_e^{\prime},
\end{equation*}
will be an $r$-separated subset of $F_u$, where the union is disjoint. 

From these observations we obtain  
\begin{equation*}
\sum_{x \in D} \mu_u\left(B(x,r)\right)^q = \sum_{e\in E_u^1 }\biggl( \ \sum_{x \in D_e}\mu_u\left(B(x,r)\right)^q \ \biggr) \leq  \sum_{ e\in E_u^1 } M_u^q\left(S_e(F_{t(e)}),r\right),
\end{equation*}
and
\begin{equation*}
M_u^q\left(F_u,r\right) \geq \sum_{x \in D^{\prime}} \mu_u\left(B(x,r)\right)^q  =  \sum_{ e\in E_u^1 }\biggl( \ \sum_{x \in D_e^{\prime} }\mu_u\left(B(x,r)\right)^q \ \biggr).
\end{equation*}
Taking the supremum over all $r$-separated subset $D$ in the first inequality, and all $r$-separated subsets $D_e^{\prime}$ in the second inequality, gives the required result.

(b) Let $D_e$ be an $r$-separated subset of $S_e(F_{t(e)})$ for $e \in E_u^1$. As $S_e$ is a similarity with contracting similarity ratio $r_e$ so $S_e^{-1}$ is a similarity with an expanding similarity ratio of $r_e^{-1}$ and for any $x \in D_e$, 
\begin{equation} 
\label{(i)}
 S_e^{-1}\left(B(x,r)\right) = B\left(S_e^{-1}(x),r_e^{-1}r\right).
\end{equation} 
As $\supp\mu_{t(f)} = F_{t(f)}$, $0<r<\varepsilon$, and the SSC is satisfied, it is clear that for all $f\in E_u^1$ with $f\neq e$,
\begin{equation}
\label{(ii)}
\mu_{t(f)}\left(S_f^{-1}(B(x,r))\right) = 0, \, \textrm{ for all } x\in D_e.   
\end{equation}
This means that
\begin{align*}
\sum_{x \in D_e} \mu_u\left(B(x,r)\right)^q &=\sum_{x \in D_e}\biggl( \ \sum_{ \substack{ f\in E_u^1}  }p_f \mu_{t(f)}\left(S_f^{-1}(B(x,r))\right)  \ \biggr)^q && (\textrm{by Equation (\ref{Invariance 2})}) \\
&=  \sum_{x \in D_e} \left( p_e \mu_{t(e)}\left(S_e^{-1}\left(B(x,r)\right)\right) \right)^q && (\textrm{by (\ref{(ii)})}) \\
&=  p_e^q\sum_{x \in D_e} \mu_{t(e)}\left(B(S_e^{-1}(x),r_e^{-1}r)\right)^q && (\textrm{by (\ref{(i)})}) \\
&=  p_e^q\sum_{x \in S_e^{-1}(D_e)} \mu_{t(e)}\left(B(x,r_e^{-1}r)\right)^q,    
\end{align*}
that is 
\begin{equation}
\label{D equality}
\sum_{x \in D_e} \mu_u\left(B(x,r)\right)^q= p_e^q\sum_{x \in S_e^{-1}(D_e)} \mu_{t(e)}\left(B(x,r_e^{-1}r)\right)^q.   
\end{equation}

From Equation (\ref{(i)}) $S_e^{-1}(D_e)$ is an $r_e^{-1}r$-separated subset of $F_{t(e)}$, so that 
\begin{equation*}
\sum_{x \in D_e} \mu_u\left(B(x,r)\right)^q = p_e^q\sum_{x \in S_e^{-1}(D_e)} \mu_{t(e)}\left(B(x,r_e^{-1}r)\right)^q \leq p_e^qM_{t(e)}^q\left(F_{t(e)}, r_e^{-1}r\right).
\end{equation*}
Similarly if $D$ is any $r_e^{-1}r$-separated subset of $F_{t(e)}$ then $D_e=S_e(D)$ will be an $r$-separated subset of $S_e(F_{t(e)})$, so that
\begin{equation*}
M_u^q\left(S_e(F_{t(e)}),r\right) \geq \sum_{x \in D_e} \mu_u\left(B(x,r)\right)^q \ = \ p_e^q\sum_{x \in D}\mu_{t(e)}\left(B(x,r_e^{-1}r)\right)^q,
\end{equation*}
where the last equality is obtained by putting $D_e=S_e(D)$ in Equation (\ref{D equality}). Taking the supremum over any $r$-separated subset $D_e$ of $S_e(F_{t(e)})$ in the first inequality, and over any $r_e^{-1}r$-separated subset $D$ of $F_{t(e)}$ in the second inequality completes the proof. 
\end{proof}

As defined in Equations (\ref{h_u function}) and (\ref{h_u}), the function $h_u$ is given by
\begin{equation*}
h_u\left(t\right)=e^{-t\beta\left(q\right)}L_u\left(e^{-t}\right)=e^{-t\beta\left(q\right)}\biggl( M_u^q\left(F_u,e^{-t}\right) - \sum_{e\in E_u^1} p_e^q M_{t(e)}^q\left(F_{t(e)}, r_e^{-1}e^{-t}\right) \biggr).
\end{equation*}

Without loss of generality we now assume  $\varepsilon<1$.

\medskip

(a) \emph{$h_u\left(t\right)$ is Riemann integrable on any compact interval $[a,b]\subset [0,+\infty)$}.
 
If $t \in [0, +\infty)$ then $e^{-t}$, $ r_e^{-1}e^{-t} \in (0,R]$ where $R = \max \{  r_e^{-1} :  e \in E^1 \}$. Let $[c,d]\subset (0,R]$ be any compact interval then $M_v^q\left(F_v, r\right)$ is Riemann integrable on $[c,d]$ for any $v \in V$ by Lemma \ref{MFvrcty} which implies (a). 

\medskip

(b) \emph{$h_u\left(t\right)$ is bounded for $t \in [0,-\ln \varepsilon]$}.

If $t \in [0,-\ln \varepsilon]$ then $e^{-t}$, $r_e^{-1}e^{-t} \in [\varepsilon,R]$ where $R$ is as defined in part (a). It follows from the definition of the $q$th packing moment in Subsection \ref{The $q$th packing moment} that $M_v^q\left(F_v,r\right)$ is bounded for $r \in [\varepsilon,R]$ for each $v \in V$. Therefore $h_u\left(t\right)$ is bounded on the interval $[0,-\ln \varepsilon]$.

\medskip

(c) \emph{$h_u\left(t\right)=0$  for $t \in (-\ln \varepsilon,+\infty)$}.

For $t \in (-\ln \varepsilon,+\infty)$, that is for $e^{-t} \in (0,  \varepsilon)$, Lemma \ref{SSCa}(c) implies
\begin{equation*}
L_u\left(e^{-t}\right)=M_u^q\left(F_u,e^{-t}\right) - \sum_{ e\in E_u^1 } p_e^q  M_{t(e)}^q\left(F_{t(e)}, r_e^{-1}e^{-t}\right) = 0,
\end{equation*}
see Equation (\ref{calc_1}). This proves $h_u\left(t\right)=0$ for $t \in (-\ln \varepsilon,+\infty)$. 

\medskip

From statements (b) and (c) it is clear that we can always find positive constants $c_1$ and $c_2$, so that Equation (\ref{directly_Riemann_integrable}) of Subsection \ref{Directly Riemann integrable functions} holds. Taken together with Statement (a), this is enough to prove that $h_u\left(t\right)$ is directly Riemann integrable.

\subsection{ Proof of Theorem \ref{theorem 1} (ii) - OSC } \label{thm1.1_OSC}

We now need to prove that the functions $(h_u)_{u \in V}$, as given in Equation (\ref{h_u}), are directly Riemann integrable for 

\begin{description}
\item (ii) \emph{$q\geq 0$ the OSC holds and the non-negative matrix $\bB\left(q, \gamma, l\right)$, as defined in Subsection \ref{matrix B}, is irreducible with $\rho\left(\bB\left(q, \gamma, l\right)\right) = 1$.}
\end{description}

For the OSC, statement (c) of Section \ref{thm1.1_SSC} may not be true, that is it may not be the case that $h_u\left(t\right)=0$ for $t \in (-\ln \varepsilon,+\infty)$. Indeed we may have a situation where $S_e\left(F_{t(e)})\cap S_f(F_{t(f)}\right) \neq \emptyset$ for $e,f \in E_u^1$, $f \neq e$, so that $\varepsilon=0$, where $\varepsilon$ is the minimum distance betweeen level-$1$ elementary pieces as defined in Equation (\ref{distance between components}). Instead to show $h_u\left(t\right)$ is directly Riemann integrable we will show, that for some small $\delta$,  $0<\delta<1$,

(a) \emph{$h_u\left(t\right)$ is Riemann integrable on any compact interval $[a,b]\subset [0,+\infty)$}.

(b) \emph{$h_u\left(t\right)$ is bounded for $t \in [0,-\ln \delta]$}.

(c$^\prime$) \emph{There exist positive constants $d_1$ and $d_2$ such that $\left|h_u\left(t\right)\right|  \leq  d_1 e^{-d_2 t}$,  for $t \in (-\ln \delta,+\infty)$}.

Statements (a) and (b) hold by exactly the same arguments as those given in Section \ref{thm1.1_SSC}. Statements (b) and (c$^\prime$) imply that there exist positive constants $c_1$ and $c_2$, so that Equation (\ref{directly_Riemann_integrable}) of Subsection \ref{Directly Riemann integrable functions} holds. It is clear then that to show $h_u\left(t\right)$ is directly Riemann integrable we need only to prove statement (c$^\prime$) which we do in two parts.

The first part is to prove Lemma \ref{OSC1} below, which states that for any $\delta$, $0 < \delta < 1$, for $t \in (-\ln \delta,+\infty)$, 
\begin{equation}
\label{bound 1}
\left|h_u\left(t\right)\right|\leq e^{-t\beta\left(q\right)}\sum_{ e\in E_u^1 } \biggl( \ \sum_{ \substack{ f\in E_u^1 \\ f \neq e } }Q_{e,f}^q\left(e^{-t}\right) \ \biggr),
\end{equation}
where $Q_{e,f}^q\left(r\right)$ is a $q$th packing moment at the vertex $u \in V$ as defined in Equation (\ref{Q_{e,f}}). In fact it is in the proof of this inequality that we require $q \geq 0$, see the proof of Lemma \ref{OSCb}(b) below.

The second part is to prove Theorem \ref{theorem 2}, this states that for a suitable choice of $\delta$, $0<\delta < 1$, as given in Subsection \ref{delta}, for $r \in (0,\delta)$ and $q \in \mathbb{R}$, 
\begin{equation}
\label{bound 2}
Q_{e,f}^q\left(r\right)   \leq   C_{e,f}\left(q\right)r^{-\gamma\left(q\right)}, 
\end{equation}
for some positive number $ C_{e,f}\left(q\right)$.

Now as $t \in (-\ln \delta,+\infty)$ if and only if $e^{-t} \in (0,\delta)$, inequalities (\ref{bound 1}) and (\ref{bound 2}) combine to give
\begin{equation*}
\left|h_u\left(t\right)\right|\leq d_1e^{-t\left(\beta\left(q\right)-\gamma\left(q\right)\right)},
\end{equation*}
for some positive constant $d_1$ (which depends on $q$). As $\beta\left(q\right)>\gamma\left(q\right)$ by Lemma \ref{g<b} below, putting $d_2=\beta\left(q\right)-\gamma\left(q\right)>0$ completes the proof of statement (c$^\prime$).

\medskip

It remains then to prove inequalities (\ref{bound 1}) and (\ref{bound 2}), which we do in Lemma \ref{OSC1} that follows and Theorem \ref{theorem 2} which is proved in Sections \ref{four} and \ref{five}. 

\begin{lem}
\label{2GlemA}
Let $\bM$ be a non-negative irreducible $n\times n$ matrix with $\rho(\bM)=1$. Suppose $\bv = (v_1, v_2,  \ldots ,  v_n)^\transpose$ is a positive vector such that
\begin{equation*}
\mathbf{0}< \bv \leq \bM \bv,
\end{equation*}
then
\begin{equation*}
 \bv = \bM \bv.
\end{equation*}
\end{lem}
\begin{proof}This follows from standard Perron-Frobenius theory, see \cite{Book_Seneta} or \cite[Lemma 3.2.1]{phdthesis_Boore}.
\end{proof}

\begin{lem}
\label{g<b} 
Let $\bigl(V,E^*,i,t,r,p,((\mathbb{R}^m,\left| \ \  \right|))_{u \in V},(S_e)_{e \in E^1}\bigr)$ be an $n$-vertex IFS with probabilities which satisfies the OSC. Suppose that the non-negative matrix $\bB\left(q, \gamma, l\right)$, as defined in Subsection \ref{matrix B}, is irreducible. Let $q \in \mathbb{R}$ and let $\beta\left(q\right)$,$\gamma\left(q\right)\in \mathbb{R}$ be the unique numbers for which $\rho\left(\mathbf{A}\left(q, \beta\right)\right) = \rho\left(\bB\left(q, \gamma, l\right)\right)= 1$, as described in Subsections \ref{matrix A} and \ref{matrix B}. 

Then
\begin{equation*}
\gamma\left(q\right) < \beta\left(q\right).
\end{equation*}
\end{lem}
\begin{proof}
We can replace $\beta\left(q\right)$ by $\gamma\left(q\right)$ in the definition of $\bA(q,\beta,l)$ in Equation \eqref{Auv(q,beta,l)}, so that the $uv$th entry of the matrix $\bA(q,\gamma,l)$ is 
\begin{equation*}
A_{uv}(q,\gamma,l)  =\sum_{\substack{\be \in E_{uv}^l }}p_{\be}^q r_\be^\gamma.  
\end{equation*} 
The uvth entry of $\mathbf{B}\left(q,\gamma,l\right)$, as defined in Equation \eqref{Buv(q,gamma,l)}, means that along each row $u$ of the matrices $\mathbf{B}\left(q,\gamma,l\right)$ and $\mathbf{A}(q,\gamma,l)$
\begin{equation*}
B_{uv}\left(q,\gamma,l\right)=\sum_{\substack{\be \in E_{uv}^l \\ \be \neq \bl_u}}p_{\be}^qr_{\be}^\gamma=\sum_{\substack{\be \in E_{uv}^l }}p_\be^qr_\be^\gamma=A_{uv}\left(q,\gamma,l\right),
\end{equation*}
except at the single entry with $v=t(\bl_u)$ where there is a strict inequality
\begin{equation*}
B_{ut(\bl_u)}\left(q,\gamma,l\right)=\sum_{\substack{\be \in E_{ut(\bl_u)}^l \\ \be \neq \bl_u}}p_{\be}^qr_{\be}^\gamma<\sum_{\substack{\be \in E_{ut(\bl_u)}^l }}p_\be^qr_\be^\gamma=A_{ut(\bl_u)}\left(q,\gamma,l\right).
\end{equation*}
It follows that
\begin{equation}
\label{Bnot=A}
\mathbf{B}\left(q,\gamma,l\right)\neq \mathbf{A}\left(q,\gamma,l\right).
\end{equation}
and
\begin{equation}
\label{B<=A}
\mathbf{B}\left(q,\gamma,l\right)\leq \mathbf{A}\left(q,\gamma,l\right).
\end{equation}
Let $M_m$ denote the set of real $m \times m$ matrices. For $\bM \in M_m$, $\rho(\bM)$ is the spectral radius of $\bM$ as defined in Equation (\ref{spectral radius eqn}). The \emph{spectral radius formula} states that
\begin{equation}
\label{Spectral radius formula}
\rho\left(\bM\right)=\lim_{k \to \infty} \left\|\bM^k \right\|^{\frac{1}{k}},
\end{equation}
see \cite{Book_Maddox,Book_Seneta}. All norms are equivalent on finite dimensional spaces so it doesn't matter which norm we use, but for our purposes it is convenient to give a specific norm, defined for a matrix $\bM \in M_m$, with $ij$th entry $M_{ij}\in \mathbb{R}$, by
\begin{equation}
\label{matrix norm}
\left\|\bM \right\|=\max\biggl\{\sum_{j=1}^m \left|M_{ij}\right|: 1\leq i \leq m\biggr\}.
\end{equation}
We now prove statements (a), (b), and (c) that follow.

\medskip

(a) \emph{For non-negative matrices $\bC,\bD$, if $\bC \leq \bD$ then $\rho\left(\bC\right)\leq \rho\left(\bD\right)$}.

Since $\mathbf{0}\leq \bC \leq \bD$, it follows that $\mathbf{0}\leq \bC^n \leq \bD^n$, for any $n \in \mathbb{N}$, and from the specific definition of the norm in (\ref{matrix norm}) this means that $0 \leq \left\|\bC^n\right\| \leq \left\|\bD^n\right\|$. Equation (\ref{Spectral radius formula}) now implies $\rho\left(\bC\right)\leq \rho\left(\bD\right)$.

\medskip

(b) \emph{For $k \in \mathbb{N}$ and a non-negative matrix $\bC$, $\rho(\bC^k)= \rho(\bC)^k$}.

Let $n \in \mathbb{N}$. If $\lambda$ is an eigenvalue of $\bC$ then $\lambda^n$ is an eigenvalue of $\bC^n$. It follows from the definition of the spectral radius in Equation (\ref{spectral radius eqn}) that $\rho(\bC^n)\geq \rho(\bC)^n$. The norm defined in Equation (\ref{matrix norm}) is submultiplicative, see \cite{Book_Maddox}, so $\left\|\bC^n \right\|\leq \left\|\bC \right\|^n$. It follows that $\left\|(\bC^n)^k \right\|^{\frac{1}{k}}\leq (\left\|\bC^k \right\|^{\frac{1}{k}})^n$ for any $k \in \mathbb{N}$ and Equation \ref{Spectral radius formula} implies $\rho(\bC^n) \leq \rho(\bC)^n$.

\medskip

(c) \emph{For non-negative matrices $\bC,\bD\in M_m$, if $\rho\left(\bC\right)=\rho\left(\bD\right)=1$, and $\bC\leq \bD$, then $\bC=\bD$}.

By the Perron-Frobenius Theorem, $\bC$ has a unique (up to scaling) positive eigenvector $\bc$, with eigenvalue $\rho\left(\bC\right)=1$, see Subsection \ref{spectral radius}. It follows that
\begin{equation}
\label{DC}
\mathbf{0}\leq \left(\bD-\bC\right)\bc=\bD\bc - \bc,
\end{equation}
which means
\begin{equation*}
\bc \leq \bD\bc.
\end{equation*}
Applying Lemma \ref{2GlemA} gives $\bc = \bD\bc$, and so by Equation (\ref{DC}),
\begin{equation*}
\left(\bD-\bC\right)\bc=\bD\bc - \bc=\mathbf{0},
\end{equation*}
and $\bD=\bC$.

\medskip

Inequality (\ref{B<=A}), together with statements (a) and (b), means
\begin{equation*}
1   =   \rho\left(\mathbf{B}\left(q,\gamma,l\right)\right)  \leq \rho\left(\mathbf{A}(q,\gamma,l)\right) = \left(\rho\left(\mathbf{A}\left(q,\gamma,1\right)\right)\right)^l,
\end{equation*}
and so $\rho\left(\mathbf{A}\left(q,\gamma,1\right)\right)\geq 1$. This implies that $\gamma \leq \beta$ because $\rho\left(\mathbf{A}\left(q,\gamma,1\right)\right)$ (strictly) decreases as $\gamma$ increases and $\rho\left(\mathbf{A}\left(q,\beta,1\right)\right)=1$, (see \cite[Section 3]{Paper_Edgar_Mauldin} for the details). If $\gamma= \beta$ then 
\begin{equation*}
1  \ =  \ \rho\left(\mathbf{B}\left(q,\gamma,l\right)\right) \ = \rho\left(\mathbf{A}\left(q,\gamma,l\right)\right).
\end{equation*}
and so by statement (c),
\begin{equation*}
\mathbf{B}\left(q,\gamma,l\right)  = \mathbf{A}\left(q,\gamma,l\right).
\end{equation*}
This contradicts Equation (\ref{Bnot=A}). Therefore $\gamma < \beta$.
\end{proof}

We will use the following notation, illustrated in Figure \ref{Ch4a}, in the preliminary Lemmas \ref{OSCa} and \ref{OSCb} that lead up to the important Lemma \ref{OSC1}. As in Section \ref{thm1.1_SSC}, $D$ indicates an $r$-separated subset of $F_u$ and for each edge $e \in E_u^1$, we will use $D_e$ to indicate an $r$-separated subset of $S_e(F_{t(e)})\subset F_u$. 

Given an $r$-separated subset $D_e$ of $S_e(F_{t(e)})$, for each $f \in E_u^1$, with $f \neq e$, let  
\begin{equation}
\label{H_{e,f}}
H_{e,f} = D_e \cap S_f(F_{t(f)})(r), 
\end{equation}
where $S_f(F_{t(f)})(r)$ is the closed $r$-neighbourhood of $S_f(F_{t(f)})$, so that $H_{e,f}$ is an $r$-separated subset of $S_e(F_{t(e)})\cap S_f(F_{t(f)})(r)$. 

Let
\begin{equation}
\label{H_e}
H_e = \bigcup_{\substack{f \in E_u^1 \\ f \neq e}}H_{e,f},
\end{equation}
where this union is not necessarily disjoint.

\begin{figure}[!htb]
\begin{center}
\includegraphics[trim = 30mm 175mm 30mm 15mm, clip, scale=0.7]{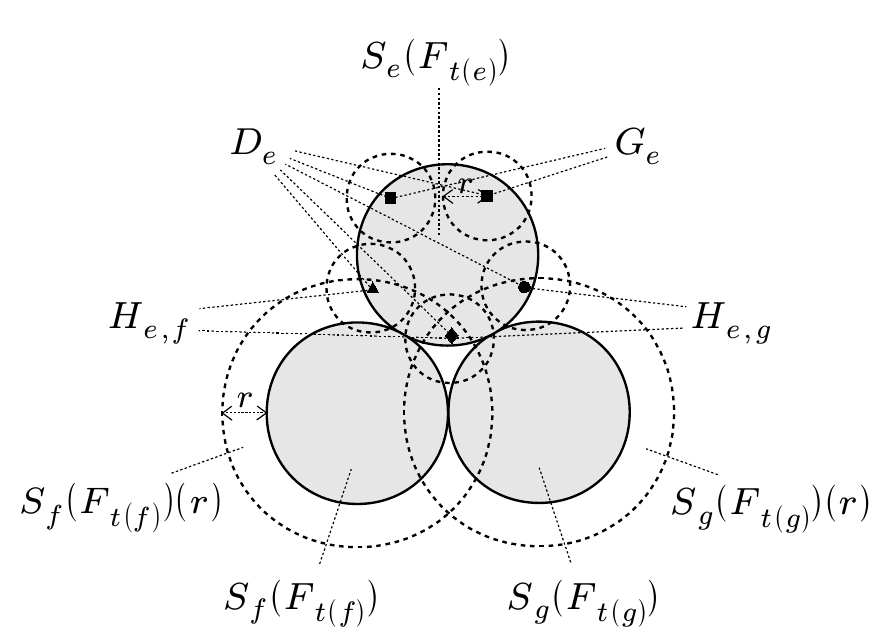}
\end{center}
\caption{A schematic representation in $\mathbb{R}^2$ of $r$-separated subsets $D_e$, $G_e$, $H_{e,f}$ and $H_{e,g}$.}
\label{Ch4a}
\end{figure}

Let 
\begin{equation}
\label{G_e}
G_e = D_e \setminus \bigcup_{\substack{ f \in E_u^1 \\ f \neq e}}S_f(F_{t(f)})(r)=D_e \setminus H_e, 
\end{equation}
so that
\begin{equation}
\label{D_e = G_e}
D_e = G_e \cup H_e,
\end{equation}
and this union is disjoint.

Let
\begin{equation}
\label{G}
G=\bigcup_{\substack {e \in E_u^1}}G_e, 
\end{equation}
then $G$ is always an $r$-separated subset of $F_u$, and this union is disjoint.

Figure \ref{Ch4a} illustrates these definitions schematically, the small squares are points in $G_e$ and $H_e=H_{e,f}\cup H_{e,g}$ consists of the three points represented by the triangle, diamond and circle, the union is not disjoint here because the diamond belongs to both $H_{e,f}$ and $H_{e,g}$. However it is clear from the diagram that $D_e=G_e\cup H_e$ is a disjoint union.
 
\begin{lem}
\label{OSCa}
Let $q \in \mathbb{R}$ and $r > 0$. Then
\begin{align*}
&\textup{(a) }M_u^q\left(F_u,r\right)  \leq  \sum_{  e\in E_u^1 } M_u^q\left(S_e(F_{t(e)}),r\right). \\
&\textup{(b) }\textrm{For each $e \in E_u^1$}, \ M_u^q\left(S_e(F_{t(e)}),r\right)  \leq p_e^qM_{t(e)}^q\left(F_{t(e)}, r_e^{-1}r\right) + \sum_{ \substack{f\in E_u^1 \\ f\neq e } } Q_{e,f}^q\left(r\right). \\
&\textup{(c) }M_u^q\left(F_u,r\right) - \sum_{ e\in E_u^1 }p_e^qM_{t(e)}^q\left(F_{t(e)}, r_e^{-1}r\right) \leq \sum_{ e\in E_u^1 }\biggl( \  \sum_{ \substack{ f\in E_u^1 \\ f\neq e } } Q_{e,f}^q\left(r\right) \ \biggr).
\end{align*}			
\end{lem}

\begin{proof}Part (c) is an immediate consequence of parts (a) and (b).

(a) Let $D$ be an $r$-separated subset of $F_u$ then, by Equation (\ref{Invariance 1}), 
\begin{equation*}
D = \bigcup_{ \substack{e\in E_u^1}  } \left(D \cap S_e(F_{t(e)})\right)= \bigcup_{ \substack{e\in E_u^1}  } D_e, 
\end{equation*}
where this union is not necessarily disjoint and each $D_e$ is an $r$-separated subset of $S_e(F_{t(e)})$. It follows that
\begin{equation*}
\sum_{x \in D} \mu_u\left(B(x,r)\right)^q  \leq  \sum_{ e\in E_u^1 } \biggl(  \ \sum_{x \in D_e}\mu_u\left(B(x,r)\right)^q  \ \biggr)  \leq  \sum_{  e\in E_u^1 }M_u^q\left(S_e(F_{t(e)}), r\right), 
\end{equation*}
where the second inequality follows directly from the definition of $M_u^q(S_e(F_{t(e)}), r)$. Taking the supremum over all $r$-separated subsets $D$ of $F_u$ gives the required result, which should be compared with Lemma \ref{SSCa}(a).


(b) Let $D_e$ be an $r$-separated subset of $S_e(F_{t(e)})$, then from Equations (\ref{H_{e,f}}), (\ref{H_e}), (\ref{G_e}), (\ref{D_e = G_e}), and (\ref{G}), it follows that,
\begin{align*}
\sum_{x \in D_e} \mu_u&\left(B(x,r)\right)^q  \\
&= \sum_{x \in G_e} \mu_u\left(B(x,r)\right)^q  +  \sum_{x \in H_e} \mu_u\left(B(x,r)\right)^q  \\
&\leq  \sum_{x \in G_e} \mu_u\left(B(x,r)\right)^q + \sum_{\substack{f \in E_u^1 \\ f \neq e}}\biggl( \ \sum_{x \in H_{e,f}}\mu_u\left(B(x,r)\right)^q \  \biggr) \\
&=\sum_{x \in G_e} \biggl( \ \sum_{  f\in E_u^1  }p_f \mu_{t(f)}\left(S_f^{-1}(B(x,r))\right) \ \biggr)^q + \sum_{\substack{f \in E_u^1 \\ f \neq e}} \biggl( \ \sum_{x \in H_{e,f}}\mu_u\left(B(x,r)\right)^q \  \biggr) \\
&=\sum_{x \in G_e} p_e^q\mu_{t(e)}\left(S_e^{-1}(B(x,r))\right)^q + \sum_{\substack{f \in E_u^1 \\ f \neq e}} \biggl( \ \sum_{x \in H_{e,f}}\mu_u\left(B(x,r)\right)^q  \ \biggr).   
\end{align*}   
Here we have the first inequality because the union may not be disjoint in Equation (\ref{H_e}). We then apply Equation (\ref{Invariance 2}), and the last equality follows because $\mu_{t(f)}\left(S_f^{-1}(B(x,r))\right) = 0$ for $ f \in E_u^1$ with $f\neq e$. Specifically, any point $x \in G_e$, is at least a distance $r$ from $S_f(F_{t(f)})$ for $f \neq e$, from the definition of $G_e$ in Equation (\ref{G_e}). For such $x$, $S_f^{-1}(B(x,r))\cap F_{t(f)}=\emptyset$, and as $\supp\mu_{t(f)}=F_{t(f)}$, this implies $\mu_{t(f)}\left(S_f^{-1}(B(x,r))\right)=0$. 

We note that as $H_{e,f}$ is an $r$-separated subset of $S_e(F_{t(e)})\cap S_f(F_{t(f)})(r)$,
\begin{equation*}
\sum_{x \in H_{e,f}}\mu_u\left(B(x,r)\right)^q \leq Q_{e,f}^q\left(r\right),
\end{equation*}
and as $S_e^{-1}\left(B(x, r)\right)=B\left(S_e^{-1}(x),r_e^{-1}r\right)$ we obtain
\begin{align*}
\sum_{x \in D_e} \mu_u\left(B(x,r)\right)^q &\leq \ p_e^q\sum_{x \in G_e} \mu_{t(e)}\left(B(S_e^{-1}(x),r_e^{-1}r)\right)^q  +  \sum_{\substack{f \in E_u^1 \\ f \neq e}}Q_{e,f}^q\left(r\right) \\
&= \ p_e^q\sum_{x \in S_e^{-1}(G_e)} \mu_{t(e)}\left(B(x,r_e^{-1}r)\right)^q  +  \sum_{\substack{f \in E_u^1 \\ f \neq e}}Q_{e,f}^q\left(r\right) \\
& \leq \ p_e^qM_{t(e)}^q\left(F_{t(e)}, r_e^{-1}r\right) +  \sum_{\substack{f \in E_u^1 \\ f \neq e}}Q_{e,f}^q\left(r\right), 
\end{align*}
where this last inequality holds since $S_e^{-1}(G_e)$ is an $r_e^{-1}r$-separated subset of $F_{t(e)}$. As $D_e$ is any $r$-separated subset of $S_e(F_{t(e)})$, this proves part (b).
\end {proof}
We need $q \geq 0$ in Lemma \ref{OSCb}(b) that follows and this is the reason we need $q \geq 0$ in the statement of Theorem \ref{theorem 1} (ii).
\begin{lem}
\label{OSCb}
Let $r > 0$.  
\begin{align*}
&\textup{(a) } \textrm{For $q \in \mathbb{R}$, } - \sum_{ e\in E_u^1 } \biggl( \ \sum_{ \substack{ f\in E_u^1 \\ f\neq e } } Q_{e,f}^q\left(r\right) \ \biggr)  \leq  M_u^q\left(F_u,r\right) - \sum_{ e\in E_u^1 } M_u^q\left(S_e(F_{t(e)}),r\right).\\
&\textup{(b) } \textrm{For $q \geq 0$, and each $e \in E_u^1$, } M_u^q\left(S_e(F_{t(e)}),r\right) \geq p_e^qM_{t(e)}^q\left(F_{t(e)}, r_e^{-1}r\right).\\
&\textup{(c) } \textrm{For $q \geq 0$, } - \sum_{ e\in E_u^1 }\biggl( \ \sum_{ \substack{ f\in E_u^1 \\ f\neq e } } Q_{e,f}^q\left(r\right) \ \biggr) \leq M_u^q\left(F_u,r\right) - \sum_{ e\in E_u^1 } p_e^qM_{t(e)}^q\left(F_{t(e)}, r_e^{-1}r\right).
\end{align*}		
\end{lem}

\begin{proof}Part (c) is an immediate consequence of parts (a) and (b).

(a) Let $D$ be an $r$-separated subset of $F_u$ and for each $e \in E_u^1$, let 
\begin{equation*}
D_e=D \cap S_e(F_{t(e)}),
\end{equation*}
so that $D_e$ is an $r$-separated subset of $S_e(F_{t(e)})$. From Equations (\ref{H_{e,f}}), (\ref{H_e}), (\ref{G_e}), (\ref{D_e = G_e}), and (\ref{G}), it follows that,
\begin{align*}
M_u^q\left(F_u, r\right)&\geq  \sum_{x \in G } \mu_u\left(B(x,r)\right)^q   \\
&= \sum_{e \in  E_u^1}\biggl( \  \sum_{x \in G_e}  \mu_u\left(B(x,r)\right)^q  \ \biggr) \\
&= \sum_{e \in  E_u^1}\biggl( \ \sum_{  x \in  D_e \setminus H_e }   \mu_u\left(B(x,r)\right)^q  \ \biggr) \\
& =  \sum_{e \in E_u^1} \biggl( \ \sum_{x \in D_e}\mu_u\left(B(x,r)\right)^q  -   \sum_{  x \in H_e} \mu_u\left(B(x,r)\right)^q  \ \biggr) \\
& \geq  \sum_{e \in E_u^1}\biggl( \ \sum_{x \in D_e}\mu_u\left(B(x,r)\right)^q  -  \sum_{\substack{f \in E_u^1 \\ f \neq e}}\biggl( \ \sum_{x \in H_{e,f}}\mu_u\left(B(x,r)\right)^q \ \biggr) \ \biggr),  
\end{align*}  
where the last line is an inequality because the union in Equation (\ref{H_e}) is not necessarily disjoint. Finally it follows from the definition of $Q_{e,f}^q\left(r\right)$ and $M_u^q\left(S_e(F_{t(e)}), r\right)$ in  Subsection \ref{The $q$th packing moment} that  
\begin{equation*}
M_u^q\left(F_u, r\right) \geq \sum_{e \in E_u^1}\biggl( \ M_u^q\left(S_e(F_{t(e)}), r\right)  -  \sum_{\substack{f \in E_u^1 \\ f \neq e}}Q_{e,f}^q\left(r\right) \ \biggr), 
\end{equation*}
which proves part (a).

\medskip

(b) Let D be an $r_e^{-1}r$-separated subset of $F_{t(e)}$ and let $D_e=S_e(D)$ so that as usual $D_e$ is an $r$-separated subset of $S_e(F_{t(e)})$. 

From Equations (\ref{H_{e,f}}), (\ref{H_e}), (\ref{G_e}), (\ref{D_e = G_e}), (\ref{G}), and using the same arguments given in detail in the proof of Lemma \ref{OSCa}(b), we obtain,
\begin{align*}
M_u^q\left(S_e(F_{t(e)}),r\right) &\geq  \sum_{x \in D_e} \mu_u\left(B(x,r)\right)^q \\
&= \sum_{x \in G_e} \mu_u(B(x,r))^q +  \sum_{x \in H_e}\mu_u\left(B(x,r)\right)^q \\
&= \sum_{x \in G_e} \biggl( \ \sum_{  f\in E_u^1  }p_f \mu_{t(f)}\left(S_f^{-1}(B(x,r))\right) \ \biggr)^q +  \sum_{x \in H_e }\mu_u\left(B(x,r)\right)^q \\
&= \sum_{x \in G_e} p_e^q \mu_{t(e)}\left(S_e^{-1}(B(x,r))\right)^q +  \sum_{x \in H_e }\mu_u\left(B(x,r)\right)^q.
\end{align*}

By Equation (\ref{Invariance 2}), if $q\geq 0$, then the following inequality must hold,
\begin{equation*}
\mu_u\left(B(x,r)\right)^q = \biggl( \ \sum_{  f\in E_u^1  }p_f \mu_{t(f)}\left(S_f^{-1}(B(x,r))\right) \ \biggr)^q \geq p_e^q \mu_{t(e)}\left(S_e^{-1}(B(x,r))\right)^q, 
\end{equation*}
that is,
\begin{equation}
\label{q ineq}
\mu_u\left(B(x,r)\right)^q - p_e^q \mu_{t(e)}\left(S_e^{-1}(B(x,r))\right)^q \geq 0.
\end{equation}

Using this inequality gives
\begin{align*}
M&_u^q\left(S_e(F_{t(e)}),r\right) && \\
&\geq \sum_{x \in G_e \cup H_e} p_e^q \mu_{t(e)}\left(S_e^{-1}(B(x,r))\right)^q \\
& \quad \quad \quad \quad + \sum_{x \in H_e }\left( \mu_u\left(B(x,r)\right)^q - p_e^q \mu_{t(e)}\left(S_e^{-1}(B(x,r))\right)^q  \right) \\
&\geq \sum_{x \in G_e\cup H_e} p_e^q \mu_{t(e)}\left(S_e^{-1}(B(x,r))\right)^q &&(\textrm{by Inequality (\ref{q ineq})})\\
&=  p_e^q\sum_{x \in D_e} \mu_{t(e)}\left(S_e^{-1}(B(x,r))\right)^q &&\\
&=  p_e^q  \sum_{x \in S_e^{-1}(D_e)} \mu_{t(e)}\left(B(x,r_e^{-1}r)\right)^q &&\\
&=  p_e^q  \sum_{x \in D} \mu_{t(e)}\left(B(x,r_e^{-1}r)\right)^q. &&
\end{align*}
As $D$ is any $r_e^{-1}r$-separated subset of $F_{t(e)}$ this completes the proof of part (b). Compare this with Lemma \ref{SSCa}(b). 
\end{proof}

\begin{lem}
\label{OSCc}
Let $q \geq  0$ and $r>0$, then   
\begin{equation*}
\biggl|M_u^q\left(F_u,r\right) - \sum_{ e\in E_u^1 }p_e^qM_{t(e)}^q\left(F_{t(e)}, r_e^{-1}r\right)\biggr| \leq \sum_{ e\in E_u^1 } \biggl( \ \sum_{ \substack{ f\in E_u^1 \\ f \neq e } }Q_{e,f}^q\left(r\right) \ \biggr).
\end{equation*}
\end{lem}
\begin{proof} This is established by Lemma \ref{OSCa}(c) and Lemma \ref{OSCb}(c). 
\end{proof}

Our next lemma proves Inequality (\ref{bound 1}).
\begin{lem}
\label{OSC1}
Let $q \geq  0$ and $t \in [0,+\infty)$, then for the functions $(h_u)_{u \in V}$, as defined in Equation (\ref{h_u}),   
\begin{equation*}
\left|h_u\left(t\right)\right| \leq e^{-t\beta\left(q\right)}\sum_{ e\in E_u^1 } \biggl( \ \sum_{ \substack{ f\in E_u^1 \\ f \neq e } }Q_{e,f}^q(e^{-t}) \ \biggr).
\end{equation*}
\end{lem}

\begin{proof}
As $t \in [0,+\infty)$ if and only if $e^{-t}\in (0,1]$, Equation (\ref{h_u}) together with Lemma \ref{OSCc} imply
\begin{align*}
\left|h_u\left(t\right)\right|&=e^{-t\beta\left(q\right)}\biggl|M_u^q\left(F_u,e^{-t}\right) - \sum_{e\in E_u^1} p_e^q M_{t(e)}^q\left(F_{t(e)}, r_e^{-1}e^{-t}\right) \biggr| \\
&\leq e^{-t\beta\left(q\right)}\sum_{ e\in E_u^1 } \biggl( \ \sum_{ \substack{ f\in E_u^1 \\ f \neq e } }Q_{e,f}^q(e^{-t}) \ \biggr).
\qedhere
\end{align*}
\end{proof}

\section{ Proof of Theorem \ref{theorem 2} - preliminary lemmas} \label{four}
We have broken the proof of Theorem \ref{theorem 2}  down into a sequence of small steps in the form of a sequence of lemmas that now follow. In this section our aim is to prove Lemmas \ref{OSC6} and \ref{OSC9} which are  needed for the final step in the proof given in Section \ref{five}. 
\begin{lem}
\label{OSCintersection}
Let $\bigl(V,E^*,i,t,r,((\mathbb{R}^m,\left| \ \  \right|))_{v \in V},(S_e)_{e \in E^1}\bigr)$ be an $n$-vertex IFS which satisfies the OSC, let $\left(U_v\right)_{v \in V}$ be the non-empty bounded open sets of the OSC, and let $\Bf,\bg \in E_u^*$ be any two paths which are not subpaths of each other, that is $\Bf \not\subset \bg$ and $\bg \not\subset \Bf$.

Then
\begin{equation*}
\textup{(a) } \quad S_{\Bf}(U_{t(\Bf)})\cap S_{\bg}(U_{t(\bg)})= \emptyset,
\end{equation*}
\begin{equation*}
\textup{(b) } \quad S_{\Bf}(F_{t(\Bf)})\cap S_{\bg}(U_{t(\bg)})= \emptyset.
\end{equation*}
\end{lem}
\begin{proof}See \cite[Lemma 4.7.6]{phdthesis_Boore}. 
\end{proof}

\begin{lem}
\label{OSCe}
Let $\bigl(V,E^*,i,t,r,p,((\mathbb{R}^m,\left| \ \  \right|))_{v \in V},(S_e)_{e \in E^1}\bigr)$ be an $n$-vertex IFS with probabilities which satisfies the OSC. Then for each vertex $ v \in V$ and all finite paths $\be \in E_v^*$
\begin{equation*}
\mu_v\left(S_{\be}(F_{t(\be)})\right)=p_\be.
\end{equation*}
\end{lem}
\begin{proof}
A proof for the $1$-vertex case is given in \cite{Paper_Graf}. 

For any $k \in \mathbb{N}$, Equation (\ref{Invariance 2}) may be iterated $k$ times to obtain
\begin{equation}
\label{mp0}
\mu_v(A_v)= \sum_{ \substack{ \Bf\in E_v^k}  }p_{\Bf} \mu_{t(\Bf)}\left(S_{\Bf}^{-1}(A_v)\right).
\end{equation}

The first step in the proof is to prove that for each $v \in V$, $\mu_v(F_v\setminus U_v)=0$, where $(U_v)_{v \in V}$ are the open sets of the OSC/SOSC. As described in Subsection \ref{matrix B} there exists a family of paths $(\bl_v)_{v \in V}$, all of the same length $l=\left|\bl_v\right|$, where $l\in \mathbb{N}$ may be chosen as large as we like, such that for each $v \in V$, $S_{\bl_v}(F_{t(\bl_v)})\subset U_v$, see Equation \eqref{l_u}.

It is clear that $\mu_v(U_v)>0$, for each $v \in V$, since by Equations (\ref{mp0}) and \eqref{l_u} we obtain
\begin{gather*}
0<p_{\bl_v}=p_{\bl_v} \mu_{t(\bl_v)}(F_{t(\bl_v)}) \\
 \leq \sum_{ \substack{ \Bf\in E_v^l}  }p_{\Bf} \mu_{t(\Bf)}\left(S_{\Bf}^{-1}(S_{\bl_v}(F_{t(\bl_v)}))\right)=\mu_v\left(S_{\bl_v}(F_{t(\bl_v)})\right) \leq \mu_v(U_v).
\end{gather*}   
We may now choose a vertex $u$ such that the open set $U_u$ is of minimal measure with
\begin{equation}
\label{mp3}
\mu_u(U_u)\leq \mu_v(U_v)
\end{equation}
for all $v \in V$. From the definition of the OSC,
\begin{equation}
\label{mp4}
\bigcup_{\Bf \in E_u^l}S_{\Bf}(U_{t(\Bf)})\subset U_u,
\end{equation} 
where the union on the left hand side is disjoint by Lemma \ref{OSCintersection}(a). This implies
\begin{align*}
\mu_u\left(U_u\right)&\geq \mu_u\biggl( \ \bigcup_{\Bf \in E_u^l}S_{\Bf}(U_{t(\Bf)}) \ \biggr) && \\
&= \sum_{\Bf \in E_u^l} \mu_u\left(S_{\Bf}(U_{t(\Bf)})\right) && \textrm{(the union is disjoint)}\\
&= \sum_{\Bf \in E_u^l} \biggl( \ \sum_{ \substack{ \bg\in E_u^l}  }p_{\bg} \mu_{t(\bg)}\left(S_{\bg}^{-1}(S_{\Bf}(U_{t(\Bf)}))\right) \ \biggr) && \textrm{(by Equation (\ref{mp0}))}\\
&\geq \sum_{\Bf \in E_u^l} p_{\Bf} \mu_{t(\Bf)}(U_{t(\Bf)}) && \\
&\geq \mu_{u}(U_u)\sum_{\Bf \in E_u^l} p_{\Bf}  && \textrm{(by Inequality (\ref{mp3}))} \\
&= \mu_{u}(U_u)  && \textrm{(by Equation (\ref{probability function}))}, 
\end{align*}
and so
\begin{equation}
\label{mp5}
\mu_u(U_u)= \mu_u\biggl( \ \bigcup_{\Bf \in E_u^l}S_{\Bf}(U_{t(\Bf)}) \ \biggr).
\end{equation}
Now 
\begin{align*}
\mu_u \bigl( S_{\bl_u}(F_{t(\bl_u)})&\setminus S_{\bl_u}(U_{t(\bl_u)}) \bigr) && \\
&=\mu_u\biggl(S_{\bl_u}(F_{t(\bl_u)})\setminus \biggl( \ \bigcup_{\Bf \in E_u^l}S_{\Bf}(U_{t(\Bf)}) \ \biggr)\biggr) && \textrm{(by Lemma \ref{OSCintersection}(b))} \\
&\leq \mu_u\biggl(U_u\setminus \biggl( \ \bigcup_{\Bf \in E_u^l}S_{\Bf}(U_{t(\Bf)}) \ \biggr)\biggr) && \textrm{(by (\ref{l_u}))} \\
&= \mu_u(U_u)-\mu_u \biggl( \ \bigcup_{\Bf \in E_u^l}S_{\Bf}(U_{t(\Bf)}) \ \biggr) && \textrm{(by (\ref{mp4}))} \\
&= 0 && \textrm{(by (\ref{mp5}))}.
\end{align*} 
That is 
\begin{align*}
0 &= \mu_u\bigl(S_{\bl_u}(F_{t(\bl_u)})\setminus S_{\bl_u}(U_{t(\bl_u)})\bigr) && \\
&= \mu_u\bigl(S_{\bl_u}(F_{t(\bl_u)}\setminus U_{t(\bl_u)})\bigr) && \\
&= \sum_{ \substack{ \Bf\in E_u^l}  }p_{\Bf} \mu_{t(\Bf)}\left(S_{\Bf}^{-1}(S_{\bl_u}(F_{t(\bl_u)}\setminus U_{t(\bl_u)}))\right) && \textrm{ (by (\ref{mp0}))}\\
&\geq p_{\bl_u} \mu_{t(\bl_u)}\left(F_{t(\bl_u)}\setminus U_{t(\bl_u)}\right). && 
\end{align*}
This proves $\mu_{t(\bl_u)}\left(F_{t(\bl_u)}\setminus U_{t(\bl_u)}\right)=0$. The path $\bl_u$ can always be extended to a path $\bl_u^{\prime} \in E_u^{l^{\prime}}$, for some $l^{\prime}>l$, with $t(\bl_u^{\prime})=v$ for any vertex $v \in V$. This is because the graph is strongly connected, so a path $\bg$ can always be found with $i(\bg)=t(\bl_u)$ and $t(\bg)=v$, for any vertex $v \in V$. So we may put $\bl_u^{\prime}=\bl_u\bg$ and Equation (\ref{l_u}) becomes 
\begin{equation*}
S_{\bl_u^{\prime}}(F_{t(\bl_u^{\prime})})\subset U_u.
\end{equation*}
Repeating the argument above with $\bl_u$ replaced by $\bl_u^{\prime}$ and  $l$ replaced by $l^{\prime}$, proves that $\mu_{t(\bl_u^{\prime})}\left(F_{t(\bl_u^{\prime})}\setminus U_{t(\bl_u^{\prime})}\right)=0$, and as $t(\bl_u^{\prime})=v$, this proves that for each $v \in V$,
\begin{equation}
\label{mp6}
\mu_v\left(F_v \setminus U_v\right)=0.
\end{equation} 

For any vertex $v \in V$, and any path $\be \in E_v^*$,  $\be \in E_v^k$ where $k = \left|\be\right|$. Let $\Bf,\bg \in E_v^k$, $\Bf \neq \bg$. By Lemma \ref{OSCintersection}(b), $S_{\Bf}(U_{t(\Bf)})\cap S_{\bg}(F_{t(\bg)})= \emptyset$, so that
\begin{equation}
\label{mp7}
S_{\Bf}^{-1}\left(S_{\Bf}(U_{t(\Bf)})\cap S_{\bg}(F_{t(\bg)})\right)=U_{t(\Bf)}\cap S_{\Bf}^{-1}\left(S_{\bg}(F_{t(\bg)})\right)=\emptyset.
\end{equation}
This means that, since $\supp \mu_{t(\Bf)} = F_{t(\Bf)}$,
\begin{gather*}
\mu_{t(\Bf)}\left(S_{\Bf}^{-1}\left(S_{\bg}(F_{t(\bg)})\right)\right)=\mu_{t(\Bf)}\left(F_{t(\Bf)}\cap S_{\Bf}^{-1}\left(S_{\bg}(F_{t(\bg)})\right)\right) \\
=\mu_{t(\Bf)}\left(U_{t(\Bf)}\cap S_{\Bf}^{-1}\left(S_{\bg}(F_{t(\bg)})\right)\right)+\mu_{t(\Bf)}\left(\left(F_{t(\Bf)}\setminus U_{t(\Bf)}\right)\cap S_{\Bf}^{-1}\left(S_{\bg}(F_{t(\bg)})\right)\right)  =0,
\end{gather*}
by (\ref{mp7}) and (\ref{mp6}). That is for $\Bf,\bg \in E_v^k$, $\Bf \neq \bg$,
\begin{equation}
\label{mp8}
\mu_{t(\Bf)}\left(S_{\Bf}^{-1}\left(S_{\bg}(F_{t(\bg)})\right)\right)=0.
\end{equation}
For any vertex $v \in V$, and any path $\be \in E_v^*$,  $\be \in E_v^k$ where $k = \left|\be\right|$, and by  Equation (\ref{mp0}) we obtain
\begin{align*}
\mu_v\left(S_{\be}(F_{t(\be)})\right)&= \sum_{ \substack{ \Bf\in E_v^k}  }p_{\Bf} \mu_{t(\Bf)}\left(S_{\Bf}^{-1}(S_{\be}(F_{t(\be)}))\right) && \\
&=p_{\be} \mu_{t(\be)}\left(F_{t(\be)}\right) && \textrm{(by (\ref{mp8}))} \\
&=p_{\be} && \textrm{(as $\mu_{t(\be)}(F_{t(\be)})=1$)}.
\qedhere
\end{align*}
\end{proof}

The statement of the next lemma is a simple adaptation of a lemma in \cite[Subsection 5.3]{Paper_Hutchinson}, and \cite[Lemma 9.2]{Book_KJF2}. 
\begin{lem}
\label{OSCf}
Let $r, c_1, c_2 > 0$, and let $\left\{V_i\right\}$ be subsets of $\mathbb{R}^m$. Suppose each set $V_i$ contains a closed ball $B_i$ of radius $c_1r$ and is contained in a closed ball of radius $c_2r$, and that $\left\{B_i\right\}$ is a disjoint set. Then for any $x \in \mathbb{R}^m$,
\begin{equation*}
\# \left\{i : \overline{V}_i \cap B(x,r) \neq \emptyset \right\} \leq \left(\frac{1+2c_2}{c_1}\right)^m.
\end{equation*}
\end{lem}

\begin{lem}
\label{OSCg}
Let $p \in \mathbb{R}$, let $a_i\geq 0$ for $1 \leq i \leq m$, and let $m \leq C$. Then
\begin{equation*}
\biggl(\sum_{i = 1}^m a_i\biggr)^p \leq \max \left\{1, C^{p-1}\right\}\sum_{i = 1}^m a_i^p.
\end{equation*} 
\end{lem}
\begin{proof} Minkowski's inequality, for $0<p<1$, can be used to show
\begin{equation*}
\biggl(\sum_{i = 1}^m a_i\biggr)^p \leq \sum_{i = 1}^m a_i^p,
\end{equation*}
and as this also clearly holds for $p\leq 0$ and $p=1$, the inequality holds for  $p \leq 1$. 

For $1 < p$, H$\ddot{\textrm{o}}$lder's inequality can be used to show
\begin{equation*}
\biggl(\sum_{i = 1}^m a_i\biggr)^p \leq m^{p-1}\sum_{i = 1}^m a_i^p\leq C^{p-1}\sum_{i = 1}^m a_i^p.
\qedhere
\end{equation*}
\end{proof}

We take $(\bl_v)_{v \in V}$ to be the fixed list of paths used in the definition of the non-negative matrix $\bB(q, \gamma, l)$, in Subsection \ref{matrix B} with $S_{\bl_v}(F_{t(\bl_v)}) \subset U_v$, as given in Equation \eqref{l_u}. The associated list of  positive constants $(c_v)_{v \in V}$ are as defined in Subsection \ref{deltadef}, Equation \eqref{l_2}, where 
\begin{equation*}
c_v = \dist\left(S_{\bl_v}(F_{t(\bl_v)}), \, \mathbb{R}^m \setminus U_v\right) > 0.
\end{equation*}
By Lemma \ref{OSCintersection}(b), for all finite paths $\be, \bg \in E_u^*$, with 
\begin{equation*}
\be = e_1\cdots e_k, \ \bg= g_1\cdots g_j, \ \textrm{and} \ e_1 \neq g_1,
\end{equation*}
\begin{equation}
\label{l_4}
S_{\be}\left(F_{t(\be)}) \cap S_{\bg}(U_{t(\bg)}\right) = \emptyset.
\end{equation}

\begin{lem}
\label{OSCd}
Let $(\bl_v)_{v \in V}$ be the list of paths defined  in Equation (\ref{l_u}) and Subsection \ref{matrix B}, and let $(c_v)_{v \in V}$ be the associated list of positive constants defined in Equation (\ref{l_2}). Let $\be, \bg \in E_u^*$ be any finite paths with
\begin{equation*}
\be = e_1\cdots e_k, \ \bg= g_1\cdots g_j,  \ \textrm{and} \ e_1 \neq g_1,
\end{equation*} 
and suppose
\begin{equation*}
\dist\left(S_{\be}(F_{t(\be)}), S_{\bg}(F_{t(\bg)})\right) \leq c_w r_{\bg}
\end{equation*}
for some vertex $w \in V$. 

Then $\bl_w$ is not a subpath of $g_2\cdots g_j$, that is $\bl_w \not\subset g_2\cdots g_j$.
\end{lem}
 
\begin{proof}
For a contradiction we assume $\bl_w$ is a subpath of $g_2\cdots g_j$ so that $\bg = \bs \bl_w \bt$ where $\bs \neq \emptyset$,  $i(\bs) = u$, $t(\bs) = i(\bl_w) = w$ and $t(\bl_w) = i(\bt)$. Clearly $S_{\bg}(F_{t(\bg)})=S_{\bs \bl_w \bt}(F_{t(\bt)})\subset S_{\bs \bl_w }(F_{t(\bl_w)})\subset S_{\bs}(U_w)$, by (\ref{l_u}). This implies that
\begin{align*}
\dist\left(S_{\bg}(F_{t(\bg)}), \mathbb{R}^m \setminus S_{\bs}(U_w)\right) &\geq \textrm{dist}\left(S_{\bs \bl_w }(F_{t(\bl_w)}), \mathbb{R}^m \setminus S_{\bs}(U_w)\right) && \\
&= r_{\bs}c_w && (\textrm{by (\ref{l_2})})  \\
&>r_{\bg}c_w &&  (\textrm{as $\bs \subset \bg$}). 
\end{align*}
By assumption $s_1 = g_1 \neq e_1$ and also $t(\bs)=w$ so that $S_{\be}(F_{t(\be)})\cap S_{\bs}(U_w)=\emptyset$ by (\ref{l_4}). This means that $S_{\be}(F_{t(\be)})\subset \mathbb{R}^m \setminus S_{\bs}(U_w)$ and so
\begin{align*}
\textrm{dist}\left(S_{\bg}(F_{t(\bg)}), S_{\be}(F_{t(\be)})\right) &\geq \textrm{dist}\left(S_{\bg}(F_{t(\bg)}), \mathbb{R}^m \setminus S_{\bs}(U_w)\right) \\
&>c_wr_{\bg}.
\end{align*}
This is the required contradiction.
\end{proof}

For $\be \in E_u^*$, $\be \vert_{\left|\be\right|-1}$ is the finite path obtained by deleting the last edge of $\be$. For $r>0$ let 
\begin{equation}
\label{E_u*r}
E_u^*\left(r\right) = \left\{\be \in E_u^*: r_{\be}\left|F_{t(\be)}\right|<r\leq  r_{\be \vert_{\left|\be\right|-1}}\left|F_{t(\be \vert_{\left|\be\right|-1})}\right|\right\}
\end{equation} 
We make the following observations about the set of finite paths $E_u^*\left(r\right)$.
\begin{itemize}
\item For paths $\be \in E_u^*\left(r\right)$, the sets $S_{\be}(F_{t(\be)}) \subset F_u$ are all roughly of diameter $r$ since $\left|S_{\be}(F_{t(\be)})\right|=r_{\be}\left|F_{t(\be)}\right|$.
\item It can be shown, using Lemma \ref{surjection phi_u}, that 
\begin{equation*}
F_u=\bigcup_{\be \in E_u^*(r)}S_{\be}\left(F_{t(\be)}\right).
\end{equation*}
\item If $(U_v)_{v \in V}$ are the open sets of the OSC, then the sets $\left\{S_{\be}(U_{t(\be)}):\be \in E_u^*\left(r\right) \right\}$ are disjoint open sets. This follows from the definition of the OSC and Lemma \ref{OSCintersection}(a), using the fact that $\be,\bg \in E_u^*\left(r\right)$, $\be \neq \bg$, implies $\be \not\subset \bg$ and $\bg \not\subset \be$.
\end{itemize}

We remind the reader that, as described in Subsection \ref{deltadef}, Inequality \eqref{N}, we choose $N \in \mathbb{N}$ large enough so that
\begin{equation*}
\frac{2d_{\max}}{c_{\min}} \leq \frac{1}{r_{\max}^{N-1}},
\end{equation*}
and for such $N$,  as given in Equation \eqref{delta}, $\delta$ is then defined as
\begin{equation*}
\delta=r_{\min}^{N+l+1}d_{\min}.
\end{equation*} 
Also for a given $r$, $H_{e,f}$ is an $r$-separated subset of $S_e(F_{t(e)})\cap S_{f}(F_{t(f)})(r)$, where the edges $e,f \in E_u^1$, are taken as fixed with $e \neq f$.
\begin{lem}
\label{OSCh}
Let $r \in (0,\delta)$, let $x \in H_{e,f}$, let $(\bl_v)_{v \in V}$ be the list of paths defined  in Equation (\ref{l_u}) and Subsection \ref{matrix B}, let $N$ be as defined in Equation (\ref{N}), and let $\be=e_1\ldots e_{\left|\be\right|} \in E_u^*\left(r\right)$ be such that $\dist(x, S_{\be}(F_{t(\be)}))\leq r$.

Then 
\begin{equation*}
\bl_v \not\subset e_2\ldots e_{\left|\be\right|-N},
\end{equation*}
for all $v \in V$.
\end{lem}
\begin{proof} For $r \in (0,\delta)$, considered fixed, and a path $\be \in E_u^*\left(r\right)$, if $\left|\be\right|< N+l+1$, then 
\begin{equation*}
r  <  \delta = r_{\min}^{N+l+1}d_{\min}  <  r_{\min}^{\left|\be\right|}d_{\min}  \leq   r_{\be}\left|F_{t(\be)}\right|  <  r,
\end{equation*}
and this contradiction ensures $\left|\be\right| \geq N+l+1$. Let $\be \in E_u^*\left(r\right)$ be written as $\be =e_1\ldots\ e_{\left|\be\right|}$. Either $e_1=e$ or $e_1 \neq e$, and so we consider these two cases in turn.

\medskip

(a) $e_1 = e$.

In this case $e_1 \neq f$. Since $S_{\be}\left(F_{t(\be)}\right)\subset S_{\be \vert_{\left|\be \right|-N}}\bigl(F_{t(\be \vert_{\left|\be \right|-N})}\bigr)$ it follows that
\begin{displaymath}
\dist\bigl(x,S_{\be \vert_{\left|\be \right|-N}}(F_{t(\be \vert_{\left|\be \right|-N})})\bigr) \leq \dist\left(x,S_{\be}(F_{t(\be)})\right) \leq r.
\end{displaymath}
As $x \in H_{e,f}$, $x \in S_f(F_{t(f)})(r)$ and from the definition of the closed $r$-neighbourhood
\begin{displaymath}
\dist\left(x, S_f(F_{t(f)})\right) \leq r.
\end{displaymath}
Hence
\begin{align*}
\dist\bigl(S_f(F_{t(f)}), \  &S_{\be \vert_{\left|\be \right|-N}}(F_{t(\be \vert_{\left|\be \right|-N})})\bigr)  && \\ 
&\leq   \dist\left(x, S_f(F_{t(f)})\right)+\dist\bigl(x,S_{\be \vert_{\left|\be \right|-N}}(F_{t(\be \vert_{\left|\be \right|-N})})\bigr) && \\ 
&\leq 2r && \\
&\leq \frac{c_{\min}}{r_{\max}^{N-1}d_{\max}}r && (\textrm{by (\ref{N})})\\
&\leq \frac{c_{\min}}{r_{\max}^{N-1}}\biggl( \ r_{\be \vert_{\left|\be \right|-1}}\frac{\bigl|F_{t(\be \vert_{\left|\be \right|-1})}\bigr|}{d_{\max}} \ \biggr)   && (\textrm{as $\be \in E_u^*\left(r\right)$}) \\
&\leq c_{\min}r_{\be \vert_{\left|\be \right|-N}} && \\
&\leq c_vr_{\be \vert_{\left|\be \right|-N}},    &&
\end{align*}
for all $v \in V$. 

Applying Lemma \ref{OSCd} it follows that $\bl_v \not\subset e_2\ldots e_{\left|\be\right|-N}$ for all $v \in V$.

\medskip

(b) $e_1 \neq	e$. 

In this case the argument is almost identical to that given in part (a), but using $S_e(F_{t(e)})$ in place of $S_f(F_{t(f)})$, where we have $\dist(x, S_e(F_{t(e)}))=0 \leq r$.
\end{proof}

\begin{lem}
\label{OSC3}
Let $r \in (0,\delta)$ and  let $x \in H_{e,f}$, then there exists a path $\be_x \in E_u^*\left(r\right)$, which depends on $x$, such that
\begin{align*}
&\textup{(a) }\quad x \in S_{\be_x}(F_{t(\be_x)})\subset F_u \cap B(x,r)\subset \bigcup_{\substack{\be \in E_u^*(r) \\ \dist(x, S_{\be}(F_{t(\be)}))\leq r}}S_{\be}(F_{t(\be)}), \\
&\textup{(b) }\quad \mu_u\left(B(x,r)\right)^q \ \leq  \ 
\begin{cases} \mu_u\left(S_{\be_x}(F_{t(\be_x)})\right)^q, & \  \text{if $q \leq 0$,}
\\
\Biggl( \ \sum\limits_{\substack{\be \in E_u^*(r) \\ \dist(x, S_{\be}(F_{t(\be)}))\leq r}} \mu_u\left(S_{\be}(F_{t(\be)})\right) \ \Biggr)^q, &  \  \text{if $q > 0$}.
\end{cases}
\end{align*}
\end{lem}
\begin{proof}
As $\supp \mu_u = F_u$, $\mu_u\left(B(x,r)\right)=\mu_u\left(F_u\cap B(x,r)\right)$, part (b) is an immediate consequence of part (a).

$H_{e,f}$ is an $r$-separated subset of $S_e(F_{t(e)})\cap S_{f}(F_{t(f)})(r)$, so $x \in F_u$ and the map $\phi_u : E_u^\mathbb{N} \to F_u$ given in Lemma \ref{surjection phi_u} ensures the existence of an infinite path $\be \in E_u^\mathbb{N}$ with
\begin{equation*}
\left\{x\right\} = \bigcap_{k=1}^{\infty} S_{\be \vert_{ k}}(F_{t(\be \vert_{ k})})
\end{equation*}
Now $\left(S_{\be \vert_{ k}}(F_{t(\be \vert_{ k})})\right)$ is a decreasing sequence of non-empty compact sets whose diameters tend to zero as $k$ tends to infinity and so there exists $j\in \mathbb{N}$ such that 
\begin{equation*}
\left|S_{\be \vert_j}(F_{t(\be \vert_j)})\right|=r_{\be \vert_j} \left|F_{t(\be \vert_j)}\right|< r \leq r_{\be \vert_{j-1}} \left|F_{t(\be \vert_{j-1})}\right|=\left|S_{\be \vert_{j-1}}(F_{t(\be \vert_{j-1})})\right|.
\end{equation*}
Putting $\be_x=\be \vert_j$, $\be_x \in E_u^*\left(r\right)$ and
\begin{equation*}
x \in S_{\be_x}(F_{t(\be_x)})\subset F_u \cap B(x,r).
\end{equation*}
By the same argument for any $y \in F_u \cap B(x,r)$ there exists a path $\be_y \in E_u^*\left(r\right)$ such that $y \in S_{\be_y}(F_{t(\be_y)})\subset F_u \cap B(y,r)$. Since $y \in B(x,r)$ it follows that 
\begin{equation*}
\dist\left(x, S_{\be_y}(F_{t(\be_y)})\right) \leq r
\end{equation*}
so that 
\begin{equation*}
F_u \cap B(x,r)\subset \bigcup_{\substack{\be \in E_u^*(r) \\ \dist(x, S_{\be}(F_{t(\be)}))\leq r}}S_{\be}(F_{t(\be)}). 
\qedhere
\end{equation*}
\end{proof}

In the next lemma $\overline{U}_v$ is the closure of $U_v$.
\begin{lem}
\label{closure U_u}
Let $(F_v)_{v \in V} \in \left(K\left(\mathbb{R}^{m}\right)\right)^{n}$ be the attractor of an $n$-vertex IFS, as given by Equation \eqref{Invariance 1}, and suppose that the OSC is satisfied by the non-empty bounded open sets $(U_v)_{v\in V}$, then
\begin{equation*}
F_v\subset \overline{U}_v, \textrm{ for each } v \in V. 
\end{equation*}   
\end{lem} 
\begin{proof} See \cite[Lemma 1.3.6]{phdthesis_Boore}.
\end{proof}

We remind the reader here that ${E_u^*\left(r\right)}$ used in the statement of Lemma \ref{OSC4} is the set of finite paths in the directed graph as defined in Equation \eqref{E_u*r}.

\begin{lem}
\label{OSC4}
Let $q \in \mathbb{R}$. Then there exists a positive number $C_2\left(q\right)$, such that for all $r \in (0,\delta)$ and all $x \in H_{e,f}$,
\begin{equation*}
\mu_u\left(B(x,r)\right)^q  \leq   
 C_2\left(q\right)\sum_{\substack{\be \in E_u^*(r) \\ \dist(x, S_{\be}(F_{t(\be)}))\leq r}} \mu_u\left(S_{\be}(F_{t(\be)})\right)^q.
\end{equation*}
\end{lem}

\begin{proof} The sets $\left\{S_{\be}(U_{t(\be)}):\be \in E_u^*\left(r\right)\right\}$ are disjoint open sets, where $(U_v)_{v \in V}$ are the open sets of the SOSC. We may assume each $U_v$ contains a closed ball of radius $\rho_1$ and is contained in a closed ball of radius $\rho_2$, so that $S_{\be}(\overline{U}_{t(\be)})$ contains a closed ball of radius $r_{\be}\rho_1$ and is contained in a closed ball of radius $r_{\be}\rho_2$. For any $\be \in E_u^*\left(r\right)$ 

\begin{equation*}
r_{\be}\left|F_{t(\be)}\right|<r\leq  r_{\be \vert_{\left|\be \right|-1}}\left|F_{t(\be \vert_{\left|\be \right|-1})}\right|,
\end{equation*}
so that
\begin{equation*}
 \frac{r_{\min}}{d_{\max}} r \leq r_{\be}<\frac{r}{d_{\min}}.
\end{equation*}

This means that, for each $\be \in E_u^*\left(r\right)$, $S_{\be}(\overline{U}_{t(\be)})$ contains a closed ball of radius $\frac{r_{\min}\rho_1}{d_{\max}}r=c_1r$, which we label as $B_\be$, and is contained in a closed ball of radius $\frac{\rho_2}{d_{\min}}r=c_2r$. The set $\left\{B_\be : \be \in E_u^*\left(r\right) \right\}$ is a disjoint set because  $\left\{S_{\be}(U_{t(\be)}) : \be \in E_u^*\left(r\right) \right\}$ is disjoint. The situation is illustrated schematically in Figure \ref{Ch4b}, in $\mathbb{R}^2$, where we have indicated the closed ball $B_{\be_x}$ contained in $S_{\be_x}(\overline{U}_{t(\be_x)})$. 

\begin{figure}[!htb]
\begin{center}
\includegraphics[trim = 30mm 138mm 35mm 12mm, clip, scale=0.7]{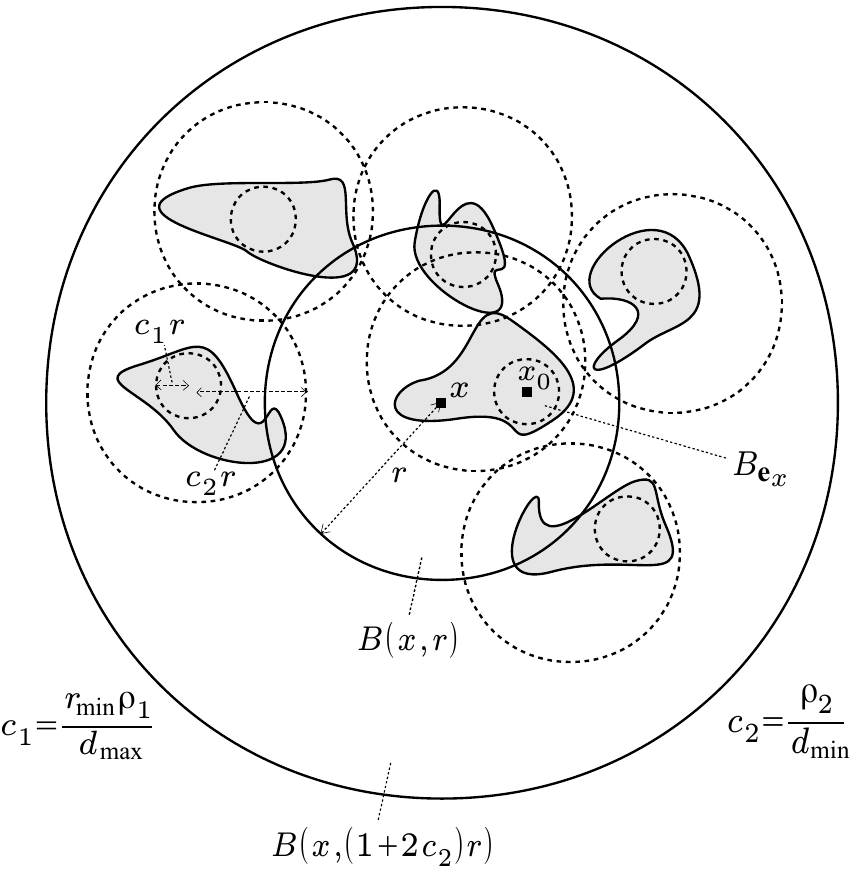}
\end{center}
\caption{The closed sets $\left\{S_{\be}(\overline{U}_{t(\be)}) : \be \in E_u^*\left(r\right), \, \dist(x, S_{\be}(F_{t(\be)}))\leq r\right\}$, illustrated schematically as grey areas in $\mathbb{R}^2$. Each $S_{\be}(\overline{U}_{t(\be)})$ contains a closed ball of radius $c_1r$ and is contained in a closed ball of radius $c_2r$.}
\label{Ch4b}
\end{figure}

We now obtain   
\begin{align*}
\# \{ \be :  \be \in &E_u^*\left(r\right), \, \dist( x, S_{\be}(F_{t(\be)}) )\leq r \}  && \\
&=  \# \left\{  \be :  \be \in E_u^*\left(r\right), \, S_{\be}(F_{t(\be)}) \cap B\left(x,r\right) \neq \emptyset \right\}  && \\  
& \leq  \# \left\{  \be :  \be \in E_u^*\left(r\right), \, S_{\be}(\overline{U}_{t(\be)}) \cap B\left(x,r\right) \neq \emptyset \right\} && (\textrm{\small{by Lemma \ref{closure U_u}}}) && \\ 
&\leq  \left(\frac{ 1 + \frac{2\rho_2}{d_{\min}}}{\frac{r_{\min}\rho_1}{d_{\max}}}\right)^m = C_1 &&(\textrm{\small{by Lemma \ref{OSCf}}}). 
\end{align*}
Here we have used Lemma \ref{OSCf} with $c_1=\frac{r_{\min}\rho_1}{d_{\max}}$ and $c_2=\frac{\rho_2}{d_{\min}}$, as shown in Figure \ref{Ch4b}. Applying Lemma \ref{OSCg} gives 
\begin{equation*}
\Biggl( \ \sum_{\substack{\be \in E_u^*(r) \\ \dist(x, S_{\be}(F_{t(\be)}))\leq r}} \mu_u\left(S_{\be}(F_{t(\be)})\right) \ \Biggr)^q \ \leq  \ 
C_2\left(q\right)\sum_{\substack{\be \in E_u^*(r) \\ \dist(x, S_{\be}(F_{t(\be)}))\leq r}} \mu_u\left(S_{\be}(F_{t(\be)})\right)^q,
\end{equation*}
where $C_2\left(q\right) = \max\left\{1, C_1^{q-1}\right\}$.

As $\be_x \in E_u^*\left(r\right)$ and $\dist(x, S_{\be_x}(F_{t(\be_x)}))=0$ we also have 
\begin{equation*}
\mu_u\left(S_{\be_x}(F_{t(\be_x)})\right)^q \leq  C_2\left(q\right)\sum_{\substack{\be \in E_u^*(r) \\ \dist(x, S_{\be}(F_{t(\be)}))\leq r}} \mu_u\left(S_{\be}(F_{t(\be)})\right)^q.
\end{equation*}
The result now follows by Lemma \ref{OSC3}(b). 
\end{proof}

In the next lemma we use a second application of Lemma \ref{OSCf} to obtain a bound for $\sum_{x \in H_{e,f}}\mu_u\left(B(x,r)\right)^q$. 
\begin{lem}
\label{OSC5}
Let $q \in \mathbb{R}$, let $r \in (0,\delta)$, and let $(\bl_v)_{v \in V}$ be the list of paths defined  in Equation (\ref{l_u}) and Subsection \ref{matrix B}.

Then
\begin{equation*}
\sum_{x \in H_{e,f}}\mu_u\left(B(x,r)\right)^q   \leq   C_2\left(q\right)C_3 \sum_{\substack{\be \in E_u^*(r) \\ \forall v \in V : \ \bl_v \, \not\subset  \, e_2\ldots e_{\left|\be\right|-N}}} p_\be^q. 
\end{equation*}
\end{lem}

\begin{proof}
By Lemma \ref{OSC3}(a), given any $y \in H_{e,f}$, we can find a path $\be_y \in E_u^*\left(r\right)$ such that
\begin{equation}
\label{subset_2}
\begin{split}
y \in S_{\be_y}(F_{t(\be_y)})&\subset F_u \cap B(y,r)\subset \bigcup_{\substack{\be \in E_u^*(r) \\ \dist(y, S_{\be}(F_{t(\be)}))\leq r}}S_{\be}(F_{t(\be)}) \\
&\subset \bigcup_{\substack{\be \in E_u^*(r) \\ \dist(y, S_{\be}(F_{t(\be)}))\leq r}}S_{\be}( \overline{U}_{t(\be)}), 
\end{split}
\end{equation}
where $(U_v)_{v \in V}$ are the open sets of the SOSC and we have used Lemma \ref{closure U_u}. For $y \in H_{e,f}$ it is convenient to use the notation 
\begin{equation*}
\overline{U}(y) = \bigcup_{\substack{\be \in E_u^*(r) \\ \dist(y, S_{\be}(F_{t(\be)}))\leq r}}S_{\be}( \overline{U}_{t(\be)}). 
\end{equation*}
As before we are assuming the open sets, $(U_v)_{v \in V}$, each contain a closed ball of radius $\rho_1$ and are contained in a closed ball of radius $\rho_2$. As explained in the proof of Lemma \ref{OSC4} this means that for each $\be \in E_u^*\left(r\right)$, $S_{\be}(\overline{U}_{t(\be)})$ contains a closed ball, $B_\be$, of radius $\frac{r_{\min}\rho_1}{d_{\max}}r$ and is contained in a closed ball of radius $\frac{\rho_2}{d_{\min}}r$, where the set $\left\{B_\be : \be \in E_u^*\left(r\right) \right\}$ is disjoint. 

Now consider $x,y \in H_{e,f}$ with $x \neq y$. The paths $\be_x, \be_y \in E_u^*\left(r\right)$, established in Equation (\ref{subset_2}) by Lemma \ref{OSC3}(a), cannot be the same since $\be_x=\be_y$ means $x \in B(y,r)$ which is impossible as $H_{e,f}$ is $r$-separated. So $\be_x \neq \be_y$ for $x,y \in H_{e,f}$, $x \neq y$. This means that for each of the sets
\begin{equation*}
\biggl\{\overline{U}(x) :x \in H_{e,f}\biggr\},
\end{equation*}
we may choose a single closed ball $B_x=B_{\be_x}\subset S_{\be_x}(\overline{U}_{t(\be_x)}) \subset \overline{U}(x)$, where $B_x$ is of radius $\frac{r_{\min}\rho_1}{d_{\max}}r=c_1r$, with $B_x=B(x_0,\frac{r_{\min}\rho_1}{d_{\max}}r)$ for some point $x_0 \in S_{\be_x}(\overline{U}_{t(\be_x)})$. The closed ball $B_x=B_{\be_x}$ is indicated in both Figures \ref{Ch4b} and \ref{Ch4c}. The set $\left\{B_x :  x \in H_{e,f} \right\}$ is then a disjoint set of closed balls.

\begin{figure}[!htb]
\begin{center}
\includegraphics[trim = 5mm 145mm 5mm 9mm, clip, scale=0.7]{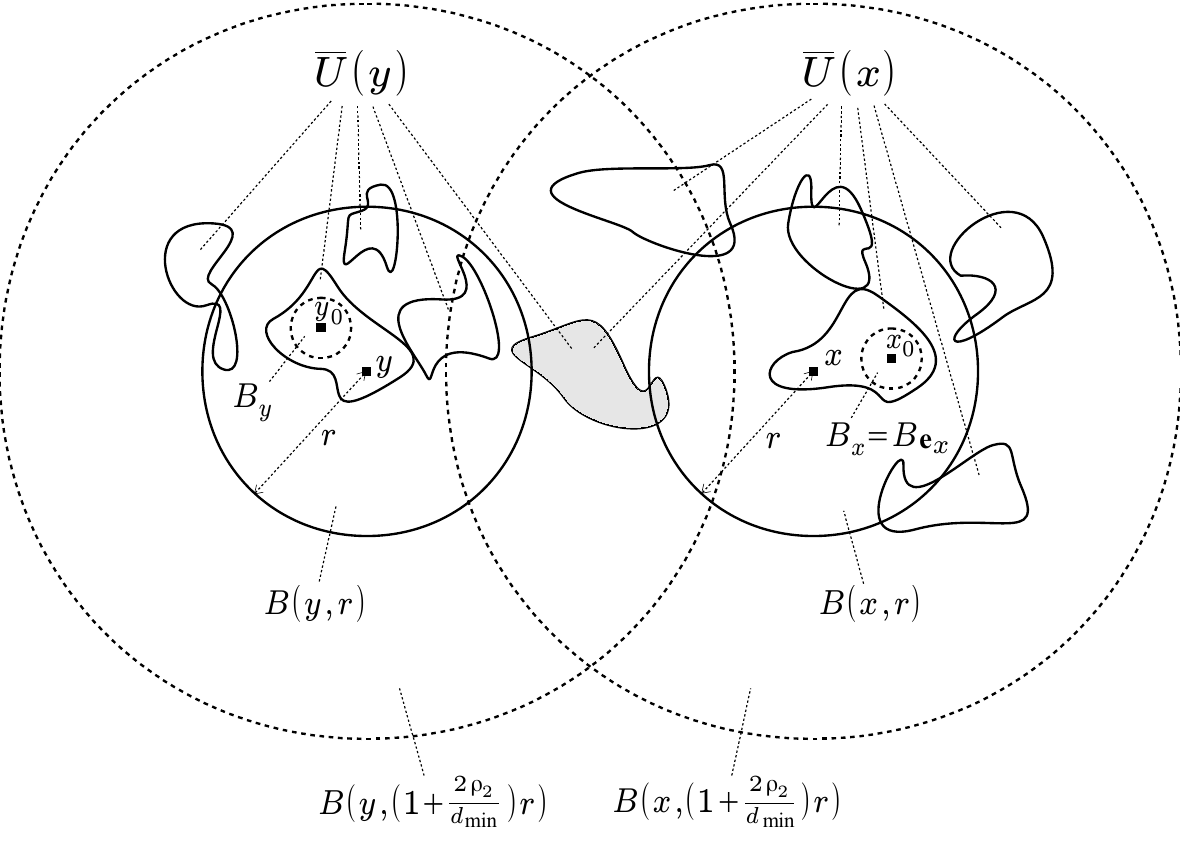}
\end{center}
\caption{Two sets $\overline{U}(x)$ and $\overline{U}(y)$, illustrated schematically in $\mathbb{R}^2$. The grey set belongs to both $\overline{U}(x)$ and $\overline{U}(y)$.}
\label{Ch4c}
\end{figure}

It is also the case that $\left| S_{\be}(\overline{U}_{t(\be)})\right| \leq \frac{2\rho_2}{d_{\min}}r$, because $S_{\be}(\overline{U}_{t(\be)})$ is contained in a closed ball of radius $\frac{\rho_2}{d_{\min}}r$, so that for each $x \in H_{e,f}$,
\begin{equation*}
\overline{U}(x) \subset B\Bigl(x, \Bigl(1 + \frac{ \scriptstyle 2\rho_2}{\scriptstyle d_{\min}}\Bigr)r\Bigr). 
\end{equation*} 
That is, for each $x \in H_{e,f}$, $\overline{U}(x)$ is contained in a closed ball of radius $\bigl(1 + \frac{ 2\rho_2}{ d_{\min}}\bigr)r=c_2r$. The situation is illustrated in Figure \ref{Ch4c} in $\mathbb{R}^2$, for two sets $\overline{U}(x)$ and $\overline{U}(y)$. 

We are now in a position to apply Lemma \ref{OSCf} again. Let $x \in H_{e,f}$ be fixed and let $\bg \in E_u^*\left(r\right)$ be a path for which $\dist(x,S_\bg(F_{t(\bg)}))\leq r$, which we also consider to be fixed. Let $N(x,\bg) \in \mathbb{N}$, be the number of times the term $\mu_u(S_{\bg}(F_{t(\bg)}))^q$ is counted in the sum
\begin{equation*}
\sum_{y \in H_{e,f}}\sum_{\substack{\be \in E_u^*(r) \\ \dist(y, S_{\be}(F_{t(\be)}))\leq r}} \mu_u\left(S_{\be}(F_{t(\be)})\right)^q.
\end{equation*}
Then 
\begin{align*}
N(x,\bg) &= \# \left\{y : y \in H_{e,f} \textrm{ and } \dist(y,S_\bg(F_{t(\bg)}))\leq r\right\} &&\\
&\leq \# \Biggl\{y : y \in H_{e,f} \textrm{ and } \bigcup_{\substack{\be \in E_u^*(r) \\ \dist(y, S_{\be}(F_{t(\be)}))\leq r}}  S_{\be}(F_{t(\be)}) &&  \cap  B(x,r)  \neq    \emptyset \Biggr\} \\
&\leq \# \left\{y : y \in H_{e,f} \textrm{ and } \overline{U}(y) \ \cap \ B(x,r) \neq \emptyset\right\} && (\textrm{\small{by (\ref{subset_2})}}) &&\\
&\leq\left(\frac{1 + 2\left(1 + \frac{\scriptstyle 2\rho_2}{\scriptstyle d_{\min}}\right)}{\frac{r_{\min}\rho_1}{d_{\max}}}\right)^m= C_3 && (\textrm{\small{by Lemma \ref{OSCf}}}).  
\end{align*}
Here we have applied Lemma \ref{OSCf} with $c_1=\frac{r_{\min}\rho_1}{d_{\max}}$ and $c_2=1 + \frac{ \scriptstyle 2\rho_2}{\scriptstyle d_{\min}}$.
Using this result it is clear that for each distinct path $\be$ in the sum
\begin{displaymath}
\sum_{x \in H_{e,f}}\sum_{\substack{\be \in E_u^*(r) \\ \dist(x, S_{\be}(F_{t(\be)}))\leq r}} \mu_u\left(S_{\be}(F_{t(\be)})\right)^q,
\end{displaymath}
the term $ \mu_u\left(S_{\be}(F_{t(\be)})\right)^q$ is counted at most $C_3$ times. As an example, if $\be^\prime$ is the path corresponding to the set $S_{\be^\prime}( \overline{U}_{t(\be^\prime)})$, coloured grey in Figure \ref{Ch4c}, then $\dist(x, S_{\be^\prime}(F_{t(\be^\prime)}))\leq r$ and $\dist(y, S_{\be^\prime}(F_{t(\be^\prime)}))\leq r$, so that $\mu_u\left(S_{\be^\prime}(F_{t(\be^\prime)})\right)^q$ would be counted at least twice in this sum, for $x,y \in H_{e,f}$. 

This implies that
\begin{align*}
\sum_{x \in H_{e,f}}\mu_u\left(B(x,r)\right)^q &\leq C_2\left(q\right)\sum_{x \in H_{e,f}}\sum_{\substack{\be \in E_u^*(r) \\ \dist(x, S_{\be}(F_{t(\be)}))\leq r}} \mu_u\left(S_{\be}(F_{t(\be)})\right)^q && (\textrm{\small{by Lemma \ref{OSC4}}}) \\
&\leq C_2\left(q\right)C_3 \sum_{\substack{\be \in E_u^*(r) \\ \forall v \in V : \ \bl_v \, \not\subset  \, e_2\ldots e_{\left|\be\right|-N}}} \mu_u\left(S_{\be}(F_{t(\be)})\right)^q && (\textrm{\small{by Lemma \ref{OSCh}}}) \\
&\leq C_2\left(q\right)C_3 \sum_{\substack{\be \in E_u^*(r) \\ \forall v \in V : \ \bl_v \, \not\subset  \, e_2\ldots e_{\left|\be\right|-N}}} p_\be^q && (\textrm{\small{by Lemma \ref{OSCe}}}).
\end{align*}
\end{proof}

We now define, for $r \in (0, +\infty)$, two related column vectors $(G_w\left(r\right))_{w \in V}^\transpose$ and $(\G_w\left(r\right))_{w \in V}^\transpose$. For $q\in \mathbb{R}$, let $\gamma=\gamma\left(q\right)\in \mathbb{R}$ be the unique number such that $\rho\left(\bB\left(q,\gamma,l\right)\right)=1$ for the matrix $\bB\left(q,\gamma,l\right)$, as defined in Subsection \ref{matrix B}, which we assume is irreducible. Let $(\bl_v)_{v \in V}$ be the list of paths defined  in Equation (\ref{l_u}) and Subsection \ref{matrix B}, let $N$ be as chosen in Inequality (\ref{N}), let $r \in (0, +\infty)$, and let 
\begin{equation*}
\alpha = \frac{1}{r_{\max}^N d_{\max}} \quad \textrm{and} \quad \beta = \frac{1}{r_{\min}^{N+1}d_{\min}}.
\end{equation*}
For each $w \in V$, let
\begin{equation}
\label{Gw(r)}
G_w\left(r\right) = \sum_{ \substack{ \bg \in E_w^* \\ \alpha r \, \leq \, r_\bg \, < \, \beta r\\ \forall v \in V : \ \bl_v \, \not\subset  \, \bg } } p_\bg^q,  
\end{equation}
and 
\begin{equation}
\label{mathcalGw(r)}
\G_w\left(r\right) = r^{\gamma\left(q\right)}G_w\left(r\right).  
\end{equation}

We point out here that for small $r$, with $r \in (0,\delta)$,
\begin{equation}
\label{g greater than l}
G_w\left(r\right) = \sum_{ \substack{ \bg \in E_w^* \\ \alpha r \, \leq \, r_\bg \, < \, \beta r\\ \forall v \in V : \ \bl_v \, \not\subset  \, \bg } } p_\bg^q \quad = \sum_{ \substack{ \bg \in E_w^* \\  \left|\bg\right| \geq \,l \\ \alpha r \, \leq \, r_\bg \, < \, \beta r \\ \forall v \in V : \ \bl_v \, \not\subset  \, \bg } } p_\bg^q.  
\end{equation}
This is because $0<r<\delta= r_{\min}^{N+l+1}d_{\min}$, so if $\left|\bg\right|< l$ and $r_\bg<\frac{r}{r_{\min}^{N+1}d_{\min}}= \beta r$ then 
\begin{equation*}
r_\bg<\frac{r}{r_{\min}^{N+1}d_{\min}}<r_{\min}^l<r_{\min}^{\left|\bg\right|} \leq r_\bg, 
\end{equation*}
and this contradiction ensures $\left|\bg\right| \geq l$.

\begin{lem}
\label{OSC6}
Let $r \in (0,\delta)$, and let $G_w\left(r\right)$ be as defined in Equation (\ref{Gw(r)}), for each $w \in V$.

Then  
\begin{equation*}
\sum_{x \in H_{e,f}}\mu_u\left(B(x,r)\right)^q  \leq  C_2\left(q\right)C_3C_4\left(q\right)  \sum_{w \in V}G_w\left(r\right).
\end{equation*}
\end{lem}
\begin{proof}
As we showed in the proof of Lemma \ref{OSCh}, for $r \in (0,\delta)$, $\be \in E_u^*\left(r\right)$ implies $\left|\be\right|\geq N+l+1$, so $\be$ can always be written as $\be=\bs\bg\bt$, for some paths $\bs,\bg,\bt$, with $\left|\bs\right|=1$, $\left|\bg\right|\geq l$ and $\left|\bt\right|=N$. From the definition of $E_u^*\left(r\right)$ in Equation \eqref{E_u*r}, for $\be=\bs\bg\bt \in E_u^*\left(r\right)$,
\begin{equation*}
r_{\min}^{N+1}r_\bg d_{\min} \leq r_{\bs\bg\bt}\left|F_{t(\bs\bg\bt)}\right| < r \leq  r_{\bs\bg\bt \vert_{\left|\bs\bg\bt\right|-1}}\left|F_{t(\bs\bg\bt \vert_{\left|\bs\bg\bt\right|-1})}\right| \leq r_{\max}^Nr_\bg d_{\max},
\end{equation*}
and so
\begin{equation}
\label{alpha beta}
\alpha r = \frac{r}{r_{\max}^N d_{\max}}\leq r_\bg < \frac{r}{r_{\min}^{N+1}d_{\min}}=\beta r.
\end{equation}
This gives  
\begin{align*}
\sum_{x \in H_{e,f}}\mu_u\left(B(x,r)\right)^q  & \leq   C_2\left(q\right)C_3 \sum_{\substack{\be \in E_u^*(r) \\ \forall v \in V : \ \bl_v \, \not\subset  \, e_2\ldots e_{\left|\be\right|-N}}} p_\be^q && (\textrm{\small{by Lemma \ref{OSC5}}}) \\
& = C_2\left(q\right)C_3 \sum_{ \substack{ \be = \bs \bg \bt \in E_u^*(r) \\ \left|\bs\right|=1, \, \left|\bt\right|=N, \, \left|\bg\right| \geq \,l \\ \forall v \in V : \ \bl_v \, \not\subset  \, e_2\ldots e_{\left|\be\right|-N} =  \bg } } p_{\bs \bg \bt}^q && \\
& \leq  C_2\left(q\right)C_3  \sum_{\bs \in E_u^1}  \sum_{\bt \in E^N} \sum_{ \substack{ \bg \in E^* \\  \left|\bg\right| \geq \,l \\ \alpha r \, \leq \, r_\bg \, < \, \beta r  \\ \forall v \in V : \ \bl_v \, \not\subset  \, \bg } } p_{\bs}^qp_{\bg}^q p_{\bt}^q && (\textrm{\small{by (\ref{alpha beta})}})  \\
& \leq  C_2\left(q\right)C_3C_4\left(q\right)  \sum_{ \substack{ \bg \in E^* \\  \left|\bg\right| \geq \,l \\ \alpha r \, \leq \, r_{\bg} \, < \, \beta r  \\ \forall v \in V : \ \bl_v \, \not\subset  \, \bg } } p_{\bg}^q &&  \\
& =  C_2\left(q\right)C_3C_4\left(q\right)  \sum_{w \in V} \sum_{ \substack{ \bg \in E_w^* \\  \left|\bg\right| \geq \,l \\ \alpha r \, \leq \, r_{\bg} \, < \, \beta r \\ \forall v \in V : \ \bl_v \, \not\subset  \, \bg } } p_\bg^q && \\
& =  C_2\left(q\right)C_3C_4\left(q\right)  \sum_{w \in V} G_w\left(r\right) && (\textrm{\small{by (\ref{g greater than l})}}).
\end{align*}
The positive constant $C_4\left(q\right)$, which depends on $q$, is given by
\begin{equation*}
C_4\left(q\right) = k^{1+N}(\max \{p_e^q:  e \in E^1 \})^{1+N},  
\end{equation*}
$k$ being the maximum number of edges leaving any vertex in the directed graph. 
\end{proof}

\begin{lem}
\label{OSC7}
Let $r \in (0,\delta)$, and let $\G_w\left(r\right)$ be as defined in Equation (\ref{mathcalGw(r)}), for each $w \in V$. Then
\begin{equation*}
\G_w\left(r\right)  \leq  \sum_{z \in V} \sum_{\substack{\bs \in E_{wz}^l \\ \bs \neq \bl_w}}p_{\bs}^qr_\bs^\gamma  \G_z\biggl(\frac{r}{r_{\bs}}\biggr).
\end{equation*}
\end{lem}
\begin{proof}
It is clear from the definition of a subpath in Subsection \ref{nvertex}, that 
\begin{equation}
\label{st}
\begin{split}
\bigl\{ \bs \bt: \bs \in E_{wz}^l, \, &\bt \in E_z^* \textrm{ and } \forall v \in V, \,  \bl_v  \not\subset  \bs \bt \bigr\} \\ 
&\subset \bigl\{ \bs \bt: \bs \in E_{wz}^l, \, \bt \in E_z^*, \, \bs \neq \bl_w \textrm{ and } \forall v \in V, \,  \bl_v  \not\subset  \bt \bigr\},
\end{split}
\end{equation}
and this implies 
\begin{align*}
G_w\left(r\right)  &= \sum_{ \substack{ \bg \in E_w^* \\ \alpha r \, \leq \, r_\bg \, < \, \beta r\\ \forall v \in V : \ \bl_v \, \not\subset  \, \bg } } p_\bg^q &&\\
&= \sum_{ \substack{ \bg \in E_w^* \\  \left|\bg\right| \geq \,l \\ \alpha r \, \leq \, r_\bg \, < \, \beta r \\ \forall v \in V : \ \bl_v \, \not\subset  \, \bg } } p_\bg^q &&(\textrm{by (\ref{g greater than l}), as $r \in (0,\delta)$}) \\
&=  \sum_{ \substack{ \bs \in E_w^l \\ \bt \in E_{t(\bs)}^*  \\  \alpha r \, \leq \, r_{\bs \bt} \, < \, \beta r \\ \forall v \in V : \ \bl_v \, \not\subset  \, \bs \bt } } p_{\bs \bt}^q &&\\
&=  \sum_{z \in V}\sum_{ \substack{ \bs \in E_{wz}^l \\ \bt \in E_z^* \\  \alpha r \, \leq \, r_{\bs \bt} \, < \, \beta r \\ \forall v \in V : \ \bl_v \, \not\subset  \, \bs \bt } } p_{\bs \bt}^q &&\\
&\leq  \sum_{z \in V}\sum_{ \substack{ \bs \in E_{wz}^l \\ \bs \neq \bl_w \\ \bt \in E_z^* \\  \alpha r \, \leq \, r_{\bs \bt} \, < \, \beta r \\  \forall v \in V : \ \bl_v \, \not\subset  \, \bt } } p_{\bs \bt}^q  && (\textrm{by (\ref{st})}) \\
&=\quad \sum_{z \in V} \sum_{\substack{\bs \in E_{wz}^l \\ \bs \neq \bl_w}}p_{\bs}^{\,q} \sum_{ \substack{ \bt \in E_z^* \\  \alpha \frac{r}{r_{\bs}} \, \leq \, r_{\bt} \, < \, \beta \frac{r}{r_{\bs}} \\ \forall v \in V : \ \bl_v \, \not\subset  \, \bt } } p_{\bt}^{\,q} &&\\
&=  \sum_{z \in V} \sum_{\substack{\bs \in E_{wz}^l \\ \bs \neq \bl_w}}p_{\bs}^{\,q} \, G_z\left(\frac{r}{r_{\bs}}\right).
\end{align*}
We have used the convention here that $E_{t(\bs)}^*, E_z^*$ include the empty path which is summed over but doesn't contribute to the sum. 

As $\G_w\left(r\right) = r^{\gamma}G_w\left(r\right)$, we obtain,
\begin{align*}
\G_w\left(r\right) \ &\leq \ \sum_{z \in V} \sum_{\substack{\bs \in E_{wz}^l \\ \bs \neq \bl_w}}p_{\bs}^{\,q}r_{\bs}^{\gamma} \left(\frac{r}{r_{\bs}}\right)^{\gamma}\, G_z\left(\frac{r}{r_{\bs}}\right)  \\
& = \ \sum_{z \in V} \sum_{\substack{\bs \in E_{wz}^l \\ \bs \neq \bl_w}}p_{\bs}^{\,q}r_{\bs}^{\gamma} \, \G_z\left(\frac{r}{r_{\bs}}\right).
\qedhere
\end{align*}
\end{proof}

In the next two lemmas we use the following notations for column vectors 
\begin{equation*}
\boldsymbol{\G}\left(r\right) = \left(\G_w\left(r\right)\right)_{w \in V}^\transpose,
\end{equation*} 
and 
\begin{equation*}
\sup_{a\leq r} \boldsymbol{\G}\left(r\right) = \biggl( \ \sup_{a \leq r} \G_w\left(r\right) \ \biggr)_{w \in V}^\transpose.
\end{equation*} 

\begin{lem}
\label{OSC8}

Let $\bb$ be the positive right eigenvector of $\bB(q,\gamma,l)$, with eigenvalue $\rho(\bB(q,\gamma,l))=1$, as given in Equation (\ref{b}). Let $a \in (0, \delta)$ and let $\G_w\left(r\right)$ be as defined in Equation (\ref{mathcalGw(r)}), for each $w \in V$. 

Then
\begin{equation*} 
\sup_{a \leq r }\boldsymbol{\G}\left(r\right)\leq C_a\left(q\right) \bb,
\end{equation*}
for some positive $C_a\left(q\right)$.
\end{lem}
\begin {proof} For each $q \in \mathbb{R}$, $\gamma\left(q\right)$ is uniquely defined as the real number which satisfies $\rho (\bB(q,\gamma,l) ) = 1$ and $\bb$ is the associated positive eigenvector with eigenvalue $1$, as given in Equation (\ref{b}). To prove the lemma it is enough to show, for each $w \in V$, that
\begin{equation*} 
\sup_{a \leq r }\G_w\left(r\right)< +\infty.
\end{equation*}
As the eigenvector $\bb>0$, a positive number $C_a\left(q\right)$ can then be determined. We remind the reader that
\begin{equation*}
\alpha  = \frac{1}{r_{\max}^N d_{\max}} \quad \textrm{and} \quad \beta  = \frac{1}{r_{\min}^{N+1}d_{\min}}.
\end{equation*}
If $\gamma\left(q\right) \geq 0$ then $ \alpha r \leq r_\bg$ implies 
\begin{equation*}
r^\gamma \leq \left(\frac{r_\bg}{\alpha}\right)^\gamma = \left(r_{\max}^N d_{\max}r_\bg\right)^\gamma.
\end{equation*}
If $\gamma\left(q\right) < 0$ then $r_\bg < \beta r$ implies 
\begin{equation*}
r^\gamma < \left(\frac{r_\bg}{\beta}\right)^\gamma = \left(r_{\min}^{N+1}d_{\min}r_\bg\right)^\gamma.
\end{equation*}
So for $r \in [a,\infty)$, and each $w \in V$,
\begin{align*}
0  \ &\leq  \ \G_w\left(r\right)  \ =  \ r^{\gamma}G_w\left(r\right) = \sum_{ \substack{ \bg \in E_w^* \\ \alpha r \, \leq \, r_\bg \, < \, \beta r\\ \forall v \in V : \ \bl_v \, \not\subset  \, \bg } } p_\bg^qr^{\gamma} \\
&\leq  \max \left\{(r_{\max}^N d_{\max})^\gamma, (r_{\min}^{N+1}d_{\min})^\gamma \right\}\sum_{ \substack{ \bg \in E_w^* \\ \alpha a \, \leq \, r_\bg \\ \forall v \in V : \ \bl_v \, \not\subset  \, \bg } } p_\bg^qr_\bg^{\gamma} <  \ +\infty.   
\end{align*} 
The strict inequality holds as there are only a finite number of paths $\bg \in E_w^*$ with $\alpha a  \leq  r_\bg$. 
\end{proof}

\begin{lem}
\label{OSC9}
Let $\bb$ be the positive right eigenvector of $\bB(q,\gamma,l)$, with eigenvalue $\rho(\bB(q,\gamma,l))=1$, as given in Equation (\ref{b}). Let $a \in (0, \delta)$ and let $\G_w\left(r\right)$ be as defined in Equation (\ref{mathcalGw(r)}), for each $w \in V$. 

Then
\begin{equation*} 
\sup_{0 < r }\boldsymbol{\G}\left(r\right)\leq C_a\left(q\right) \bb,
\end{equation*}
for some positive $C_a\left(q\right)$.
\end{lem}
\begin{proof}
As in Lemma \ref{OSC8}, let $a \in(0,\delta)$ be fixed, and let $\Delta=r_{\max}^l<1$, then $\Delta \geq r_{\bs}$ for all paths $\bs \in E^l$, and $a\Delta<a<\delta$. Consider $r \in [a\Delta,\delta)$, with $a\Delta \leq r<\delta$, then
\begin{equation}
\label{x}
a \leq a \frac{\Delta}{r_{\bs}} \leq \sigma <\frac{\delta}{r_{\bs}},
\end{equation} 
where $\sigma=\frac{r}{r_{\bs}}$. It follows that
\begin{align*}
\sup_{a\Delta\leq r < \delta} \G_w\left(r\right) &\leq  \sup_{a\Delta\leq r < \delta}\Biggl\{  \ \sum_{z \in V} \sum_{\substack{\bs \in E_{wz}^l \\ \bs \neq \bl_w}}p_\bs^qr_\bs^\gamma  \G_z\left(\frac{r}{r_{\bs}}\right) \ \Biggr\} &&(\textrm{by Lemma \ref{OSC7}}) \\
&= \sup_{a \frac{\Delta}{r_{\bs}} \leq \sigma < \frac{\delta}{r_{\bs}}}\Biggl\{  \ \sum_{z \in V} \sum_{\substack{\bs \in E_{wz}^l \\ \bs \neq \bl_w}}p_{\bs}^qr_\bs^\gamma  \G_z(\sigma) \ \Biggr\} &&\\
&\leq  \sup_{a  \leq \sigma}\Biggl\{  \ \sum_{z \in V} \sum_{\substack{\bs \in E_{wz}^l \\ \bs \neq \bl_w}}p_{\bs}^qr_\bs^\gamma \G_z(\sigma) \ \Biggr\} && (\textrm{by Inequality (\ref{x})})  \\
&\leq \sum_{z \in V} \sum_{\substack{\bs \in E_{wz}^l \\ \bs \neq \bl_w}}p_\bs^qr_\bs^\gamma  \sup_{a  \leq \sigma }\biggl\{\G_z(\sigma)\biggr\} && \\
&= \sum_{z \in V} \sum_{\substack{\bs \in E_{wz}^l \\ \bs \neq \bl_w}}p_\bs^qr_\bs^\gamma  \sup_{a  \leq r }\biggl\{\G_z(r)\biggr\}.
\end{align*} 
In matrix form, using the notation $\sup \boldsymbol{\G}\left(r\right) = (\sup \G_w\left(r\right) )_{w \in V}^\transpose$, this is
\begin{equation*}
\sup_{a\Delta\leq r < \delta} \boldsymbol{\G}\left(r\right)   \leq   \mathbf{B}\left(q,\gamma,l\right)   \sup_{a \leq r }\boldsymbol{\G}\left(r\right). 
\end{equation*}
By Lemma \ref{OSC8}, we can find a positive number $C_a\left(q\right)$ such that 
\begin{equation} \label{G_vector_bound_1}
\sup_{a \leq r }\boldsymbol{\G}\left(r\right)\leq C_a\left(q\right) \bb,
\end{equation}
and this means that
\begin{equation} \label{G_vector_bound_2}
\sup_{a\Delta\leq r < \delta} \boldsymbol{\G}\left(r\right) \leq \bB(q,\gamma,l) \sup_{a \leq r }\boldsymbol{\G}\left(r\right) \leq \mathbf{B}\left(q,\gamma,l\right) C_a\left(q\right) \bb = C_a\left(q\right) \bb, 
\end{equation}
the last equality holds as $\bb$ is a positive eigenvector of $\mathbf{B}\left(q,\gamma,l\right)$, with eigenvalue $1$.
Since $a\Delta<a<\delta$, inequalities (\ref{G_vector_bound_1}) and (\ref{G_vector_bound_2}) together imply 
\begin{equation} 
\label{G_vector_bound_3} 
\sup_{a\Delta\leq r } \boldsymbol{\G}\left(r\right) \leq \ C_a\left(q\right) \bb. 
\end{equation}

Now let $n\in \mathbb{N}$ and let $P\left(n\right)$ be the following statement
\begin{equation*} 
\sup_{a\Delta^n \leq r } \boldsymbol{\G}\left(r\right) \leq \ C_a\left(q\right) \bb. 
\end{equation*}
We use induction to prove that $P\left(n\right)$ is true for all $n \in \mathbb{N}$. 

\medskip

\emph{Induction base.}

$P(1)$ is true by Inequality (\ref{G_vector_bound_3}). 

\medskip

\emph{Induction hypothesis.}

Let $k \in \mathbb{N}$ and suppose $P\left(k\right)$ is true, that is suppose
\begin{equation} 
\label{IH} 
\sup_{a\Delta^k \leq r } \boldsymbol{\G}\left(r\right) \leq \ C_a\left(q\right) \bb. 
\end{equation}

\medskip

\emph{Induction step.}

Let $r \in [a\Delta^{k+1}, \delta)$ then as in Inequality (\ref{x}) above it follows that 
\begin{equation}
\label{x^k}
a\Delta^k \leq a \frac{ \Delta^{k+1}}{r_{\bs}} \leq \sigma<\frac{\delta}{r_{\bs}},
\end{equation}
where $\sigma=\frac{r}{r_{\bs}}$. We now obtain
\begin{align*}
\sup_{a\Delta^{k+1}\leq r < \delta} \G_w\left(r\right) &\leq  \sup_{a\Delta^{k+1}\leq r < \delta}\Biggl\{  \ \sum_{z \in V} \sum_{\substack{\bs \in E_{wz}^l \\ \bs \neq \bl_w}}p_\bs^qr_\bs^\gamma \G_z\left(\frac{r}{r_{\bs}}\right) \ \Biggr\} &&(\textrm{by Lemma \ref{OSC7}}) \\
&= \sup_{a \frac{\Delta^{k+1}}{r_{\bs}} \leq \sigma < \frac{\delta}{r_{\bs}}}\Biggl\{  \ \sum_{z \in V} \sum_{\substack{\bs \in E_{wz}^l \\ \bs \neq \bl_w}}p_\bs^qr_\bs^\gamma \G_z(\sigma) \ \Biggr\}  && \\
&\leq   \quad  \sup_{a\Delta^k  \leq \sigma }\Biggl\{  \ \sum_{z \in V} \sum_{\substack{\bs \in E_{wz}^l \\ \bs \neq \bl_w}}p_\bs^qr_\bs^\gamma \G_z(\sigma) \ \Biggr\} &&(\textrm{by Inequality (\ref{x^k})}) \\
&\leq \quad  \sum_{z \in V} \sum_{\substack{\bs \in E_{wz}^l \\ \bs \neq \bl_w}}p_\bs^qr_\bs^\gamma \sup_{a\Delta^k  \leq \sigma }\biggl\{\G_z(\sigma)\biggl\} && \\
&= \quad  \sum_{z \in V} \sum_{\substack{\bs \in E_{wz}^l \\ \bs \neq \bl_w}}p_\bs^qr_\bs^\gamma \sup_{a\Delta^k  \leq r }\biggl\{\G_z(r)\biggl\}.
\end{align*} 
In matrix form, with $\sup \boldsymbol{\G}\left(r\right) = (\sup \G_w\left(r\right) )_{w \in V}^\transpose$, this is
\begin{equation*}
\sup_{a\Delta^{k+1}\leq r < \delta} \boldsymbol{\G}\left(r\right)   \leq   \mathbf{B}\left(q,\gamma,l\right)   \sup_{a\Delta^k \leq r }\boldsymbol{\G}\left(r\right). 
\end{equation*}
By the induction hypothesis, Inequality (\ref{IH}),
\begin{equation} 
\label{G_vector_bound_4}
\sup_{a\Delta^{k+1}\leq r < \delta} \boldsymbol{\G}\left(r\right)  \leq  \mathbf{B}\left(q,\gamma,l\right) \sup_{a\Delta^k \leq r }\boldsymbol{\G}\left(r\right)  \leq   \mathbf{B}\left(q,\gamma,l\right) C_a\left(q\right) \bb = C_a\left(q\right) \bb. 
\end{equation}
As $a\Delta^{k+1}<a\Delta^k<\delta$, inequalities (\ref{IH}) and (\ref{G_vector_bound_4}) together imply
\begin{equation*}
\sup_{a\Delta^{k+1} \leq r } \boldsymbol{\G}\left(r\right) \leq C_a\left(q\right) \bb, 
\end{equation*}
which proves that $P\left(k\right)$ implies $P\left(k+1\right)$. This completes the induction step.

By the principal of mathematical induction $P(n)$ is true for all $n \in \mathbb{N}$, that is 
\begin{equation*} 
\sup_{a\Delta^n \leq r } \boldsymbol{\G}\left(r\right) \leq \ C_a\left(q\right) \bb 
\end{equation*}
for all $n \in \mathbb{N}$. Therefore
\begin{equation*} 
\sup_{0 < r } \boldsymbol{\G}\left(r\right) \leq \ C_a\left(q\right) \bb. 
\qedhere
\end{equation*}
\end{proof}

We now have all the results needed for the proof of Theorem 1.2 which is given in the next section.

\section{Proof of Theorem 1.2}\label{five}

\begin{proof}
For $r\in \left(0,\delta\right)$,
\begin{align*}
\sum_{x \in H_{e,f}}\mu_u\left(B(x,r)\right)^q  &\leq  C_2\left(q\right)C_3C_4\left(q\right)  \sum_{w \in V}G_w\left(r\right) && (\textrm{by Lemma \ref{OSC6}}) \\
&= C_2\left(q\right)C_3C_4\left(q\right)  \sum_{w \in V}r^{-\gamma}\G_w\left(r\right) && (\textrm{by (\ref{mathcalGw(r)})})\\
&\leq  C_2\left(q\right)C_3C_4\left(q\right)C_5\left(q\right)r^{-\gamma}, 
\end{align*}
where the last inequality follows by Lemma \ref{OSC9}, putting 
\begin{equation*}
C_5\left(q\right) = n C_a\left(q\right) \max\bigl\{b_v  : v \in V \bigr\}, 
\end{equation*}
where $n$ is the number of vertices in the graph, and the positive eigenvector of the matrix $\bB(q,\gamma,l)$ is $\bb = (b_v)_{v \in V}^\transpose$. This means that for $e,f \in E_u^1,\, e \neq f$,
\begin{align*}
Q_{e,f}^q\left(r\right)&=M_u^q\left(S_e(F_{t(e)})\cap S_f(F_{t(f)})(r), r \right) \\
&=\sup \Bigl\{ \  \sum_{x \in H_{e,f}} \mu_u\left(B(x,r)\right)^q :  H_{e,f}  \ \textrm{  is an } \\
& \quad \quad \quad \quad \quad \quad \quad \quad r\textrm{-separated subset of }  S_e(F_{t(e)})\cap S_f(F_{t(f)})(r) \ \Bigr\}  \\
&\leq  C_2\left(q\right)C_3C_4\left(q\right)C_5\left(q\right)r^{-\gamma}  \\
&=C_{e,f}\left(q\right)r^{-\gamma}. 
\qedhere
\end{align*}
\end{proof}

 \emph{Acknowledgement.} The author was supported by an EPSRC doctoral training award whilst undertaking this work (except for Subsection \ref{Riemann integrable}).

 

\end{document}